\documentclass[11pt]{amsart}
\usepackage{amsthm, amsfonts, amsmath,amscd, amssymb, latexsym, mathrsfs, slashed, verbatim}
\usepackage{amsmath}
\usepackage{amscd}
\usepackage{amsthm, amssymb, amsfonts, latexsym}
\usepackage[all]{xy}
\usepackage{slashed}
\usepackage{hyperref}

\newtheorem*{acknowledgements}{Acknowledgements}
\newtheorem{theorem}{Theorem}
\newtheorem{corollary}[theorem]{Corollary}
\newtheorem{definition}[theorem]{Definition}

\newtheorem{proposition}[theorem]{Proposition}
\newtheorem{remark}[theorem]{Remark}
\newtheorem{lemma}[theorem]{Lemma}

\numberwithin{equation}{section} \numberwithin{theorem}{section}

\def\SL{\mathop{\rm SL}}
\def\Re{\mathop{\rm Re}}

\def\det{\mathop{\rm det}}%
\def\tr{\mathop{\rm Tr}}%
\def\dim{\mathop{\rm dim}}%
\def\inj{\mathop{\rm inj}}
\def\conf{\mathop{\rm Conf }}%
\def\supp{\mathop{\rm supp}}%
\def\dom{\mathop{\rm Dom}}%
\def\diver{\mathop{\rm div}}%
\def\grad{\mathop{\rm grad}}%

\def\Id{\mathop{\rm Id}}
\newcommand{\wt}{\widetilde}%
\newcommand\cR{\mathcal{R}}
\newcommand\cF{\mathcal{F}}

\def\R{\mathop{\mathbb R}}
\def\C{\mathop{\mathbb C}}

\def\Z{\mathop{\mathbb Z}}
\def\H{\mathop{\mathbb H}}

\newcommand\paperintro%
        {%
         }
\newcommand\paperbody%
        {%
         }

\begin{document}

\title[Heat trace and relative determinants]
{Asymptotics of relative heat traces and determinants on open surfaces of finite area}
\author{Clara L. Aldana}
\address{Albert Einstein Institute, MPG}
\email{clara.aldana@aei.mpg.de}
\thanks{This article is register at the MPG, AEI-2012-200}

\begin{abstract}
The goal of this article is to prove that on surfaces with asymptotically cusp ends the relative determinant
of pairs of Laplace operators is well defined.
We consider a surface with cusps $(M,g)$ and a metric $h$ on the surface that is a
conformal transformation of the initial metric $g$.
We prove the existence of the relative determinant of the pair
$(\Delta_{h},\Delta_{g})$ under suitable conditions on the conformal factor.
The core of the paper is the proof of the existence of an asymptotic expansion of the relative heat trace for small times.
We find the decay of the conformal factor at infinity for which
this asymptotic expansion exists and the relative determinant is defined.

Following the paper by B. Osgood, R. Phillips and P. Sarnak about extremal of determinants
on compact surfaces, we prove Polyakov's formula for the
relative determinant and discuss the extremal problem inside a conformal class.
We discuss necessary conditions for the existence of a maximizer.
\end{abstract}

\maketitle

\paperintro
\section*{Introduction}
In this paper we study the relative determinant of Laplace
operators on surfaces with asymptotically cusp ends and the asymptotic expansion of the
corresponding relative heat traces for small values of time.
A surface with asymptotically cusp ends is
defined in Section \ref{subsection:conftransf}.

Regularized determinants of
elliptic operators play an important role in many fields of
mathematics and mathematical physics. They were initially introduced
by D.B. Ray and I.M. Singer in \cite{RaySinger} in relation to
$R$-torsion. The regularized determinant of the Laplace operator on a
compact Riemannian manifold is defined via a zeta function regularization process. It is an important spectral invariant. For instance, in the
$2$-dimensional case, B. Osgood, R. Phillips and P. Sarnak (OPS)\footnote{From now on we abbreviate B. Osgood, R. Phillips and P. Sarnak
as OPS.} showed in \cite{OPS1} that the determinant,
considered as a functional on the space of metrics, has very
interesting extremal properties. They proved the following result:
Let $M$ be a closed surface, then in a given conformal
class, among all metrics of unit area, there exists a unique metric
of constant curvature at which the regularized determinant attains a
maximum. They also proved a corresponding statement for compact
surfaces with boundary and suitable conditions at the boundary.

Relative determinants were introduced in a general setting by W. M\"uller in \cite{Mu3} as a way to
genera\-li\-ze regularized determinants in the compact case. Previously, a relative determinant for admissible
surfaces was introduced by R. Lundelius in \cite{Lundelius}, and for Dirac operators in $\R^n$ by V. Bruneau in
\cite{Bruneau}. A good example of a non-compact space is a surface with cusps.
A surface with cusps is a $2$-dimensional complete Riemannian manifold $(M,g)$ of finite
area such that outside a compact set the metric is hyperbolic. The
hyperbolic ends are called cusps.
The Laplace operator $\Delta_{g}$ associated to the metric $g$ on $M$ has
continuous spectrum. Therefore its zeta
regularized determinant can not be defined in the same way as in the compact case.
Here is when relative determinants enter into the play.
The relative determinant is defined for a pair of
non-negative self-adjoint operators $(A,B)$ in a Hilbert space provided they
satisfy certain conditions. It is defined through a zeta function using the trace of
the relative heat semigroup $\tr(e^{-tA}-e^{-tB})$, $t>0$.

For surfaces with cusps in \cite{Mu3} W. M\"uller proved that the relative
determinant of the Laplacian is well defined when the Laplacian is
compared with a model operator defined on the cusps.
In this paper we extend this result to surfaces with asymptotically cusp ends. We
also prove Polyakov's formula for metrics for which the relative determinant of the corresponding Laplacians is
defined. The analysis of the extremal of the determinant in this
case is performed in the same way as in OPS, \cite{OPS1}. Unfortunately, the maximizer (the metric of constant
curvature) is not always among the class of metrics for which we can define the relative determinant.

The paper is organized as follows:

We start by fixing a surface with cusps and a class of metrics on $M$ that are conformal to
$g$ and that satisfy sui\-ta\-ble conditions. Let
$h=e^{2\varphi}g$ be a metric in the conformal class of $g$; if the cusps are \lq \lq kept\rq\rq \ but
the metric $h$ is not hyperbolic on them, then we say that $(M,h)$ is a surface with asymptotically cusp ends.
Associated to the metric $h$, there is a Laplacian
which we denote by $\Delta_{h}$. We will consider the relative
determinant of pairs of the form $(\Delta_{h},\Delta_{g})$ and
$(\Delta_{h},\bar{\Delta}_{1,0})$, where
$\bar{\Delta}_{1,0}$ is a model operator over $M$ that is associated
to the cusps.

In Section \ref{section:defandnot} we introduce all the notation and background theory
that we need throughout the
paper. In Section \ref{section:trace_class1} we prove the trace class property of the relative
heat operator for all positive values of $t$, when the conformal factor $\varphi$ as well as its
derivatives up to second order decay as $O(y^{-\alpha})$,
$\alpha>0$ as $y$ goes to infinity; here we are using coordinates $(y,x)$ in the cusps
$Z=[1,\infty)\times S^{1}$.

In Section \ref{section:asypexpant0}, we prove the existence of an asymptotic expansion of the
relative heat trace for small values of $t$.
Theorem \ref{theorem:asympexpanrht} gives precise conditions for the existence of such an expansion up to order $\nu\geq 1$.
The expansion exists if
the function $\varphi|_{Z}(y,x)$ and its derivatives up to second order are
$O(y^{-k})$ as $y$ goes to infinity, with $k\geq 5\nu + 8$; although if $\nu \geq 3$,
more derivatives of $\varphi$ should decay at infinity as well. The precise decay of the higher derivatives
is given in the statement of the theorem.

The proof of this result is very technical but uses classical
methods such as parametrices, Duhamel's principle, upper bounds of heat kernels,
universal coverings, very particular inequalities, and the explicit form of the local heat invariants.
The idea of the proof is to write the relative heat trace as an integral
over the manifold, and to split this integral into three areas of integration: the compact part,
a cutoff of the cusps and the end of the cusps. The cutoff is done at a height $a>1$ that is fixed at the beginning.
The conditions on the conformal factor come from assumptions in different parts of the proof.
Along the paper, we will explain each of these assumptions in detail.
The main point is that later in the proof we let $a$ be a function of $t$ and take the limit as $t\to 0$.
Then, the integral over the cutoff will have a complete asymptotic expansion as $t\to 0$ (as $a\to \infty$).
The integral on the end of the cusps is estimated by a term $t^{\nu}$, $\nu>0$.
The estimation is obtained using the trace norm of some auxiliary operators.
The order $k\geq 5\nu + 8$ in the decay condition of the conformal factor comes from this bound.

In Section \ref{section:relativedet}, we use the previous results to define
the relative determinant of the pairs $(\Delta_{h},\Delta_{g})$ and $(\Delta_{h},\Delta_{1,0})$ using
relative zeta functions. In spite of not having an optimal result in Section \ref{section:asypexpant0},
the result is good enough to have a well-defined relative
determinant for a pair of metrics $(h,g)$ satisfying the conditions above.

In Section \ref{section:variational} we study $\det(\Delta_{h},\Delta_{1,0})$ as
a functional on metrics of a given area in a conformal class and look for its extremal values.

We give a proof of a Polyakov's-type
formula for $\det(\Delta_{h},\Delta_{1,0})$.
The proof of this formula follows the same lines as the proof of OPS in the
compact case in \cite{OPS1} and the formula is the same as the one
obtained by R. Lundelius in \cite{Lundelius} for heights of pairs of
admissible surfaces. However, let us point out that our methods are
different from the ones in \cite{Lundelius}. In the same way as in \cite{OPS1} and
in \cite{Lundelius}, we see that if there exists a maximum it is
attained at the metric of constant curvature. The equation relating
the curvature of the metrics $g$ and $h=e^{2\varphi}g$ is
$R_{h}=e^{-2\varphi}(\Delta_{g}\varphi +R_{g})$. The study of the associated
differential equation for $\varphi$, together with the
constant curvature condition in the cusps for $g$ and constant curvature
everywhere for $h$, leads to a precise decay for the function
$\varphi$ at infinity. Unfortunately this decay is not included in
the conditions required to define the relative determinant. Therefore the
metric of constant curvature will not be in the conformal
class under consideration unless we start with a metric of constant curvature.\\

In relation with this problem, there is a recent paper by P. Albin, F. Rochon and the author, \cite{AAR}.
We worked with renormalized integrals to define renormalized determinants of Laplacians on surfaces that
have asymptotically hyperbolic ends, cusps as well as funnels (funnels involve infinite area).

An earlier version of this paper was published in the ArXiv, under the title
\lq\lq Relative determinants of Laplacians on surfaces with asymptotically cusp ends.\rq\rq

\paperbody

\section{Notation and definitions}
\label{section:defandnot}

\subsection{Relative determinants} \label{subsection:reldet}
Let us recall the definition of relative determinants introduced by W. M\"uller in \cite{Mu3}:
The relative determinant is defined for two
self-adjoint, nonnegative linear operators, $H_1$ and $H_0$, in a
separable Hilbert space ${\mathcal H}$ satisfying the following
assumptions:

\begin{enumerate}
\item For each $t>0$, $e^{-t H_1} - e^{-t H_0}$ is a trace class operator.
\item As $t\to 0$, there is an asymptotic expansion for the relative trace of the
form:
$$\tr(e^{-t H_1}-e^{-t H_{0}}) \sim \sum_{j=0}^{\infty}\sum_{k=0}^{k(j)} a_{jk} t^{\alpha_{j}}
\log^{k}t,$$ where $-\infty < \alpha_0 < \alpha_1 < \cdots$ and
$\alpha_k \to \infty$. Moreover, if $\alpha_j=0$ we assume that
$a_{jk}=0$ for $k>0$.
\item $\tr(e^{-t H_1}-e^{-t H_0}) = h + O(e^{-ct})$, as $t\to \infty$ for some constant $c>0$
where $h=\dim \ker H_1 -\dim \ker H_0$.
\end{enumerate}
These properties allow us to define the relative zeta function as:
\begin{equation}\zeta(s;H_1,H_0)= {\frac{1}{\Gamma(s)}} \int_{0}^{\infty}
(\tr(e^{-t H_1}-e^{-t H_{0}})-h) t^{s-1} dt.\label{eq:relazetafunct}\end{equation}
Using the meromorphic continuation of $\zeta(s;H_1,H_0)$ to the complex plane, the relative
determinant is defined as:
$$\det(H_1, H_0) := e^{-\zeta'(0;H_1,H_{0})}.$$

In a more general setting, condition $(3)$ is replaced by an
asymptotic expansion as $t\to\infty$. In that case, in order to define the
relative zeta function, the integral in (\ref{eq:relazetafunct}) has to be split in two parts,
see \cite{Mu3}.

\subsection{Surfaces with cusps} \label{subsection:swc} A surface with cusps (swc)\footnote{
From now on we abbreviate \lq\lq surface with cusps\rq\rq \ as swc.} is a $2$-dimensional Riemannian manifold that is complete,
non-compact, has finite volume and is hyperbolic in the complement
of a compact set. Therefore it admits a decomposition of the form
\begin{equation*}
M=M_0\cup Z_{1}\cup\cdots\cup Z_{m},
\end{equation*}
where $M_0$ is a compact surface with smooth boundary and for each
$i=1,...,m$ we assume that
$$Z_i\cong [a_i,\infty)\times S^{1},\quad g|_{Z_i}=y_{i}^{-2}(dy_{i}^2+dx_{i}^2),\quad a_{i}>0$$
The subsets $Z_{i}$ are called cusps. Sometimes we denote $Z_{i}$ by
$Z_{a_{i}}$ to indicate the \lq \lq starting
point\rq \rq \ $a_{i}$. For simplicity, by {\bf $S^1$ we mean the circle with radius $1/2\pi$ with length $1$}.
Instances of surfaces with cusps are
quotients of the form $\Gamma(N) \backslash {\mathbb H}$, where
${\mathbb H}$ is the upper half plane and $\Gamma(N)\subseteq
\SL_{2}(\Z)$ is a congruence subgroup, i.e. $\Gamma(N)=\{\gamma \in
\SL_{2}(\Z) \vert \gamma \equiv \Id \pmod{N}\}$. These quotients
play an important role in the theory of automorphic forms.

To any surface with cusps $(M,g)$ we can associate a compact surface
$\overline{M}$ such that $(M,g)$ is diffeomorphic to the complement
of $m$ points in $\overline{M}$.
Let $p$ denote the genus of the compact surface $\overline{M}$; then
the pair $(p,m)$ is called the conformal type of $M$.\\

Later we use the following estimate of the Riemannian distance in the
cusp $Z$
$$d_{g_0}(z,z') \geq \vert\log(y/y')\vert,$$
for $z=(y,x)$, $z'=(y',x')$, see for example \cite{Mu}.\\

For any oriented Riemannian manifold $(M,g)$ the Laplace-Beltrami
ope\-ra\-tor on functions is defined as
$\Delta f = -\diver \grad f$. It is equal to $\Delta=d^{*}d$. We
consider positive Laplacians. If $(M,g)$ is
complete, $\Delta$ has a unique closed extension that we denote by
$\Delta_{g}$.

On a cusp $Z$, the Laplacian is given by
$$\Delta_{Z}= -y^{2} \left({\frac{\partial^{2}}{\partial y^{2}}} +
{\frac{\partial^{2}}{\partial x^{2}}}\right).$$

Let us consider the following operators:
\begin{definition}\label{def:Laplmcuspas}
Let $a>0$, let $\Delta_{a,0}$ denote the
self-adjoint extension of the operator
$$-y^{2} {\frac{\partial^{2}}{\partial y^{2}}}: C_{c}^{\infty}((a,\infty))\to L^{2}([a,\infty),y^{-2}dy)$$
with respect to Dirichlet boundary conditions at $y=a$. The
domain of $\Delta_{a,0}$ is then given by $\dom(\Delta_{a,0}) =
H_{0}^{1}([a,\infty)) \cap H^{2}([a,\infty))$, where
$H_{0}^{1}([a,\infty))=\{f\in H^{1}([a,\infty)): f(a)=0\}$.

Let $\bar{\Delta}_{a,0}=\oplus_{j=1}^m
\Delta_{a_{j},0}$ be defined as the
direct sum of the self-adjoint operators
$\Delta_{a_{j},0}$ defined above. The operator $\bar{\Delta}_{a,0}$
acts on a subspace of $\oplus_{j=1}^m
L^2([a_j,\infty),y_j^{-2}dy_j)$.
\end{definition}

Now, let $a>0$, let $Z_{a}$ be endowed with the hyperbolic metric
$g$ and let $\Delta_{Z_{a},D}$ be the
self-adjoint extension of
$$-y^{2}\left( {\frac{\partial^{2}}{\partial
y^{2}}}+{\frac{\partial^{2}}{\partial x^{2}}}\right):
C_{c}^{\infty}((a,\infty)\times S^{1})\to L^{2}(Z_{a},dA_{g})$$
with respect to Dirichlet boundary conditions at
$\{a\}\times S^{1}$. It is known that the operator $\Delta_{Z_{a},D}$ can be
decomposed as follows: Put
\begin{equation}L^{2}_{0}(Z_{a})=\{f\in
L^{2}(Z_{a},dA_{g})\vert \int_{S^{1}} f(y,x) dx = 0 \text{ for a. e.
} y\geq a\}.\label{eq:spavanizfmZa}\end{equation}
The orthogonal complement of
$L^{2}_{0}(Z_{a})$ in $L^{2}(Z_{a},dA_{g})$ consists of functions
that are independent of $x\in S^{1}$.

Then we can decompose $L^{2}(Z_{a},dA_{g})$ as the orthogonal direct
sum
$$L^{2}(Z_{a},dA_{g}) = L^{2}([a,\infty),y^{-2}dy) \oplus L^{2}_{0}(Z_{a}).$$
This decomposition is invariant under $\Delta_{Z_{a},D}$ so in terms
of this decomposition we can write $\Delta_{Z_{a},D} = \Delta_{a,0}
\oplus \Delta_{Z_{a},1}$, where
$\Delta_{Z_{a},1}$ acts on
$L^{2}_{0}(Z_{a})$.

\begin{remark}
The operator $\Delta_{Z_{a},1}$ has compact resolvent; in particular
it has only point spectrum, see Lemma 7.3 in \cite{MuSa}. In
addition, the counting function for $\Delta_{Z_{a},1}$,
$N_{\Delta_{Z_{a},1}}(\lambda) = \# \{\tilde{\lambda}_{j}\leq
\lambda\}$, where $\{\tilde{\lambda}_{j}\}$ are the eigenvalues of
$\Delta_{Z_{a},1}$, satisfies $N_{\Delta_{Z_{a},1}}(\lambda)\sim
{\frac{\lambda}{4 \pi}} A_{g}$. See \cite[Thm.6]{ColinV}. This
implies that the heat operator $e^{-t\Delta_{Z_{a},1}}$ is trace
class.
\end{remark}

\subsection{Spectral theory of surfaces with cusps}
\label{section:stswc} For the spectral theory of manifolds with cusps
we refer the reader to W. M\"uller \cite{Mu}, Y. Colin de Verdi\`{e}re \cite{ColinV}, and the references therein.
The results in \cite{Mu} hold for any dimension. For surfaces in
particular we refer to \cite{Mu2}. Here we recall only the main
facts and definitions that we use in this article.

For a surface with cusps $(M,g)$, the spectrum of the Laplacian
$\sigma(\Delta_{g})$ is the union of the point spectrum $\sigma_{p}$
and the continuous spectrum $\sigma_{c}$. The point spectrum consist
of a sequence of eigenvalues
$$0=\lambda_{0}<\lambda_{1}\leq \lambda_{2} \leq \dots $$
Each eigenvalue has finite multiplicity, and the counting function
$N(\Lambda)= \# \{\lambda_{j}\vert \lambda_{j}\leq \Lambda^{2}\}$
for $\Lambda>0$ satisfies $\limsup N(\Lambda) \Lambda^{-2} \leq
A_{g}(4\pi)^{-1}$, where $A_{g}$ denotes the area of $(M,g)$.
Depending on the metric, the set of eigenvalues may be infinite or
not.

The continuous spectrum $\sigma_{c}$ of $\Delta_{g}$ is the interval
$[{\frac{1}{4}},\infty)$ with multiplicity equal to the number of
cusps of $M$. For a proof of this fact, see for example
\cite[p.206]{Mu}. The spectral decomposition of the absolutely
continuous part of $\Delta_{g}$ is described by the generalized
eigenfunctions $E_{j}(z,s)$, for $j=1,\dots, m$ with $z\in M$, $s\in
\C$. To each cusp we can associate such generalized eigenfunctions,
they are also called Eisenstein functions by
analogy with the Eisenstein series on hyperbolic surfaces. They are
closely related to the wave operators
$W_{\pm}(\Delta_{g},\bar{\Delta}_{a,0})$ and to the scattering
matrix $S(\lambda)$. For details, see \cite[sec.7]{Mu}. The main properties of the Eisenstein functions and the scattering
matrix can be found in \cite[Theorem 7.24]{Mu}.

\subsection{Conformal transformations} \label{subsection:conftransf}
In this section we give few properties of metrics that are conformal to each other.

A conformal transformation of a metric $g$ on $M$ is a metric $h$
defined as $h=\rho g$ where $\rho\in C^{\infty}(M)$ and $\rho>0$.
In this paper we write the function $\rho$ as $\rho = e^{2\varphi}$ with $\varphi \in C^{\infty}(M)$.
{\it We call the function $\varphi$ the conformal factor}.
Depending on the case the conformal factor may have compact support or
not. If the support is not compact we require $\varphi$ as well as some of its derivatives
to decay at infinity. In what follows the metric $h$ will always denote a conformal transformation
of $g$.

Two metrics $g_{1}$, $g_{2}$ are said to be quasi-isometric if there exist
constants $C_{1}, C_{2} >0$ such that
\begin{equation*}
C_{1}  g_{1}(z) \leq g_{2}(z) \leq C_{2}  g_{1}(z), \quad \text{for
all } z\in M,
\end{equation*}
in the sense of positive definite forms.

Quasi-isometric metrics have equivalent geodesic distances. The
associated $L^{2}$-spaces coincide as sets, thought the inner
product is not the same.
\begin{remark}
Let $h=e^{2\varphi}g$. If the function $\varphi$ is
bounded on $M$, the metrics $g$ and $h$
are quasi-isometric and the geodesic distances, $d_{g}$ and $d_{h}$,
are equivalent. If in addition the metric $g$ is complete, so is the
metric $h$.
\end{remark}

Let us first give a handwaving definition of what we mean by a
surface with asymptotically cusps ends. The reason to do that is
that we need flexibility in the conditions on the conformal factors:\\

\textsl{A surface with asymptotically cusp ends (swac)\footnote{
From now on we abbreviate \lq\lq surface with asymptotically cusp ends\rq\rq \ as swac.} is a surface $(M,h)$
where the metric $h$ is a conformal transformation of the metric on
a swc $(M,g)$ such that the conformal factor as well
as some of its derivatives have a suitable decay in the cusps.}\\

Now, let $(M,g)$ be a swc and $h$ be as above. A point
$z=(y,x)$ in a cusp has injectivity radius ${\inj}_{g}(z)\sim {\frac{1}{y}}$. If we
assume that $\Delta_{g}\varphi=O(1)$ as $y\to \infty$,
the surface $(M,h)$ has bounded Gaussian curvature.
Then by \cite[Prop.2.1]{MuSa}, the injectivity radius of both metrics are comparable. Thus
the injectivity radius of a swac also vanishes.\\

Let $A_{g}$ denote the area of $(M,g)$, $dA_{g}$ the volume element,
and $R_{g}(z)$ its Gaussian curvature. Let $A_{h}$, $dA_{h}$ and
$R_{h}$ be the quantities corresponding to $(M,h)$, for any
conformal transformation $h$ of $g$. Let $\Delta_{h}$ be the
Laplacian associated to $h$. Then the following relations hold:
\begin{equation*}
dA_{h} = e^{2\varphi} dA_{g},\quad \Delta_{h} = e^{-2\varphi}
\Delta_{g},\quad R_{h} = e^{-2\varphi} (\Delta_{g}\varphi +
R_{g})\quad
\end{equation*}

The domains of the Laplacians $\Delta_{g}$ and $\Delta_{h}$ lie in
different Hilbert spaces. Thus, sometimes it is necessary to consider
a unitary map between the spaces $L^{2}(M,dA_{g})$ and
$L^{2}(M,dA_{h})$. From the definition of the metrics and the
transformation of the area element the unitary map is given by:
\begin{equation} T:L^{2}(M,dA_{g})\to L^{2}(M,dA_{h}), \ f\mapsto
e^{-\varphi}f.\label{eq:defT}\end{equation}
The Laplacian operators transform in the following way:
\begin{align}T^{-1}\Delta_{h}T f &= e^{-2\varphi}\left(\Delta_{g} f + 2\langle \nabla_{g} f, \nabla_{g}
\varphi\rangle_{g} - (\Delta_{g} \varphi + \vert \nabla_{g} \varphi
\vert_{g}^{2})f\right) \label{eq:transflaplaceop}\\
T \Delta_{g} T^{-1} f &= e^{2\varphi}\left(\Delta_{h}f - 2 \langle
\nabla_{h} \varphi, \nabla_{h} f\rangle_{h} + (\Delta_{h}\varphi -
\vert \nabla_{h}\varphi \vert_{h})f\right)\notag
\end{align}

Note that the operators $T^{-1}\Delta_{h}T$ and $T\Delta_{g}T^{-1}$
are self-adjoint in the corresponding transformed domain.\\

Let us finish this section recalling Gauss-Bonnet theorem on a swc.
The Euler characteristic a surface $M$ with $m$ cusps
is given by $\chi(M)= (2-2p-m)$, where $p$
is the genus of the compact surface $\overline{M}$ defined in
Section \ref{subsection:swc}. A Gauss-Bonnet formula is valid in
this setting:
$$\int_{M} R_{g} dA_{g} = 2\pi \chi(M),$$
where $R_{g}$ denotes the Gaussian curvature of the metric $g$. The
same formula is valid for the metric $h=e^{2\varphi}g$ when
$\varphi$ and $\Delta_{g}\varphi$ suitably decay at infinity,
since
$$\int_{M} R_{h} \ dA_{h} =\int_{M} e^{-2\varphi} (\Delta_{g}\varphi + R_{g}) e^{2\varphi} \ dA_{g} =
\int_{M} R_{g}\ dA_{g}.$$

\subsection{Heat kernels and their estimates}
\label{subsection:heatkernels&estimates}
\subsubsection{Heat kernels}
\label{subsection:heatkernels}
The heat semigroup associated to a closed self-adjoint operator can
be constructed using the spectral theorem. For the existence and
uniqueness of the heat kernel on a complete open manifold with Ricci
curvature bounded from below see J. Dodziuk, \cite{Dodziuk}. For the main
properties of heat kernels see \cite{Dodziuk} and \cite{Chavel}.

Let $(M,g)$ and $h=e^{2\varphi}g$ be as above, and let
$e^{-t\Delta_{h}}$, $e^{-t\Delta_{g}}$, $e^{-t\Delta_{a,0}}$ denote
the heat semigroups associated to the Laplacians $\Delta_{h}$,
$\Delta_{g}$ and $\Delta_{a,0}$, respectively. Since the Laplacians
are positive, the heat equation is $\Delta +
\partial_{t}=0$. Let $K_{h}(z,z',t)$ and $K_{g}(z,z',t)$ denote the
heat kernels corresponding to $\Delta_{h}$ and $\Delta_{g}$
respectively. The heat kernel on a surface
with cusps was constructed by W. M\"uller in \cite{Mu}.

Like the Laplacians, the heat semigroups act on different spaces.
The operator $e^{-t \Delta_{h}}$ may act on $L^{2}(M,dA_{g})$, but
it is not self-adjoint with respect to this inner product. To make
$e^{-t \Delta_{h}}$ and $e^{-t \Delta_{g}}$ act on the same space
and preserve self-adjointness we use the unitary map
$T$ defined by (\ref{eq:defT}).
The transformed operators $T^{-1} e^{-t \Delta_{h}} T$ and $T e^{-t
\Delta_{g}} T^{-1}$ are self-adjoint on the corresponding space.
The integral kernel of the transformed operator $T^{-1}
e^{-t \Delta_{h}} T : L^{2}(M,dA_{g})\to L^{2}(M,dA_{g})$ is given
by $K_{T^{-1} e^{-t \Delta_{h}}
T}(z,z',t)=e^{\varphi(z)}K_{h}(z,z',t)e^{\varphi(z')}$.

\subsubsection{Estimates of the heat kernels}
\label{subsection:estimates}
In this section we recall the bounds of the heat kernels, since we use them
repeatedly. If the manifold is closed, there exists a constant $c>0$
such that for any fixed $0<\tau<\infty$, the heat kernel satisfies
the following bounds
\begin{equation}K(x,y,t)\ll t^{-n/2} e^{-{\frac{c d(x,y)^{2}}{t}}},
\quad \text{for } t\leq \tau.\label{eq:esthkcptm}\end{equation}
If the manifold has a boundary, consider the closed self-adjoint
extension of the Laplacian with respect to Dirichlet boundary conditions.
In this case, let $K\subset M$ be compact and
a $\tau>0$, then there exist positive constants $c, c'$ such that
$$K_{D}(x,y,t)\leq c' t^{-n/2}(e^{-{\frac{c d(x,y)^{2}}{t}}} + e^{-\frac{d(y, \partial M)^{2}}{8t}}),$$
for $(x,y,t)\in K\times M \times (0,\tau]$, see \cite[chapter
VII]{Chavel}.

Now, let $\wt{Z}={\mathbb R}^{+}\times S^{1}$ be the complete cusp. Let us consider the
hyperbolic metric on it, $g_{0} = y^{-2}(dy^{2}+dx^{2})$. Then $(\wt{Z},g_{0})$ is a complete Riemannian manifold and it is called a \emph{horn}.
Let $\Delta_{1}$ be the
unique self-adjoint extension of the Laplacian defined on
$C_{c}^{\infty}({\mathbb R}^{+}\times S^{1})$. The notation $\Delta_{1}$ is
arbitrary. The construction of the heat kernel for $\Delta_{1}$ on ${\mathbb
R}^{+}\times S^{1}$ can be found in \cite{Mu}.
We denote this heat kernel by $K_{1}$.

Let $\tau>0$ be arbitrary, then there exist constants $C, c> 0$ such
that for $0<t<\tau$, $y, y'\geq 1$, and $k, l, m \in {\mathbb N}$ one
has:
\begin{equation} \label{eq:estdK1} \left\vert {\frac{\partial^{k}}{\partial t^{k}}} d_{z}^{l}
d_{z'}^{m} K_{1}(z,z',t) \right\vert \leq
C (y y')^{{\frac{1}{2}}} t^{-1-k-l-m}
        \exp{\left({-{\frac{c d_{g_0}^{2}(z,z')}{t}}}\right)}\end{equation}
where $d_{g_0}$ the hyperbolic distance in the horn, and the constants depend on $\tau$, see \cite[Prop.2.32]{Mu}.

Let $(M,g)$ be a surface with one cusp that we denote by $Z$, $Z=[a,\infty)\times S^1$ for some $a\geq1$.
Let $i(z)$ be the function given by:
\begin{equation} i(z) = \left\{
\begin{array}{ll}
    1, & \text{if $z\in M\setminus Z$;} \\
    y, & \text{if $z\in Z$ and $z=(y,x)$.} \\
\end{array}
\right. \label{eq:defi}\end{equation}
Given $\tau>0$, there exist $C, c > 0$ such that
\begin{equation} \label{eq:estK} \vert K_{g}(z,z',t)\vert  \leq C (i(z)i(z'))^{{\frac{1}{2}}} t^{-1}
        \exp{\left({-{\frac{c d_{g}^{2}(z,z')}{t}}}\right)}\end{equation}
    for $0<t<\tau$, where $d_g$ is the Riemannian distance in $(M,g)$, see \cite[eq.(4.12)]{Mu}.

Let us now go back to the metric $h=e^{2\varphi}g$. Its restriction to $Z$ can be extended
to a metric on the horn $\wt{Z}$ in the
following way: On $\wt{Z}$ we have the hyperbolic metric $g_{0}$, and
$g\vert_{Z}=g_{0}$. We start by extending the function
$\varphi\vert_{Z}$ to a smooth function $\widetilde{\varphi}$ on
$\wt{Z}$ that vanishes in a small neighborhood of zero. Then on
$(0,\infty)\times S^{1}$ we define $h$ as
$h:=e^{2\wt{\varphi}}g_{0}$. It is a complete metric and $h=g_{0}$
close to the boundary $\{0\}\times S^{1}$. In this way we can define the Laplacian on
$(\wt{Z},h)$. Denote its unique self-adjoint extension by
$\Delta_{1,h}$. Clearly $\Delta_{1,h}=e^{-2\wt{\varphi}}\Delta_{1}$.
The heat kernel associated to $\Delta_{1,h}$ is denoted by
$K_{1,h}(z,z',t)$, for $z,z' \in \wt{Z}$ and $t>0$.

The estimates of the heat kernel of the operator $\Delta_{1,h}$ can be derived from
S. Y. Cheng, P. Li and S. T. Yau's paper \cite{ChLY}, Theorems $4$, $6$ and $7$.
However, in the estimates appears the injectivity radius to a power $\alpha$
that depends only on the dimension of the manifold; from the proof in \cite{ChLY} it is not clear how to determine
the value $\alpha = 1$ that we need. In order to pin down the value of $\alpha$ in this particular case we prove
in Appendix \ref{appendix:hkeacm} the following lemma.

\begin{lemma} \label{lemma:hkeacm}
Let $h$ and $g$ be as above and such that $\varphi$ and $\Delta_{g} \varphi$ decay in the cusp.
Then the heat kernel $K_{h}$ satisfies:
\begin{equation}K_{h}(z,z',t) \ll
       (i(z)i(z'))^{{\frac{1}{2}}} t^{-1}
\exp{\left({-{\frac{\tilde{c} \ d_{h}^{2}(z,z')}{t}}}\right)}
\label{eq:estKh}
\end{equation} for $0<t<\tau$, where $\tilde{c}>0$ is a constant.

Let $*$ denote the metric $g$ or $h$, then derivatives of the heat kernel $K_{*}$ satisfy:
\begin{equation}
\vert \nabla K_{*} (z,z',t)\vert \leq c\ (i(z)i(z'))^{1/2}
t^{-3/2}
\exp{\left( -{\frac{\tilde{c}\ d_{*}^{2}(z,z')}{t}}\right)}, \text{ and }
\label{eq:estgrdK*}
\end{equation}
\begin{equation}
\vert \Delta_{*} K_{*} (z,z',t)\vert \leq C\
(i(z)i(z'))^{1/2} t^{-2} \exp{\left( -{\frac{\tilde{c} \ d_{*}^{2}(z,z')}{t}}\right)},\label{eq:estLaplK*}
\end{equation}
\noindent where the constants $c$, $C$ depend on $\tau$, the curvature, and the covariant derivatives of the curvature. Even more,
we can exchange the distances $d_g$ and $d_h$ in the exponentials on the right-hand side by adjusting the constant in the exponential.
In the same way, the heat kernel $K_{1,h}$ and its derivatives satisfy the same estimates as $K_{h}$ above.
\end{lemma}
For a surfaces with hyperbolic cusps, the estimates in the lemma above were established in \cite{Mu}.

\subsubsection{Heat kernels for other operators} \label{subsection:oheatkernels}
In this part we introduce the other heat operators that we
will use throughout this article.

For $a>1$ let $\Delta_{a,0}$ be the operator defined in Definition \ref{def:Laplmcuspas}.
The heat kernel $p_{a}(y,y',t)$ associated to
$\Delta_{a,0}$ can be computed explicitly, see
\cite[sec.14.2]{CarslawJaeger} or \cite[p.258]{Mu}. It is given by
\begin{equation}
p_{a}(y,y',t) = {\frac{e^{-t/4}}{\sqrt{4\pi t}}} \ (yy')^{1/2}
\left\{ e^{-(\log(y/y'))^{2}/4t} -
e^{-(\log(yy')-\log(a^{2}))^{2}/4t}\right\},\label{eq:psuba}\end{equation}
for $y, y'>a$. This is easy to verify by direct computation. Also
note that for $1\leq y\leq a$, $p_{a}(y,y',t)=0$.

The operator $e^{-t\Delta_{a,0}}$ acts on
$L^{2}([a,\infty),y^{-2}dy)$. However, we can regard it as an
operator acting on $L^{2}([1,\infty),y^{-2}dy)$ by con\-si\-de\-ring the corres\-pon\-ding inclusion
and restriction. Similarly, the operator $e^{-t\Delta_{1,0}}$ can be regarded as acting on
$L^{2}([a,\infty),y^{-2}dy)$.

Now, let us assume that $M$ can be decomposed as $M=M_{0}\cup Z$
with $Z=[1,\infty)\times S^{1}$. Then we can make the operator
$e^{-t\Delta_{a,0}}$ act on $L^{2}(M,dA_{g})$ in the following way:
$$e^{-t\Delta_{a,0}}f (z) = \int_{a}^{\infty}\int_{S^{1}} p_{a}(y,y',t)
\left.f\right\vert_{Z_{a}}(y',x')dx' {\frac{dy'}{y'^{2}}} \quad
\text{ for } z=(y,x)\in Z_{a},$$ and zero otherwise. From the
symmetry of $p_{a}(y,y',t)$, is clear that the operator
$e^{-t\Delta_{a,0}}$ acting on $L^{2}(M,dA_{g})$ is symmetric.

Recall the operator $\Delta_{Z,D}$ defined in Section
\ref{subsection:swc}. The kernel of the operator
$e^{-t\Delta_{Z,D}}$ is constructed by a classical method (see
\cite[chapter VII]{Chavel}) and it is given by:
\begin{equation} K_{Z,D} ((y,x),(y',x'),t)= K_{1}((y,x),(y',x'),t) +
p_{1,D}((y,x),(y',x'),t)\label{eq:dechkcdc}\end{equation} where $y,
y' \geq 1$, $x, x'\in S^{1}$, $t>0$, and $p_{1,D}((y,x),(y',x'),t)$
is a function that satisfies:
for every $\tau>0$ there exist constants $C, c >0$ such that:
\begin{equation}
\vert p_{1,D}(z,z',t)\vert \leq C t^{-1} (i(z)i(z'))^{1/2}
e^{-{\frac{c(d_{g}(z,\partial Z) + d_{g}(z',\partial Z))^{2}}{t}}}
\label{eq:esth}
\end{equation}
for all $z, z' \in Z$ and $0<t<\tau$.

Now let $\Delta_{Z,h}$ be the self-adjoint extension of the operator
$$-e^{-2\varphi}y^{2}(\partial_{y}^{2}+\partial_{x}^{2}):
C_{c}^{\infty}(Z)\to L^{2}(Z,dA_{h})$$ obtained after imposing
Dirichlet boundary conditions at $\{1\}\times S^1$. Let $K_{Z,h}$
denote the kernel of the operator $e^{-t\Delta_{Z,h}}$. As in the
case of the heat kernel associated to the operator $\Delta_{Z,D}$, given in (\ref{eq:dechkcdc}), the kernel $K_{Z,h}$ is given by:
\begin{equation}
K_{Z,h}(z,z',t) = K_{1,h}(z,z',t) + p_{h,D}(z,z',t),
\label{eq:dhkdecfcmhc}\end{equation}
for $z,z'\in Z$ and $t>0$ where
the term $p_{h,D}(z,z',t)$ is determined by the boundary condition.
In the same way as above, $p_{h,D}(z,z',t)$
satisfies, up to some constants, the same estimate as the one in equation (\ref{eq:esth}).

\subsubsection{Duhamel's Principle} \label{section:Duhamel's} There are
several ways to state and use Duhamel's principle, see for example
\cite[VII.3]{Chavel}.

Duhamel's principle can be applied in the non-compact setting under
certain assumptions on the decay of the functions. This is the case
of the heat kernels on surfaces with cusps and asymptotically cusp
ends. In terms of the operators, Duhamel's principle can be stated
as:
\begin{equation}  T^{-1}e^{-t\Delta_{h}}T - e^{-t\Delta_{g}} = \int_{0}^{t}
T^{-1}e^{-s\Delta_{h}}T (\Delta_{g}-
T^{-1}\Delta_{h}T)e^{-(t-s)\Delta_{g}} \ ds.
\label{eq:Duhprinchohg1}
\end{equation}

\section{Trace class property of relative heat operators}
\label{section:trace_class1}
In this section we prove Theorem \ref{theorem:trchocdacm}, which
says that the difference of the heat operators corresponding to the metrics $g$ and $h$ is
trace class. As we know, none of the heat operators
$e^{-t\Delta_{h}}$ nor $e^{-t\Delta_{g}}$ is trace class, which is
the reason why we consider their difference. This is the first step
to define the relative determinant of the pair $(\Delta_{h},\Delta_{g})$.\\

In the second part we consider other relative heat traces that are naturally associated to a surface with cusps.

\subsection{Trace class property} \label{subsection:tcp}
Let $(M,g)$, $M_{0}$, $Z$ as well as $\Delta_{g}$, $\Delta_{Z,D}$,
and $\Delta_{1}$ be as in Section \ref{section:defandnot}. For simplicity,
we assume that $M$ has only one cusp so it can be decomposed as $M=M_{0} \cup Z$ with $M_{0}$ compact
and $Z=[1, \infty)\times S^{1}$.

\begin{theorem}
Let $h=e^{2\varphi}g$, and assume that on the cusp $Z$ the functions
$\varphi(y,x)$, $\vert \nabla_{g} \varphi (y,x)\vert$ and
$\Delta_{g} \varphi(y,x)$ are $O(y^{-\alpha})$ with $\alpha>0$, as
$y\to \infty$. Let $T$ be the unitary map defined in equation (\ref{eq:defT}).
Then for any $t>0$ the operator
$$T^{-1}e^{-t\Delta_{h}}T - e^{-t\Delta_{g}}$$ is trace class.\label{theorem:trchocdacm}
\end{theorem}

To prove this statement we follow a procedure similar to that used
by W. M\"uller and G. Salomonsen in \cite{MuSa}. We use Duhamel's
principle which was stated in Section \ref{section:Duhamel's}.

Let $\Vert \cdot \Vert$ denote the operator norm and $\Vert \cdot
\Vert_{1,g}$, ($\Vert \cdot \Vert_{1,h}$, resp.), denote the trace
norm in $L^{2}(M,dA_{g})$, (in $L^{2}(M,dA_{h})$, resp.). From equation (\ref{eq:Duhprinchohg1}), we have:

\begin{multline} \Vert T^{-1}e^{-t\Delta_{h}}T - e^{-t\Delta_{g}}\Vert_{1,g} \\
\leq \int_{0}^{t/2} \Vert
(\Delta_{g}-T^{-1}\Delta_{h}T)e^{-(t-s)\Delta_{g}}\Vert_{1,g} \ ds \\ +
\int_{t/2}^{t} \Vert e^{-s\Delta_{h}}
(T\Delta_{g}T^{-1}-\Delta_{h})\Vert_{1,h} \ ds
\label{eq:b1gntsi1g1h}
\end{multline}

When considering the trace of the operator on the right-hand side of
(\ref{eq:Duhprinchohg1}) as an integral using heat kernels
and their estimates one has to take two aspects into account. One is
related with the time singularity at $t=0$ and the other one is
related with the convergence of the space integral. The idea of
breaking up the integral in equation (\ref{eq:b1gntsi1g1h}) comes
from the need to avoid the time singularities coming from the heat
kernel $K_{h}(z,z',s)$ ($K_{g}(z,z',t-s)$) close to $s=0$ ($t-s=t$)
that do not integrate to something finite in a neighborhood of zero
(of $t$). Equation (\ref{eq:b1gntsi1g1h}) reduces the proof of Theorem \ref{theorem:trchocdacm} to
the following Proposition:
\begin{proposition} Let $0< a < b <\infty$, under the same conditions of Theorem
\ref{theorem:trchocdacm} we have that for each $t\in [a,b]$, the
operators
$$(\Delta_{g}-T^{-1}\Delta_{h}T)e^{-t\Delta_{g}} \quad \text{ and } \quad \ e^{-t\Delta_{h}}
(T\Delta_{g}T^{-1}-\Delta_{h})$$ are trace class and each trace norm
is uniformly bounded on $[a,b]$. \label{lemma:traceclassro}
\end{proposition}

\begin{proof}
The proof follows in several steps. The idea is to decompose each
operator as the product of two Hilbert-Schmidt (HS)\footnote{We use HS to abbreviate Hilbert-Schmidt.} operators whose
norms are uniformly bounded on $t$ at the corresponding interval.
To prove the HS property we use that if $R$ is an integral operator on $M$
with kernel $r$, its HS norm is given by the $L^{2}(M\times M)$-norm of $r$.
Let $\alpha$ be as in the statement of Theorem \ref{theorem:trchocdacm}, i.e. $\alpha$ denotes the decay of the conformal factor $\varphi$.
Let $\beta=\alpha/2$, if $\alpha\in (0,1)$, and $\beta=1/2$ if
$\alpha\geq 1$; so that $0<\beta\leq 1/2$. Let us define an
auxiliary function $\phi$ that we will use repeatedly, such that
$\phi \in C^{\infty}(M)$ satisfies $\phi
> 0$ and
\begin{equation}\phi(y,x)=y^{-\beta}, \ \ (y,x)\in Z.\label{eq:phiaux}\end{equation} Let $M_{\phi}$ and
$M_{\phi}^{-1}$ denote the operators multiplication by $\phi$ and
$\phi^{-1}$, respectively. The motivation to introduce the function $\phi$ is the fact that the heat operator
$e^{-t\Delta_{g}}$ itself is not HS but when multiplied by $\phi$ it becomes HS. The proof is given below.

{\bf Step 1.} To proof the trace class property of
$(\Delta_{g}-T^{-1}\Delta_{h}T)e^{-t\Delta_{g}}$,
we write
$$(\Delta_{g}-T^{-1}\Delta_{h}T)e^{-t\Delta_{g}} =
((\Delta_{g}-T^{-1}\Delta_{h}T)e^{-(t/2)\Delta_{g}}M_{\phi}^{-1})\circ
(M_{\phi} e^{-(t/2)\Delta_{g}}),$$ and prove that for every $t
> 0$, $(\Delta_{g}-T^{-1}\Delta_{h}T)e^{-t\Delta_{g}}M_{\phi}^{-1}$
and $M_{\phi} e^{-t\Delta_{g}}$ are HS operators.

{\bf Step 1.1}.
$(\Delta_{g}-T^{-1}\Delta_{h}T)e^{-t\Delta_{g}}M_{\phi}^{-1}$ is HS.
 Equation (\ref{eq:transflaplaceop}) implies:
\begin{multline*}(\Delta_{g}-T^{-1}\Delta_{h}T)e^{-t\Delta_{g}}M_{\phi}^{-1} =
((1-e^{-2\varphi(z)})\Delta_{g}) e^{-t\Delta_{g}} M_{\phi}^{-1}\\ +
e^{-2\varphi} (-2\langle \nabla_{g} \varphi , \nabla_{g} \ \cdot \
\rangle_{g} + (\Delta_{g} \varphi + \vert \nabla_{g} \varphi
\vert_{g}^{2})) e^{-t\Delta_{g}}M_{\phi}^{-1}.\end{multline*}

Let us start with the term $((1-e^{-2\varphi(z)})\Delta_{g})
e^{-t\Delta_{g}}M_{\phi}^{-1}$; to prove that it is HS, we
just need to prove that the following integral is finite:
$$\int_{M} \int_{M} \vert(1-e^{-2\varphi(z)})\Delta_{g,z} K_{g}(z,z',t) \phi(z')^{-1}\vert^{2} dA_{g}(z) dA_{g}(z').$$
Let us use the decomposition of $M$ as  $M=M_{0}\cup Z$ to split the
integral as:
\begin{multline}\int_{M} \int_{M} \ \cdots\ \ dA_{g}(z) dA_{g}(z') = \int_{M_{0}} \int_{M_{0}} \ \cdots\ \ dA_{g}(z) dA_{g}(z')
\\+ \int_{M_{0}} \int_{Z} \ \cdots\ \ dA_{g}(z) dA_{g}(z')
+ \int_{Z} \int_{M_{0}} \ \cdots\ \ dA_{g}(z) dA_{g}(z')\\ +\int_{Z}
\int_{Z} \ \cdots\ \ dA_{g}(z) dA_{g}(z').
\label{eq:decomposeintegral1}\end{multline}

Now, we use the estimates of the derivatives of heat kernel
$K_{g}(z,z',t)$ given in (\ref{eq:estLaplK*}), the fact that
$1-e^{-2\varphi(z)}$ decays as $y^{-\alpha}$ at infinity, and the
definition of the function $i(z)$ given in (\ref{eq:defi}). To
estimate the resulting integrals we use the equations in Observation \ref{ss:convint}.
For simplicity let us just write $c$ instead of $2c$ for
the constant in the exponential factor of the estimates of the heat
kernels.

For the first term in the sum in equation
(\ref{eq:decomposeintegral1}) which involves $z\in M_{0}$  and
$z'\in M_{0}$ we have:
\begin{multline*}\int_{M_0} \int_{M_0} \vert(1-e^{-2\varphi(z)})\Delta_{g,z} K_{g}(z,z',t) \phi(z')^{-1}\vert^{2} dA_{g}(z) dA_{g}(z')\\
\ll \int_{M_0} \int_{M_0} t^{-4} e^{-\frac{c}{t} d_{g}^{2}(z,z')}\
dA_{g}(z)\ dA_{g}(z')\ll t^{-4}.\end{multline*}

For the second term in the sum in (\ref{eq:decomposeintegral1}) which involves $z'\in M_0$ and $z\in
Z$ we have:
\begin{multline*}
\int_{M_{0}} \int_{Z} \vert(1-e^{-2\varphi(z)})\Delta_{g,z}
K_{g}(z,z',t) \phi(z')^{-1}\vert^{2} dA_{g}(z) dA_{g}(z')\\ \ll
t^{-4} \int_{M_0} \int_{S^1}\int_{1}^{\infty}
{\frac{1}{y^{1+2\alpha}}} \ e^{-\frac{c}{t} d_{g}^{2}((y,x),z')} \
dy \ dx \ dA_{g}(z') \ll t^{-4}.
\end{multline*}

The third term in the sum in equation (\ref{eq:decomposeintegral1})
involves variables $z\in M_0$ and $z'\in Z$. In this case we use
that the Riemannian distance satisfies $d_{g}(z,z')\geq
d_{g}(\partial Z,z') \geq \vert\log(y')\vert$ from which we infer:
\begin{multline*}
\int_{Z} \int_{M_0} \vert(1-e^{-2\varphi(z)})\Delta_{g,z}
K_{g}(z,z',t) \phi(z')^{-1}\vert^{2} dA_{g}(z) dA_{g}(z')\\  \ll
\int_{1}^{\infty} \int_{S^{1}} \int_{M_0} y'^{1+2\beta} t^{-4}
e^{-\frac{c}{t} d_{g}^{2}(z,(y',x'))} \ dA_{g}(z) \ dx'
{\frac{dy'}{y'^{2}}}\\  \ll t^{-4} \int_{1}^{\infty} e^{-\frac{c}{t}
(\log(y'))^{2}}\ dy'=t^{-4} \int_{0}^{\infty}
e^{-\frac{c}{t}u^{2}}e^{u} \ du \ll t^{-7/2} e^{t/c'}.
\end{multline*}

Finally, the last term in the sum in (\ref{eq:decomposeintegral1}) in which the variables $z, z'$ lie in
$Z$ we have:
\begin{multline*}
\int_{Z} \int_{Z} \vert(1-e^{-2\varphi(z)})\Delta_{g,z}
K_{g}(z,z',t) \phi(z')^{-1}\vert^{2} dA_{g}(z) dA_{g}(z')\\ \ll
t^{-4} \int_{1}^{\infty} \int_{1}^{\infty}
y^{-1-2\alpha}y'^{-1+2\beta} e^{-\frac{c}{t} (\log(y/y'))^{2}}\ dy \
dy' \ll t^{-7/2} e^{{\frac{t}{c}}},
\end{multline*}
since $\alpha>\beta$. Thus we obtain:
$$\Vert (1-e^{-2\varphi})\Delta_{g} e^{-t\Delta_{g}}M_{\phi}^{-1}\Vert_{2}^{2} \ll
t^{-4}\left(1 + t^{1/2}e^{t/c}\right).$$

We proceed now with the operators $e^{-2\varphi}\langle \nabla_{g}
\varphi , \nabla_{g} \ \cdot \
\rangle_{g}e^{-t\Delta_{g}}M_{\phi}^{-1}$ and
$e^{-2\varphi}(\Delta_{g} \varphi + \vert \nabla_{g} \varphi
\vert_{g}^{2})) e^{-t\Delta_{g}}M_{\phi}^{-1}$. Their integral
kernels are given by
\begin{eqnarray*}
e^{-2\varphi(z)}\langle \nabla_{g,z} \varphi(z) , \nabla_{g,z}
K_{g}(z,z',t)\rangle_{g} \phi^{-1}(z'), \quad \text{ and }\\
e^{-2\varphi(z)}(\Delta_{g} \varphi(z) + \vert \nabla_{g,z}
\varphi(z) \vert_{g}^{2})K_{g}(z,z',t)\phi^{-1}(z'), \end{eqnarray*}
respectively. For which we have respectively the following
estimates:
\begin{multline*}
\vert e^{-2\varphi(z)}\langle \nabla_{g,z} \varphi(z) , \nabla_{g,z}
K_{g}(z,z',t)\rangle_{g} \phi^{-1}(z')\vert^{2}\\ \ll t^{-3} i(z)
i(z') \vert \nabla_{g} \varphi(z)\vert^{2} e^{-\frac{c}{t}
d_{g}^{2}(z,z')}\phi^{-1}(z')^{2}, \quad \text{and}\end{multline*}
\begin{multline*}
\vert e^{-2\varphi(z)}(\Delta_{g} \varphi(z) + \vert \nabla_{g,z}
\varphi(z) \vert_{g}^{2} K_{g}(z,z',t)\phi^{-1}(z')\vert^{2}\\ \ll
t^{-2} (\vert\Delta_{g} \varphi(z)\vert + \vert \nabla_{g}
\varphi(z) \vert_{g}^{2})^{2} i(z) i(z') e^{-\frac{c}{t}
d_{g}^{2}(z,z')}\phi^{-1}(z')^{2}.
\end{multline*}

We split the integrals on $M\times M$ in the same way as in equation
(\ref{eq:decomposeintegral1}), and the integrals obtained are very
similar to those carried out in the previous part for the operator
$(1-e^{-2\varphi})\Delta_{g} e^{-t\Delta_{g}}$. The main difference
occurs in the power of $t$.

For the operator $e^{-2\varphi}\langle \nabla_{g} \varphi ,
\nabla_{g} \ \cdot \ \rangle_{g}e^{-t\Delta_{g}}M_{\phi}^{-1}$ we
use the estimates in (\ref{eq:estgrdK*}) and the decay of
the function $\vert \varphi\vert$ at infinity.

Now, for the operator $e^{-2\varphi}(\Delta_{g} \varphi + \vert
\nabla_{g} \varphi \vert_{g}^{2})) e^{-t\Delta_{g}}M_{\phi}^{-1}$ we
use the estimate of the heat kernel given in equation
(\ref{eq:estK}) and the decay of the functions involving $\varphi$.
Let us only show the integral on $Z\times Z$. For $z\in Z$ we have
$(\Delta_{g} \varphi(z) + \vert \nabla_{g,z} \varphi(z)
\vert_{g}^{2})^{2} \ll (y^{-\alpha} + y^{-2\alpha})^{2} \ll
y^{-2\alpha}$. Then
\begin{align*}
&\int_{Z}\int_{Z} \vert e^{-2\varphi(z)}(\Delta_{g} \varphi(z) +
\vert \nabla_{g,z} \varphi(z) \vert_{g}^{2})
K_{g}(z,z',t)\phi^{-1}(z')\vert^{2} dA_{g}(z) dA_{g}(z')\\
& \quad \ll t^{-2} \int_{1}^{\infty} \int_{1}^{\infty}
y^{-1-2\alpha}y'^{-1+2\beta} e^{-{\frac{c}{t}}(\log(y/y'))^{2}} dy
dy' \ll t^{-3/2} e^{{\frac{t}{c}}}.
\end{align*}
Thus in the same way as above we obtain:
\begin{align*}\Vert
e^{-2\varphi}\langle \nabla_{g} \varphi , \nabla_{g} \ \cdot \
\rangle_{g}e^{-t\Delta_{g}}M_{\phi}^{-1} \Vert_{2}^{2} \ll
t^{-3}\left(1 + t^{1/2}e^{t/c}\right), \text{and}\\
\Vert e^{-2\varphi} (\Delta_{g} \varphi + \vert \nabla_{g} \varphi
\vert_{g}^{2}) e^{-t\Delta_{g}}M_{\phi}^{-1} \Vert_{2}^{2} \ll
t^{-2}\left(1 + t^{1/2}e^{t/c}\right).
\end{align*}

{\bf Step 1.2.} The operator $M_{\phi} e^{-t\Delta_{g}}$ is HS. To
see this, we have to prove that the following integral is finite:
$$\int_{M} \int_{M} \vert \phi(z) K_{g}(z,z',t) \vert^{2} dA_{g}(z) dA_{g}(z').$$

We decompose the integral as in equation
(\ref{eq:decomposeintegral1}), and proceed in the same way as above,
using in this case the estimates of $K_{g}(z,z',t)$ given in
(\ref{eq:estK}) and the definition of the functions $\phi$ and
$i(z)$. Again, for the sake of simplicity we just write $c$ instead
of $2c$ in the exponential factor of the heat estimates. The
computations are very similar to those in the previous case.

The integrals over $M_{0}\times M_{0}$, $M_0\times Z$, and $Z\times
M_{0}$ do not have any problem. As for the last term, whose variables $z, z'$ lie in $Z$,
we have:
\begin{multline}
\int_{Z} \int_{Z} \vert \phi(z) K_{g}(z,z',t)\vert^{2} dA_{g}(z')
dA_{g}(z)\\ \ll \int_{1}^{\infty} \int_{1}^{\infty} \ y^{1-2\beta} \
y' t^{-2} e^{\frac{-c}{t} (\log(y/y'))^{2}} {\frac{dy'}{y'^{2}}}
{\frac{dy}{y^{2}}}\\ =  t^{-2} \int_{1}^{\infty} \int_{1}^{\infty}
y^{-1-2\beta} y'^{-1} e^{\frac{-c}{t} (\log(y/y'))^{2}}\
dy \ dy'\leq t^{-3/2}e^{c't}. \label{eq:mphkgcaux}
\end{multline}

Therefore
$$\Vert M_{\phi} e^{-t\Delta_{g}} \Vert_{2}^{2}
\ll t^{-2}+ t^{-3/2}e^{t/4c}.$$

In this way we have that
$(\Delta_{g}-T^{-1}\Delta_{h}T)e^{-t\Delta_{g}}$ is a trace class
operator and the trace norm satisfies:
\begin{multline*}
\Vert(\Delta_{g}-T^{-1}\Delta_{h}T)e^{-t\Delta_{g}}\Vert_{1,g}\\
\leq \Vert (\Delta_{g}-T^{-1}\Delta_{h}T) e^{-(t/2)\Delta_{g}}
M_{\phi}^{-1}\Vert_{2} \cdot \Vert M_{\phi}
e^{-(t/2)\Delta_{g}}\Vert_{2}\\ \ll (t^{-2} + t^{-3} + t^{-4})^{1/2}
\left(1 + t^{1/2}e^{t/c}\right)^{1/2} \left(t^{-2} +
t^{-3/2}e^{t/c'}\right)^{1/2};\end{multline*} the last expression is
integrable for $t$ in compact subsets of $(0,\infty)$.

{\bf Step 2.} In this step we prove that the operator
$e^{-t\Delta_{h}} (T\Delta_{g}T^{-1}-\Delta_{h})$ is trace class.
The proof is very similar to the proof for
$(\Delta_{g}-T^{-1}\Delta_{h}T)e^{-t\Delta_{g}}$ since the heat
kernels satisfy the same estimates, and the metrics are
quasi-isometric. Let us write:
$$e^{-t\Delta_{h}}
(T\Delta_{g}T^{-1}-\Delta_{h}) = (e^{-(t/2)\Delta_{h}}M_{\phi})
\circ (M_{\phi}^{-1}
e^{-(t/2)\Delta_{h}}(T\Delta_{g}T^{-1}-\Delta_{h})),$$ where $\phi
\in C^{\infty}(M)$ is as above. Then we have to prove that for every
$t > 0$, the kernels of the operators $e^{-t \Delta_{h}}M_{\phi}$ and
$M_{\phi}^{-1} e^{-t \Delta_{h}}(T\Delta_{g}T^{-1}-\Delta_{h})$ are square integrable.\\

The operator $M_{\phi}^{-1} e^{-t \Delta_{h}}
(T\Delta_{g}T^{-1}-\Delta_{h})$ is HS. First of all let us consider
the kernel of the operator $e^{-t \Delta_{h}}
(T\Delta_{g}T^{-1}-\Delta_{h})$. For $f\in C^{\infty}_{c}(M)$
we have that:
\begin{multline*}
(e^{-t \Delta_{h}} (T\Delta_{g}T^{-1}-\Delta_{h}) f)(z)\\ =
\int_{M} K_{h}(z,z',t) \cdot (T\Delta_{g,z'}T^{-1}-\Delta_{h,z'})f(z') dA_{h}(z')\\
=\int_{M} ((T\Delta_{g,z'}T^{-1} - \Delta_{h,z'}) K_{h}(z,z',t))
\cdot f(z') dA_{h}(z'),
\end{multline*}
since the operators $T\Delta_{g,z'}T^{-1}$ and $\Delta_{h}$ are
symmetric on $L^{2}(M,dA_{h})$. Now, let us use the equation
\begin{multline}
T\Delta_{g}T^{-1} - \Delta_{h}\\ = (e^{2\varphi}-1)\Delta_{h} -
2e^{2\varphi}\langle \nabla_{h} \varphi, \nabla_{h}\ \cdot \
\rangle_{h} + (\Delta_{g} \varphi - \vert \nabla_{g} \varphi
\vert_{g}^{2})\label{eq:difLaphtLapg}
\end{multline}
to write
\begin{multline*} M_{\phi}^{-1} (T\Delta_{g}T^{-1}-\Delta_{h}) e^{-t
\Delta_{h}} = M_{\phi}^{-1} e^{-t\Delta_{h}}
\{(e^{2\varphi}-1)\Delta_{h}\\ - 2e^{2\varphi}\langle \nabla_{h}
\varphi, \nabla_{h}\ \cdot \ \rangle_{h} + (\Delta_{g} \varphi -
\vert \nabla_{g} \varphi \vert_{g}^{2})\}. \end{multline*} It
follows that $M_{\phi}^{-1} e^{-t \Delta_{h}}
(T\Delta_{g}T^{-1}-\Delta_{h})$ is HS if the following functions
\begin{enumerate}\item
$\phi(z)^{-1}(e^{2\varphi}(z')-1)\Delta_{h,z'}K_{h}(z,z',t)$, \item
$\phi(z)^{-1} e^{2\varphi(z')}\langle \nabla_{h,z'} \varphi,
\nabla_{h,z'}K_{h} \rangle_{h}$ and \item $\phi(z)^{-1}(\Delta_{g}
\varphi(z') - \vert \nabla_{g,z'} \varphi \vert_{g}^{2})
K_{h}(z,z',t)$
\end{enumerate}
are in $L^{2}(M\times M, dA_{h}dA_{h})$.

We split again the integral in the same way as in equation
(\ref{eq:decomposeintegral1}) and use the estimates of the heat
kernel $K_{h}(z,z',t)$ and its derivatives given in equations
(\ref{eq:estKh}), (\ref{eq:estgrdK*}) and (\ref{eq:estLaplK*}). We
also use that for any function $f\in L^{1}(M,dA_{h})$ we have:
$$\int_{M} \vert f \vert  dA_{h}\ll \int_{M} \vert f\vert dA_{g}.$$

For the first function listed above, the integrals are almost the
same as the ones corresponding to the operator
$(1-e^{-2\varphi})\Delta_{g} e^{-t\Delta_{g}} M_{\phi}^{-1}$. Then,

$$\int_{M} \int_{M} \vert \phi(z)^{-1} (e^{2\varphi(z')}-1) \Delta_{h,z'} K_{h}(z,z',t) \vert^{2}
dA_{h}(z) dA_{h}(z') \ll t^{-4} + t^{-7/2}e^{t/c}$$ for some
constant $c>0$.

Similarly for the other two functions we get bounds by $t^{-3}(1 +
t^{1/2}e^{t/c})$ and $t^{-2}(1 + t^{1/2}e^{t/c})$, respectively.
Combining these estimates we obtain:
$$\Vert M_{\phi}^{-1} e^{-t \Delta_{h}}  (T\Delta_{g}T^{-1}-\Delta_{h})\Vert_{2}^{2}
\ll (t^{-4} + t^{-3} + t^{-2})(1 + t^{1/2}e^{t/c}).$$

In the same way as in Step 1.2 we can prove that
$e^{-t \Delta_{h}}M_{\phi}$ is HS with HS norm satisfying:
$$\Vert e^{-t \Delta_{h}}M_{\phi} \Vert_{2}^{2} \ll t^{-2}(1+t^{1/2}e^{\frac{t}{c}}).$$

Finally, for the operator $e^{-t\Delta_{h}}
(T\Delta_{g}T^{-1}-\Delta_{h})$ we obtain:
\begin{align*}\Vert e^{-t\Delta_{h}}
(T\Delta_{g}T^{-1}-\Delta_{h})\Vert_{1,h} &\leq \Vert
e^{-(t/2)\Delta_{h}}M_{\phi}\Vert_{2} \cdot \Vert
M_{\phi}^{-1} e^{-(t/2)\Delta_{h}}(T\Delta_{g}T^{-1}-\Delta_{h})\Vert_{2}\\
&\ll t^{-1} (t^{-4} + t^{-3} + t^{-2})^{1/2} \left(1 +
t^{1/2} e^{t/c} \right)\end{align*} This expression is clearly
integrable for $t$ on compact subsets of $(0,\infty)$.

This finishes the proofs of Proposition \ref{lemma:traceclassro} and
Theorem \ref{theorem:trchocdacm}.
\end{proof}

\begin{corollary}
Let $\psi$ satisfy the same conditions as $\varphi$ in the statement of Theorem \ref{theorem:trchocdacm}.
Then, for any $t>0$ the operator $\psi e^{-t\Delta_{h}}$ is trace class.
\label{lemma:tcppsihoh}
\end{corollary}
\begin{proof}
To proof this Lemma we follow the same method as above.
Namely, we use the semigroup property of $e^{-t\Delta_{h}}$ to decompose the operator $\psi
e^{-t\Delta_{h}}$ as
$$\psi e^{-t\Delta_{h}} = \psi e^{-(t/2)\Delta_{h}} M_{\phi^{-1}} M_{\phi} e^{-(t/2)\Delta_{h}},$$
where $\phi$ is the function given by equation (\ref{eq:phiaux}) and $M_{\phi}$ denotes
the multiplication operator by $\phi$. We already proved that the operators $\psi
e^{-t/2\Delta_{h}} M_{\phi^{-1}}$ and $M_{\phi} e^{-t/2\Delta_{h}}$
are HS.
\end{proof}

\subsection{Relative trace for other heat operators} \label{section:2opeinthecusp}
In this section, we consider relative heat traces of some operators naturally associated to the surface with cusps.
\begin{proposition} The operator $e^{-t\Delta_{g}} - e^{-t \Delta_{Z,D}}$ is trace class for all $t >
0$, where $e^{-t \Delta_{Z,D}}$ is considered as acting on
$L^{2}(M,dA_{g})$. \label{prop:trcphomhoc}
\end{proposition}
This is a corollary of Proposition 6.4 in \cite{Mu}. The statement
of that proposition can be rewritten in our notation as follows:

Assume that $M$ can be decomposed as $M=M_{0}\cup Z$ with
$Z=[1,\infty)\times S^{1}$. Let $P_{0}$ be the orthogonal projection
of $L^{2}(M,dA_{g})$ onto $L^{2}([1,\infty),y^{-2}dy)$. Then for
every $t>0$, $e^{-t\Delta_{g}}-e^{-t\Delta_{1,0}}P_{0}$ is a trace
class operator.

To see that Proposition \ref{prop:trcphomhoc} follows from this
statement, recall what we explained in Section
\ref{subsection:swc}: the operator $\Delta_{Z,D}$ can be
decomposed as $\Delta_{Z,D}=\Delta_{1,0}\oplus \Delta_{Z,1}$, where
the heat operator $e^{-t\Delta_{Z,1}}$ is trace class. So we have:
$$\Vert e^{-t\Delta_{g}}-e^{-t \Delta_{Z,D}}\Vert_{1} =
\Vert e^{-t\Delta_{g}}-e^{-t \Delta_{1,0}}\Vert_{1} + \Vert
e^{-t\Delta_{Z,1}}\Vert_{1}$$

Now, let us consider the operator $\Delta_{a,0}$ for $a>1$. To see
that $e^{-t\Delta_{g}}-e^{-t\Delta_{a,0}}$ is trace class, we will
proceed by writing  the difference as
$$e^{-t\Delta_{g}}-e^{-t\Delta_{a,0}} =
e^{-t\Delta_{g}}-e^{-t\Delta_{1,0}} +
e^{-t\Delta_{a,0}}-e^{-t\Delta_{1,0}}.$$ By Proposition \ref{prop:trcphomhoc},
the first difference is trace class, so it suffices to show that
$e^{-t\Delta_{a,0}}-e^{-t\Delta_{1,0}}$ is trace class.

\begin{proposition} For any $a> 1$ and $t>0$ the operator $e^{-t\Delta_{a,0}}-e^{-t\Delta_{1,0}}$ acting on
$L^{2}([1,\infty),y^{-2}dy)$ is trace class and the trace is given
by:
$$\tr(e^{-t\Delta_{a,0}}-e^{-t\Delta_{1,0}}) = -{\frac{1}{\sqrt{4\pi t}}} \ e^{-t/4} \log(a).$$
As an operator on $L^{2}([a,\infty),y^{-2}dy)$ the trace is given
by:
$$\tr(e^{-t\Delta_{a,0}}-e^{-t\Delta_{1,0}}) = -{\frac{e^{-t/4}}{\sqrt{4\pi}}} \text{Erf}\,(\log(a)/\sqrt{t}),$$
where $\text{Erf}\,(s)= \int_{0}^{s} e^{-v^{2}}dv$.
\label{prop:tcphocahoc1}
\end{proposition}

\begin{proof}
Let us just sketch the proof. For the complete proof, see
\cite{Aldana}.

We use the explicit expression of each heat kernel given
by equation (\ref{eq:psuba}) to prove that, for each $t>0$,
$e^{-t\Delta_{a,0}}-e^{-t\Delta_{1,0}}$ is a Hilbert Schmidt
operator. We prove this by direct computation, showing that the
difference of the heat kernels is in $L^{2}([1,\infty)\times
[1,\infty), {\frac{dy'}{y'^{2}}} {\frac{dy}{y^{2}}})$. The
computations are tiresome and involve functions of the form
$\exp{\left(-{\frac{\log(yy'/a^{2})^{2}}{4t}}\right)}$ and
$\exp{\left(-{\frac{\log(y/y')^{2}}{{2t}}}\right)}$ that should be
properly bounded.

The second step is to decompose the difference as the following sum:
\begin{multline*}
e^{-t \Delta_{a,0}}-e^{-t\Delta_{1,0}}=
e^{-t/2\Delta_{a,0}}M_{\phi}\cdot M_{\phi}^{-1}
(e^{-t/2\Delta_{a,0}}-e^{-t/2\Delta_{1,0}})\\ +
(e^{-t/2\Delta_{a,0}}-e^{-t/2\Delta_{1,0}})M_{\phi}^{-1}\cdot
M_{\phi}e^{-t/2\Delta_{1,0}},\end{multline*} where $M_{\phi}$ is
multiplication by the function $\phi$ defined in equation
(\ref{eq:phiaux}) with $\beta=1/2$. We then prove that each term is
Hilbert Schmidt in a similar fashion as we did in Section \ref{subsection:tcp}.

Now, let us compute the trace:
\begin{multline*}
\tr(e^{-t\Delta_{a,0}}-e^{-t\Delta_{1,0}}) = \int_{1}^{\infty}
(p_{a}(y,y,t)-p_{1}(y,y,t)) {\frac{dy}{y^{2}}}\\
={\frac{e^{-t/4}}{\sqrt{4\pi t}}} \int_{a}^{\infty}
(e^{-(\log(y^{2}))^{2}/4t} - e^{-(\log(y^{2})-\log(a^{2}))^{2}/4t})
{\frac{dy}{y}}\\ - {\frac{e^{-t/4}}{\sqrt{4\pi t}}} \int_{1}^{a} (1
-e^{-(\log(y^{2}))^{2}/4t}) {\frac{dy}{y}} =
-{\frac{e^{-t/4}}{\sqrt{4\pi t}}} \log(a).
\end{multline*}
If we consider $e^{-t\Delta_{a,0}}-e^{-t\Delta_{1,0}}$ as an
operator acting on $L^{2}([a,\infty),y^{-2}dy)$ we have that:
\begin{align*}
\tr(e^{-t\Delta_{a,0}}-e^{-t\Delta_{1,0}}) &= \int_{a}^{\infty}
(p_{a}(y,y,t)-p_{1}(y,y,t)) {\frac{dy}{y^{2}}}\\ &=
-{\frac{e^{-t/4}}{\sqrt{4\pi t}}} \int_{1}^{a} e^{-(\log(y))^{2}/t}
{\frac{dy}{y}}.
\end{align*}
\end{proof}

\begin{remark}
The trace of $e^{-t\Delta_{a,0}}-e^{-t\Delta_{1,0}}$ as an operator
on $L^{2}([a,\infty),y^{-2}dy)$ has an asymptotic expansion for
small values of $t$. This follows from Proposition
\ref{prop:tcphocahoc1} and the fact that $\text{Erf}\,(x)$ has an
expansion for $x\gg 1$. Taking into account only the first term we
have that $\text{Erf}\,(x) = {\frac{\sqrt{\pi}}{2}} + O(x^{-1})$, as
$x\to \infty$ from which we infer that:
$$\tr(e^{-t\Delta_{a,0}}-e^{-t\Delta_{1,0}})_{L^{2}([a,\infty),y^{-2}dy)} =
-{\frac{1}{4}} + O(\sqrt{t}) \quad \text{ as } t\to 0.$$
\end{remark}

\begin{remark}
Let us study the case when the manifold $M$ can be decomposed as
$M=M_{0}\cup Z_{a}$ with $a\geq 1$ and we want to compare the
operators $e^{-t\Delta_{g}}$ and $e^{-t \Delta_{1,0}}$. In this case
we could consider the operator $e^{-t \Delta_{1,0}}$ acting on
$L^{2}(M,dA_{g})$ in the way explained in Section
\ref{subsection:oheatkernels}. However it is more convenient and
accurate to consider the extended space:
\begin{multline}L^{2}(M,dA_{g})\oplus L^{2}([1,a],y^{-2}dy)\\ = L^{2}(M_{0},dA_{g})\oplus
L^{2}_{0}(Z_{a})\oplus L^{2}([a,\infty),y^{-2}dy) \oplus
L^{2}([1,a],y^{-2}dy)\notag \\ = L^{2}(M_{0},dA_{g})\oplus
L^{2}_{0}(Z_{a}) \oplus L^{2}([1,\infty),y^{-2}dy)
\end{multline}
where $L^{2}_{0}(Z_{a})$ is the space defined in equation (\ref{eq:spavanizfmZa}). Then the
operators $e^{-t\Delta_{g}}$ and $e^{-t \Delta_{1,0}}$ act on the
extended space by being null where they are not defined. In this way
we have that
\begin{multline}
\tr(e^{-t\Delta_{g}}-e^{-t \Delta_{1,0}})_{L^{2}(M)\oplus
L^{2}([1,a])}\\ = \tr(e^{-t\Delta_{g}}-e^{-t \Delta_{a,0}}
)_{L^{2}(M)} + \tr(e^{-t\Delta_{a,0}}-e^{-t \Delta_{1,0}})_{L^{2}([1,\infty))}
\label{eq:reltrexspLgL1}
\end{multline}
where for the sake of simplicity we dropped the densities in the
notation of the $L^{2}$ spaces. \label{remark:domLgL1}\end{remark}

\section{Asymptotics of relative heat traces for small time}
\label{section:asypexpant0}

In this section we prove the existence of an asymptotic expansion in $t$ of the relative heat trace
$\tr(T^{-1}e^{-t\Delta_{h}}T-e^{-t\Delta_{g}})$ for small time.
More precisely, we prove that for any $\nu \geq 1$, there exists an expansion up to order $\nu$ of the relative heat trace
as $t\to 0$. By an expansion up to order $\nu$ we mean that the remainder term is an $O(t^\nu)$.

We give explicit conditions on the decay of the conformal factor and its derivatives that guarantee the existences of such expansion.

\subsection{Asymptotics for non-compactly supported perturbations}
\label{subsection:expanstrace1} Let $(M,g)$ be a swc. For the sake of simplicity we assume
that $(M,g)$ has only one cusp $Z \cong [1,\infty)\times S^1$ with the
hyperbolic metric on it. We take $g$ as the background metric on
$M$. Let $h = e^{2\varphi}g$. To start with, let us assume that
for $(y,x) \in Z$, the functions $\varphi(y,x)$ and $\Delta_{g}
\varphi(y,x)$ are $O(y^{-1})$ as $y\to \infty$.

Let $n>1$, let us introduce the following notation:
\begin{equation}M_{n} :=
M_{0}\cup([1,n]\times S^{1}), \quad
Z_{n}'=[1,n]\times S^{1}, \quad
Z_{n}=[n,\infty)\times S^{1}.\label{eq:decompmfd}\end{equation}

We start by constructing
the kernel of a parametrix $Q_{h}(z,w,t)$ of the heat operator
associated to $\Delta_{h}$ by patching together suitable heat
kernels over $Z_{3}' = M_{3}\cap Z = [1, 3]\times S^{1}$. Let us
consider the following kernels:
\begin{itemize}
\item $K_{1,h}(z,w,t)$: the heat kernel of $\Delta_{1,h}$ on the
horn $\wt{Z} = \R^{+}\times S^1$, as was defined in Section \ref{subsection:heatkernels&estimates}.
\item $K_{Z,h}(z,w,t)$: the heat kernel for $\Delta_{Z,h}$, as defined in Section
\ref{subsection:oheatkernels}. $K_{Z,h}$ is given by equation
(\ref{eq:dhkdecfcmhc}).
\item For the compact part we consider a closed manifold $W$ containing
$M_{2}$ isometrically. Let $\Delta_{W,h}$ be the Laplacian on $W$
and $K_{W,h}(z,w,t)$ be the kernel of the corresponding heat
operator $e^{-t\Delta_{W}}$.
\end{itemize}

For any two constants $1<b<c$, let $\phi_{(b,c)}$ be a smooth
function on $[1,\infty)\times S^{1}$ that is constant in the second
variable, is non-decreasing in the first variable, and satisfies
$\phi_{(b,c)}(y,x) = 0$ for $y\leq b$, and $\phi_{(b,c)}(y,x) = 1$
for $y\geq c$. Let $\psi_{2} = \phi_{({\frac{5}{4}}, 2)}$ and
$\psi_{1} = 1 - \psi_{2}$; then $\{\psi_{1},\psi_{2}\}$ is a
partition of unity on $[1, 2]\times S^{1}$. Let $\varphi_{2} =
\phi_{(1,{\frac{9}{8}})}$ and $\varphi_{1} = 1 -
\phi_{({\frac{5}{2}}, 3)}$, so that $\varphi_{i}=1$ on the support
of $\psi_{i}$, $i=1,2$. Extend these functions to $M$ in the obvious
way. Note that $\vert\nabla_{h} \varphi_{i}(z)\vert \ll 1$ and
$\vert \Delta_{h} \varphi_{i}(z)\vert \ll 1,$ for $i=1,2$. For this
choice of functions we have that:
\begin{itemize}
\item $\supp \nabla_{h}\varphi_{1}\subseteq [{\frac{5}{2}},3]\times S^1$, and,  $\supp \psi_{1}\subseteq M_{2}$.
\item $\supp \nabla_{h}\varphi_{2}\subseteq [1,{\frac{9}{8}}]\times S^1$, and,
$\supp \psi_{2}\subseteq [{\frac{5}{4}},\infty)\times S^1$.
\end{itemize}

Now, we put:
\begin{equation}
Q_{h}(z,w,t) = \varphi_{1}(z)K_{W,h}(z,w,t)\psi_{1}(w) +
\varphi_{2}(z)K_{1,h}(z,w,t)\psi_{2}(w).
\label{eq:parametrixhchas}\end{equation} From the properties of the
heat kernels, $K_{W,h}$ and $K_{1,h}$, and the construction of the
gluing functions it is easy to see that $Q_{h}(z,w,t)\to
\delta_{w-z}$, as $t\to 0$.
\begin{lemma} There exist constants $C \geq 0$ and $c>0$ such that
$$\left\vert \left( {\frac{\partial}{\partial t}} +\Delta_{h,z}\right)
Q_{h}(z,w,t)\right\vert \leq C e^{-c/t}, \quad \text{ for }\quad
0<t\leq 1.$$
\end{lemma}

\begin{proof}
We use the estimates of the heat kernels given by equations
(\ref{eq:estKh}), (\ref{eq:estgrdK*}) and (\ref{eq:estLaplK*}) as
well as Theorem \ref{theorem:trchocdacm} and the equivalence of the
geodesic distances $d_{g}$ and $d_{h}$. From the definition of $Q_h$ and the properties of the
heat kernels it follows that:
\begin{multline*}
\left\vert\left( {\frac{\partial}{\partial t}} +\Delta_{h,z}\right)
Q_{h}(z,w,t) \right\vert \ll \vert(\langle
\nabla\varphi_{1},\nabla_{z}K_{W,h}\rangle +
(\Delta_{h}\varphi_{1})K_{W,h})\psi_{1}(w)\vert \\ + \vert(\langle
\nabla\varphi_{2},\nabla_{z}K_{1,h}\rangle +
(\Delta_{h}\varphi_{2})K_{1,h})\psi_{2}(w)\vert.
\end{multline*}

Note that $\left\vert\left( {\frac{\partial}{\partial t}}
+\Delta_{h,z}\right) Q_{h}(z,w,t) \right\vert$ has compact support
in $z$. We consider the following terms separately:
\begin{eqnarray*} S_{1}&:=&\vert(\langle
\nabla\varphi_{1},\nabla_{z}K_{W,h}\rangle +
(\Delta_{h}\varphi_{1})K_{W,h})\psi_{1}(w)\vert,\\
S_{2}&:=&\vert(\langle \nabla\varphi_{2},\nabla_{z}K_{1,h}\rangle +
(\Delta_{h}\varphi_{2})K_{1,h})\psi_{2}(w)\vert.\end{eqnarray*}
$S_{1}= 0$ unless $z\in \supp \nabla\varphi_{1}$ and
$w\in \supp \psi_{1}$. In this case $d_{g}(z,w)\geq \log(5/4)$, then
that taking $c_{1}'=c\log(5/4)$ we obtain:
\begin{eqnarray*}
S_{1} &\leq& (\vert\nabla\varphi_{1}(z)\vert \ \vert
\nabla_{z}K_{W,h}(z,w,t)\vert + \vert\Delta_{h}\varphi_{1}(z)\vert \
\vert K_{W,h}(z,w,t)\vert) \chi_{\supp \psi_{1}}(w)\\ &\ll&
t^{-3/2}e^{-c d_{g}^{2}(z,w)/t} + t^{-1}e^{-c d_{g}^{2}(z,w)/t} \ll
e^{-c_{1}'/2t} \text{ for } t\in (0,1].
\end{eqnarray*}
In the same way as above, $S_{2} = 0$ unless $z\in \supp
\nabla\varphi_{2}$ and $w=(v,u)\in \supp \psi_{2} =
[{\frac{5}{4}},\infty)\times S^1$. In this case $d_{g}(z,w)\geq
\log(v/(9/8))\geq \log(10/9)$. Therefore:
\begin{equation*}
S_{2} \ll v^{1/2} e^{-c(\log(8v/9))^2/2t} (t^{-3/2} + t^{-1})
e^{-c_{2}'/2t} \ll e^{-c_{2}'/4t},
\end{equation*}
where $c_{2}'=c\log(10/9)$. This finishes the proof of the lemma.
\end{proof}

\begin{remark} \label{remark:nulhoatpQa} Note that
$$\left. \left( {\frac{\partial}{\partial t}}
+\Delta_{h,z}\right) Q_{h}(z,w,t) \right\vert_{w=z} = 0.$$

In order that the expression above does not vanish we need that
$$d_{g}(z,w)\geq \min \{ \log(5/4), \log(10/9)\}>0.$$
\end{remark}

We now prove that in the expression of asymptotic expansion of the
relative heat trace we can replace the heat kernel $K_{h}$ by the
parametrix $Q_{h}$ defined above.

\begin{lemma}
There exist constants $C\geq 0$ and $c_{3}>0$ such that, for any
$0<t\leq 1$:
$$\int_{M} \vert Q_{h}(z,z,t)-K_{h}(z,z,t)\vert dA_{h}(z) \leq C e^{-\frac{c_{3}}{t}}.$$
\label{lemma:paramhkcm3}
\end{lemma}

\begin{proof} Applying Duhamel's principle to the heat kernel $K_{h}$ and the parametrix $Q_{h}$ we obtain:
\begin{multline*}
Q_{h}(z,z',t) - K_{h}(z,z',t) =\\ \int_{0}^{t} \int_{M} K_{h}(z,w,s) \left({\frac{\partial}{\partial t}}
+\Delta_{h,w}\right) Q_{h}(w,z',t-s) \ dA_{h}(w) \ ds.\end{multline*}
Remark \ref{remark:nulhoatpQa} implies that:
\begin{align*} & \int_{M} \vert Q_{h}(z,z,t)-K_{h}(z,z,t)\vert dA_{h}(z)\\ & \leq \int_{0}^{t}
\int_{M} \int_{M} \vert K_{h}(z,w,s) \left({\frac{\partial}{\partial
t}} +\Delta_{h,w}\right) Q_{h}(w,z,t-s)\vert \ dA_{h}(w) \
dA_{h}(z)\ ds
\\ &= \int_{0}^{t} \left(\int_{M_{2}}\int_{[\frac{5}{2},3]\times S^1} \cdot \ dA_{h}(w) \ dA_{h}(z)
+ \int_{Z_{\frac{5}{4}}}\int_{[1,\frac{9}{8}]\times S^1} \cdot \
dA_{h}(w) \ dA_{h}(z)
 \right)\ ds.\end{align*}

The first integral on the right-hand side is bounded by:
\begin{multline*}
\int_{0}^{t} \int_{M_{2}}\int_{[\frac{5}{2},3]\times S^1} i(z)^{1/2}
s^{-1} e^{-{\frac{c_2}{s}}} e^{-{\frac{c'}{t-s}}}\
dA_{h}(w) \ dA_{h}(z)\ ds\\
\ll \left(\int_{0}^{t}  e^{-{\frac{c_2}{2s}}} e^{-{\frac{c'}{t-s}}}\
ds\right) \left(\int_{\frac{5}{2}}^{3} {\frac{dv}{v^{2}}}\right) \ll
t e^{-{\frac{c_3}{t}}} \ll e^{-{\frac{c_3}{t}}}
\end{multline*}
since $0<t\leq 1$.

For the second integral on the right-hand side above, recall that
$\supp \psi_{2} \subset [5/4,\infty)\times S^1$. Thus:
\begin{multline*}\int_{0}^{t}
\int_{Z_{\frac{5}{4}}} \int_{[1,\frac{9}{8}]\times S^1} \vert
K_{h}(z,w,s) \left({\frac{\partial}{\partial t}}
+\Delta_{h,w}\right) Q_{h}(w,z,t-s)\vert \ dA_{h}(w) \ dA_{h}(z)\
ds\\ \ll \int_{0}^{t} \int_{\frac{5}{4}}^{\infty}
\int_{1}^{\frac{9}{8}} y^{1/2} e^{-{\frac{c_2}{2s}}}
e^{-{\frac{c_1}{t-s}}}\ \frac{dv}{v^2} \ \frac{dy}{y^2} \ ds \leq t
e^{-{\frac{c_3}{t}}} \leq e^{-{\frac{c_3}{t}}}.
\end{multline*}
\end{proof}
Since the function $e^{-2\varphi}$ is bounded, the
derivatives of the gluing functions $\varphi_{1}$ and $\varphi_{2}$
with respect to the metric $g$ satisfy the same bounds as the
derivatives with respect to the metric $h$. Then we can perform the
same construction for the kernel $K_{g}(z,w,t)$ to replace it by
$Q_{g}(z,w,t)$.


The relative heat trace is given by:
$$\tr(T^{-1}e^{-t\Delta_{h}}T - e^{-t\Delta_{g}}) = \int_{M} (K_{h}(z,z,t)e^{2\varphi(z)} - K_{g}(z,z,t))\ dA_{g}(z).$$
Using Lemma \ref{lemma:paramhkcm3}, we obtain:
\begin{multline*} \left\vert \int_{M} (K_{h}(z,z,t)e^{2\varphi(z)} - K_{g}(z,z,t)) dA_{g}(z)
\right. \\ \left.- \int_{M}  (Q_{h}(z,z,t)e^{2\varphi(z)} -
Q_{g}(z,z,t))  dA_{g}(z)\right\vert
 \ll e^{-c_{3}/t}.\end{multline*}
Therefore we have to determine the asymptotic expansion of the integral:
$$\int_{M} Q_{h}(z,z,t)e^{2\varphi(z)} - Q_{g}(z,z,t) dA_{g}(z).$$
The definitions of $Q_h$ and $Q_g$ induce a natural decomposition of the integral into
two regions of integration, the compact part and the cusp. However, when we use the local expansion of the heat kernel in the cusp
we need to integrate the remainder term uniformly. For this purpose we decompose the cusp as in (\ref{eq:decompmfd}): Let $a>1$, then
$$Z = Z_{a}' \cup Z_a.$$
Therefore the integral decomposes as:
$$\int_{M} Q_{h}(z,z,t)e^{2\varphi(z)} - Q_{g}(z,z,t) dA_{g}(z)= I_{0}(t) + I_{1}(t) + I_{2}(t),$$
where
\begin{eqnarray}
I_{0}(t)&=&\int_{M}
\psi_{1}(z)(K_{W,h}(z,z,t)e^{2\varphi(z)}-K_{W,g}(z,z,t)) \
dA_{g}(z),
 \\ I_{1}(t)&=& \int_{Z_{a}'}
\psi_{2}(z)(K_{1,h}(z,z,t)e^{2\varphi(z)}-K_{1,g}(z,z,t)) \
dA_{g}(z),\label{eq:asymaux1}\\
 I_{2}(t)&=& \int_{Z_{a}}
\psi_{2}(z)(K_{1,h}(z,z,t)e^{2\varphi(z)}-K_{1,g}(z,z,t)) \
dA_{g}(z).\label{eq:asymaux2}
\end{eqnarray}
For the moment we consider $a$ fixed, but later we will assign to it a value depending on $t$.

The integral $I_{0}$ has a complete asymptotic expansion in $t$. To see that, note that in the local
expansions of the kernels $K_{W,g}(z,z,t)$ and $K_{W,h}(z,z,t)$ the corresponding remainder terms
are uniformly bounded on compact sets, therefore they can be
integrated.

The other two integrals can be rewritten as traces of the operators:
\begin{eqnarray*}
A(t) &=& M_{\chi_{Z_{a}'}}M_{\psi_{2}} (T^{-1}e^{-t\Delta_{1,h}}T -
e^{-t\Delta_{1,g}}) \ \text{ and }\\
B(t) &=& M_{\chi_{Z_{a}}} M_{\psi_{2}} (T^{-1}e^{-t\Delta_{1,h}}T -
e^{-t\Delta_{1,g}}),\end{eqnarray*}
respectively. Propositions \ref{prop:casmpexptraoA} and \ref{prop:b1nauxoBaccf} below take care of these integrals.

\begin{proposition} Under the conditions of Theorem
\ref{theorem:trchocdacm}, there is a complete asymptotic
expansion as $t\to 0$ of the integral $I_{1}(t)$ in equation
(\ref{eq:asymaux1}). The asymptotic expansion has the following form:
$$ \int_{[1,a]\times S^{1}} \psi_{2}(z)(K_{1,h}(z,z,t)e^{2\varphi(z)}-K_{1,g}(z,z,t)) \
dA_{g}(z) \sim t^{-1} \sum_{j=0}^{\infty}\hat{a}_{j}t^{j}.$$
The coefficients $\hat{a}_j$ depend on the parameter $a$. There is a remainder term that also depends on $a$ as
$O(e^{-\frac{c}{a^{4}t}})$, for a positive constant $c$.
\label{prop:casmpexptraoA}
\end{proposition}
\begin{proof} In order to deal with the integral $I_1(t)$ we first recall what $K_{1,h}$ and $K_{1,g}$ are. Recall that $h$
was extended to the horn $\wt Z$ and that $K_{1,h}(z,w,t)$
denotes the heat kernel for $\Delta_{h}$ on $\wt Z$. The idea of this
proof is to use the local asymptotic expansion of the corresponding heat kernels and find a
uniform bound on the remainder term.

The universal covering of $\wt{Z}$ is $\hat{Z}= {\R}^{+} \times \R$ with projection
$\pi: \hat{Z} \to \wt{Z}$ and group of deck transformations $\Gamma= \Z$.
The metric $h$ on $\wt Z$ induces a metric $\hat{h}$ on $\hat{Z}$,
that has the same curvature properties as $h$. In addition,
$\hat{h}=e^{2\hat{\varphi}}\hat{g}_{0}$, where $\hat{g}_{0}$ is the
lift of $g_{0}$ to $\hat{Z}$ and is precisely the hyperbolic metric
on $\H$, and the function $\hat{\varphi}$ is a lift of $\wt\varphi$
($\wt\varphi$ the extension of $\varphi$ to $\wt{Z}$),
$\hat{\varphi}=\wt{\varphi} \circ \pi$. It follows that $\hat{h}$
and $\hat{g}_{0}$ are quasi-isometric.
Therefore by Proposition 2.1 in \cite{MuSa}, the
injectivity radius of $\hat{h}$ is bounded from below by a positive
constant independent of the point. In this way $(\hat{Z},\hat{h})$
has bounded geometry. Let $k_{h}$ denote the heat kernel of
$\Delta_{\hat{h}}$ in $\hat{Z}$. It satisfies the following estimate:
\begin{equation}k_{h}(\wt z,\wt w,t)\leq C t^{-1} e^{-\frac{c\ d^{2}(\wt z, \wt w)}{t}}, \label{eq:ehklhH2}\end{equation}
where $\wt z, \wt w \in \hat{Z}$ and $0< t \leq 1$, \cite{ChLY}.
It is not difficult to verify that
\begin{equation} K_{1,h}(z,w,t)=\sum_{m\in \Z}
k_{h}(\wt z, \wt w+m,t),\label{eq:hkhitklh}\end{equation}
where $\pi(\wt z)= z$, $\pi(\wt w)= w$.

The construction above can be performed for the kernel $K_{1,g}$ as
well. Then the integral $I_{1}(t)$ becomes:
$$\int_{1}^{a}\int_{0}^{1} \wt{\psi}_{2}(\wt z)
\left(\sum_{m\in \Z} k_{h}(\wt{z}, \wt{z}+m,
t)e^{2\hat{\varphi}(\wt{z}+m)} - \sum_{l\in \Z} k_{g}(\wt{z},
\wt{z}+l , t) \right) \ dA_{\hat{g}}(\wt z),$$
because $F={\R}^{+} \times [0,1]$ is a fundamental domain for $\Gamma$
and the domain corresponding to $Z_{a}'$ in $F$ is $[1,a]\times
[0,1]$; and $\wt{\psi}_{2}$ is the natural extension and lift of $\psi_2$ to $\H$. Thus
\begin{align}I_{1}(t) &= \int_{1}^{a}\int_{0}^{1}
\wt{\psi}_{2}(\wt z) (k_{h}(\wt{z}, \wt{z},
t)e^{2\hat{\varphi}(\wt{z})} - k_{g}(\wt{z}, \wt{z}, t)) \
dA_{\hat{g}}(\wt z)\notag \\ & \  + \int_{1}^{a}\int_{0}^{1}
\wt{\psi}_{2}(\wt z) \sum_{m\neq 0} (k_{h}(\wt{z}, \wt{z}+m,
t)e^{2\hat{\varphi}(\wt{z}+m)} - k_{g}(\wt{z}, \wt{z}+m , t)) \
dA_{\hat{g}}(\wt z). \label{eq:sI1seraux1}\end{align}

We will start by estimating the second term on the right-hand side
of (\ref{eq:sI1seraux1}). Note that $\hat{\varphi}=\wt{\varphi}\circ
\pi$ implies that the function $e^{2\hat{\varphi}}$ is bounded.
This, the fact that the metrics $\hat{h}$ and $\hat{g}$ are
quasi-isometric and the estimate on the heat kernel $k_{h}$ imply
that: \begin{equation}\sum_{m\neq 0} k_{h}(\wt{z}, \wt{z}+m,
t)e^{2\hat{\varphi}(\wt{z}+m)}\ll t^{-1} \sum_{m\neq 0} \exp{\left(
-{\frac{c_{1}d_{\hat g}^{2}(\wt z, \wt z
+m)}{t}}\right)}.\label{eq:esserkhbhdchae}\end{equation}
The explicit expression of the hyperbolic distance in the upper
half plane gives: $$d_{\hat g}((\wt x, \wt y), (\wt x + m,
\wt y)) = \cosh^{-1} \left(1 + {\frac{m^{2}}{2 \wt{y}
^{2}}}\right).$$ If $s\geq 1$, $\cosh^{-1}(s)=\log(s+\sqrt{s^{2}-1})$; this implies:
\begin{eqnarray*}
d_{\hat g}((\wt x, \wt y), (\wt x + m, \wt y))
= \log \left(1 + {\frac{m^{2}}{2\wt{y} ^{2}}}+ {\frac{\vert
m\vert}{\wt{y}}} \sqrt{{\frac{m^{2}}{4\wt{y}^{2}}}+1} \right) \geq
\log \left(1 + {\frac{m^{2}}{2\wt{y} ^{2}}}\right).
\end{eqnarray*}
For $\wt{y} = y \in [1,a]$, $\log (1 + {\frac{m^{2}}{2\wt{y}
^{2}}})\geq \log(1 + {\frac{1}{2a^{2}}})$. Thus
$$e^{-{\frac{c_{1}d_{\hat g}^{2}(\wt z, \wt z +m)}{t}}} \leq e^{ -{\frac {c_{1}
\log(1 + {1}/{2 a^{2}})^{2}} {2t}}} e^{ -{\frac{c_{1} \log(1 +
{{m^{2}}/{2\wt{y} ^{2}}})^{2}}{2t}}}.$$ In addition, $0\leq s \leq 1$ satisfies $\log(1+s)\geq s/2$. Applying this to
$s=(2a^{2})^{-1}$ gives:
\begin{equation}
\sum_{m\neq 0} e^{-{\frac{c_{1}d_{\hat g}^{2}(\wt z, \wt z +m)}{t}}}
\leq e^{ -{\frac {c_{1}} {2^{5} a^{4} t}}} \sum_{m\neq 0} e^{
-{\frac{c_{1} \log(1 + {\frac{m^{2}}{2\wt{y} ^{2}}})^{2}}{2t}}} \leq
e^{ -{\frac {c_{2}}
{a^{4} t}}} \sum_{m\neq 0}e^{-{\frac{c_{1} \log(1 +
{\frac{m^{2}}{2a^{2}}})^{2}}{2t}}}, \label{eq:bshkmnzwra1}
\end{equation} with $c_{2}$ a positive constant. In order to estimate the series, we compare it with an
integral using the fact that $\exp{\left(-{\frac{c_{1} \log(1 +
{\frac{m^{2}}{2a^{2}}})^{2}}{2t}}\right)}$ is a decreasing function of $m$.
We proceed in the following way:
\begin{multline}
\sum_{m\neq 0}e^{-{\frac{c_{1} \log(1 +
{\frac{m^{2}}{2a^{2}}})^{2}}{2t}}} \ll \int_{1}^{\infty} e^{
-{\frac{c_{1} \log(1 + {\frac{u^{2}}{2a^{2}}})^{2}}{2t}}} du \\ \leq
\int_{1}^{\sqrt{2}a} e^{ -{\frac{c_{1} \log(1 +
{\frac{u^{2}}{2a^{2}}})^{2}}{2t}}} du + \int_{\sqrt{2}a}^{\infty}
e^{ -{\frac{2c_{1} \log({\frac{u}{\sqrt{2}a}})^{2}}{t}}} du\\ \ll
(\sqrt{2}a-1) + a\int_{0}^{\infty} e^{-{\frac{2c_{1} v^{2}}{t}}}
e^{v}  dv \ll a (1+\sqrt{t}e^{ct}) \ll a,
\label{eq:bbaauxintfs3}\end{multline} where for the integral on the right-hand side,
we used the change of variables $v=\log({\frac{u}{\sqrt{2}a}})$; and in
the middle step we used that for $x\geq 1$, $(\log(x^2+1))^{2}\geq
(\log(x))^{2}$.
Now we can use (\ref{eq:esserkhbhdchae}) and the bounds above to
estimate the second term on the right-hand side of equation
(\ref{eq:sI1seraux1}):
\begin{multline}
\int_{1}^{a}\int_{0}^{1} \vert \wt{\psi}_{2}(\wt z) \sum_{m\neq
0} (k_{h}(\wt{z}, \wt{z}+m, t)e^{2\hat{\varphi}(\wt{z}+m)} -
k_{g}(\wt{z}, \wt{z}+m , t))\vert \
dA_{\hat{g}}(\wt z)\\
\ll t^{-1} \int_{1}^{a}\int_{0}^{1} \vert \wt{\psi}_{2}(\wt z)
\sum_{m\neq 0} e^{-{\frac{c_{1}d_{\hat g}^{2}(\wt z, \wt z
+m)}{t}}}\vert dA_{\hat{g}}(\wt z)\\ \ll t^{-1}
e^{-{\frac{c_{2}}{a^{4} t}}} \int_{1}^{a} \sum_{m\neq 0} e^{
-{\frac{c_{1} \log(1 +
{\frac{m^{2}}{2a^{2}}})^{2}}{2t}}}{\frac{dy}{y^{2}}} \ll t^{-1} a
e^{-{\frac{c_{2}}{a^{4} t}}}. \label{eq:auxbisa}
\end{multline}
Let us remark that in equation (\ref{eq:auxbisa}), the right-hand side is a $O(e^{-c/a^{4}t})$ as $t\to 0$ with $c>0$.

Now, let us denote the first term on the right-hand side of equation (\ref{eq:sI1seraux1}) by $\wt{I}_1(t)$.
The heat kernels $k_{h}(\wt{z}, \wt{z}, t)$ and $k_{g}(\wt{z},
\wt{z}, t)$ have a uniform local asymptotic expansion as $t\to 0$ of
the usual form:
\begin{equation}
k_{*}(\wt{z},\wt{z},t) = t^{-1} \sum_{k=0}^{N}
a_{k}(\hat{*},\wt{z})t^{k} + {\mathcal R}_{N}(\hat{*},\wt{z},t), \text{ for any } N\geq 0,
\label{eq:aelhkgh}
\end{equation}
where $* = g, h$. For the remainder terms there is a constant $C>0$
such that
\begin{equation}
\vert {\mathcal R}_{N}(\hat{h},\wt{z},t)\vert \leq C t^{N} \quad \text{and}
\quad \vert {\mathcal R}_{N}(\hat{g},\wt{z},t)\vert \leq C t^{N}
\label{eq:unifbrtlaehklucc}
\end{equation}
independent of $\wt{z}$. Replacing the corresponding expansion in $\wt{I}_1(t)$
we obtain:
\begin{multline} \wt{I}_1(t)=\int_{1}^{a}\int_{0}^{1} \wt{\psi}_{2}(\wt z)
 t^{-1}\left(\sum_{k=0}^{N} a_{k}(\hat h,\wt
z)e^{2\hat{\varphi}(\wt{z})}- a_{k}(\hat g,\wt z)\right)t^{k}  dA_{\hat{g}}(\wt z)\\
  + \int_{1}^{a}\int_{0}^{1} ({\mathcal R}_{N}(\hat
h,\wt z,t)e^{2\hat{\varphi}(\wt{z})}
-{\mathcal R}_{N}(\hat g,\wt z,t))  dA_{\hat{g}}(\wt z). \label{eq:aelhkiv}
\end{multline}
Note that each integral converges separately since the integrands are bounded and the domain has finite area. So, strictly
speaking we do not need to consider relative objects in this part. However, when we take $a=t^{-1/5}$ and we take the limit as
$t\to 0$, the need of considering the relative integral becomes clear.

We estimate the integrals of the remainder terms using equation
(\ref{eq:unifbrtlaehklucc}):
\begin{multline} \left\vert \int_{1}^{a}\int_{0}^{1} \wt{\psi}_{2}(\wt z)
({\mathcal R}_{N}(\hat h,\wt z,t)e^{2\hat{\varphi}(\wt{z})}-{\mathcal R}_{N}(\hat g,\wt z,t)) dA_{\hat{g}}(\wt z)\right\vert\\
\leq  \int_{1}^{a}\int_{0}^{1} (\vert {\mathcal R}_{N}(\hat h,\wt
z,t)e^{2\hat{\varphi}(\wt{z})}\vert + \vert {\mathcal R}_{N}(\hat g,\wt
z,t)\vert) dA_{\hat{g}}(\wt z) \ll  t^{N} \int_{1}^{\infty}
{\frac{dy}{y^{2}}} \ll t^{N},\label{eq:estRt}
\end{multline}
for $0< t \leq 1$. Note that this estimation is independent of $a$. This finishes the proof of Proposition
\ref{prop:casmpexptraoA}.
\end{proof}

\begin{proposition} Let $\varphi\vert_{Z}(z)$, $\Delta_{g}\varphi\vert_{Z}(z)$,
and $\vert \nabla_{g}\varphi\vert_{g} \vert_{Z}(z)$ with $z=(y,x)$,
be $O(y^{-k})$ as $y\to \infty$, with $k\geq 1$. Then For
$0<t\leq 1$, we have:
\begin{equation}\vert I_{2}(t)\vert = \vert \tr(M_{\chi_{Z_{a}}} M_{\psi_{2}} (T^{-1}e^{-t\Delta_{1,h}}T -
e^{-t\Delta_{1,g}})) \vert \ll a^{-k+1/2} t^{-3/2}.\label{eq:bI2itakt}\end{equation}
\label{prop:b1nauxoBaccf}
\end{proposition}
\begin{proof}
To prove Proposition \ref{prop:b1nauxoBaccf} we want to apply
Duhamel's principle on the cusp $Z$. However the heat operators
involved in the trace correspond to Laplacians in the horn
$\wt{Z}$. Therefore in order to make the computations easier, we
first replace them by the heat operators $e^{-t\Delta_{Z,h}}$ and
$e^{-t\Delta_{Z,g}}$ corresponding to the extensions of the Laplacians on the cusps
with respect to Dirichlet boundary conditions.
Then, we apply Duhamel's principle to $e^{-t\Delta_{Z,h}}$ and
$e^{-t\Delta_{Z,g}}$. We have to take into account more terms, but
we avoid the problem of the singularity at $y=0$. Using equations
(\ref{eq:dechkcdc}) and (\ref{eq:dhkdecfcmhc}) to replace the
respective kernels we obtain:
\begin{multline*} \tr( M_{\chi_{Z_{a}}} M_{\psi_{2}} (T^{-1}e^{-t\Delta_{1,h}}T -
e^{-t\Delta_{1,g}})) = \tr(M_{\chi_{Z_{a}}} M_{\psi_{2}}
(T^{-1}e^{-t\Delta_{Z,h}}T - e^{-t\Delta_{Z,g}}))\\ - \int_{M}
\chi_{Z_{a}}(z) \psi_{2}(z) (p_{h,D}(z,z,t) e^{2\varphi(z)}
-p_{1,D}(z,z,t)) dA_{g}(z).\end{multline*} From equation
(\ref{eq:esth}) and $\supp(\psi_{2})=Z_{5/4}$ it follows that:

\begin{multline*}\left\vert \int_{M} \psi_{2}(z) (p_{h,D}(z,z,t)
e^{2\varphi(z)} -p_{1,D}(z,z,t)) dA_{g}(z)\right\vert\\ \ \ll
\int_{Z_{\frac{5}{4}}} t^{-1} y (e^{-{\frac{c d_{h}(z,\partial
Z)}{t}}} + e^{-{\frac{c' d_{g}(z,\partial Z)}{t}}}) dA_{g}(z)
\ll \int_{\frac{5}{4}}^{\infty} t^{-1} y
e^{-{\frac{c_{1}\log(y)^{2}}{t}}} {\frac{dy}{y^{2}}}\\ \leq t^{-1}
e^{-{\frac{c_{1}\log(5/4)^{2}}{2t}}} \int_{\frac{5}{4}}^{\infty}
y^{-1} e^{-{\frac{c_{1}\log(y)^{2}}{2t}}} dy \ll
e^{-{\frac{c_{1}\log(5/4)^{2}}{4t}}}.
\end{multline*}

Let us now continue with the estimation of the trace of the operator:
$$M_{\chi_{Z_{a}}} M_{\psi_{2}} (T^{-1}e^{-t\Delta_{Z,h}}T -
e^{-t\Delta_{Z,g}}).$$ The kernel of $T^{-1}e^{-t\Delta_{Z,h}}T -
e^{-t\Delta_{Z,g}}$ is given by
$$e^{\varphi(z)} K_{Z,h}(z,w,t)e^{\varphi(w)} - K_{Z,g}(z,w,t),$$ and for $z=w$ it takes the form $K_{Z,h}(z,z,t)e^{2\varphi(z)}-K_{Z,g}(z,z,t)$. From the usual form
of Duhamel's principle we infer:
\begin{multline*}
K_{Z,h}(z,w,t)e^{2\varphi(w)} - K_{Z,g}(z,w,t) = \\ \int_{0}^{t}
\int_{M} K_{Z,h}(z,z',s) e^{2\varphi(z')}
(\Delta_{Z,g}-\Delta_{Z,h}) K_{Z,g}(z',w,t-s) dA_{g}(z')\ ds.
\end{multline*}
Then taking $z=w$ in the equation above and using the transformation of the
Laplacian we obtain:
\begin{multline*}
\tr(M_{\chi_{Z_{a}}} M_{\psi_{2}} (T^{-1}e^{-t\Delta_{Z,h}}T -
e^{-t\Delta_{Z,g}}))\\ = \int_{Z_{a}}\psi_{2}(z) \int_{0}^{t} \int_{Z} \left\{K_{Z,h}(z,z',s)
e^{2\varphi(z')}(1-e^{-2\varphi(z')})\right.  \\ \left. \Delta_{Z,g}
K_{Z,g}(z',z,t-s)\right\} dA_{g}(z')\ ds\ dA_{g}(z).
\end{multline*}

Recall that $\supp(\psi_{2})=Z_{5/4}$, let us first assume that
$a > 5/4$, so $4a/5 >1$. Split the integral as the sum of the
following terms:
\begin{enumerate}
\item $J_{1}=\int_{0}^{t}\int_{Z_{a}}\int_{[1,\frac{4a}{5}]\times S^{1}}\ \cdot \
dA_{g}(z')dA_{g}(z)ds$. \item
$J_{2}=\int_{0}^{t/2}\int_{Z_{a}}\int_{Z_{\frac{4a}{5}}}\ \cdot \
dA_{g}(z')dA_{g}(z)ds$.
\item $J_{3}=\int_{t/2}^{t}\int_{Z_{a}}\int_{Z_{\frac{4a}{5}}}\ \cdot \
dA_{g}(z')dA_{g}(z)ds$.
\end{enumerate}

In this part, we only describe the main lines of the proof. The proof of the estimation of each integral
is given in the Appendix. The methods are very similar
to the ones used to prove Theorem \ref{theorem:trchocdacm}.

Let $k\geq 1$ and suppose that $\varphi(y,x)=O(y^{-k})$ as $y\to
\infty$. Then so are $\psi=1-e^{-2\varphi}$ and $\wt{\psi}=e^{2\varphi}-1$.
Thus for $J_{1}$ we have:
\begin{multline}
J_{1} = \int_{0}^{t}\int_{Z_{a}}\int_{[1,\frac{4a}{5}]\times S^{1}}
\psi_{2}(z) (K_{1,h}(z,z',s) +
p_{h,D}(z,z',s))e^{2\varphi(z')}\\
\psi(z') \Delta_{Z,g} (K_{1,g}(z',z,t-s) + p_{1,D}(z',z,t-s)) \
dA_{g}(z')\ dA_{g}(z)\ ds. \label{eq:J1comp} \end{multline}

On this region $a\leq y<\infty$ and $1\leq y' \leq
{\frac{4a}{5}}$. Thus $1<{\frac{5}{4}}\leq {\frac{y}{y'}}$, so
$\log(y/y')$ is bounded away from $0$. Using the estimates of the heat kernels we obtain:
$$\vert J_{1}\vert \ll a e^{-{\frac{c'}{t}}},$$
for some constants $c' >0$.

For $J_{2}$, let us use that the variable $z'\in Z_{\frac{4a}{5}}$
to multiply the inside integral by the characteristic function
$\chi_{Z_{\frac{4a}{5}}}(z')$. Then,
\begin{multline*}
J_{2} = \int_{0}^{t/2}\int_{Z_{a}}\int_{Z_{\frac{4a}{5}}}
\psi_{2}(z) K_{Z,h}(z,z',s) e^{2\varphi(z')}\\
\chi_{Z_{\frac{4a}{5}}}(z')\psi(z') \Delta_{Z,g} K_{Z,g}(z',z,t-s)
dA_{g}(z')dA_{g}(z)ds.
\end{multline*}
Writing this integral in terms of traces of the corresponding
operators we infer:
\begin{multline*}
\vert J_{2}\vert  = \left\vert \int_{0}^{t/2} \tr( M_{\psi_{2}}
e^{-s\Delta_{Z,h}} M_{e^{2\varphi}}  M_{\chi_{Z_{\frac{4a}{5}}}}
M_{\psi} \Delta_{Z,g}e^{-(t-s)\Delta_{Z,g}})ds\right\vert\\
\ll \int_{0}^{t/2} \Vert M_{\chi_{Z_{\frac{4a}{5}}}} M_{\psi}
\Delta_{Z,g}e^{-(t-s)\Delta_{Z,g}}\Vert_{1} ds = \int_{t/2}^{t}
\Vert M_{\chi_{Z_{\frac{4a}{5}}}} M_{\psi}
\Delta_{Z,g}e^{-s\Delta_{Z,g}} \Vert_{1} ds.
\end{multline*}
To obtain a bound, we use a similar method as in Section \ref{subsection:tcp}.
Let $\phi$ be the auxiliary function defined by
equation (\ref{eq:phiaux}) with $\beta=1/2$. Then the trace norm of the operator
$M_{\chi_{Z_{\frac{4a}{5}}}} M_{\psi}
\Delta_{Z,g}e^{-s\Delta_{Z,g}}$ satisfies:
$$\Vert M_{\chi_{Z_{\frac{4a}{5}}}} M_{\psi}
\Delta_{Z,g}e^{-s\Delta_{Z,g}}\Vert_{1} \leq \Vert
M_{\chi_{Z_{\frac{4a}{5}}}} M_{\psi}
\Delta_{Z,g}e^{-s/2\Delta_{Z,g}}M_{\phi}^{-1}\Vert_{2} \Vert
M_{\phi}e^{-s/2\Delta_{Z,g}}\Vert_{2}.$$ The terms on the right-hand side can be estimated in a similar way as before to obtain:
\begin{eqnarray*}
\Vert M_{\chi_{Z_{\frac{4a}{5}}}} M_{\psi}
\Delta_{Z,g}e^{-s/2\Delta_{Z,g}}M_{\phi}^{-1}\Vert_{2} &\ll &
s^{-7/4}(a^{-k} + a^{-k+1/2}),\\
\Vert M_{\phi}
e^{-s/2\Delta_{Z,g}} \Vert_{2} &\ll & s^{-3/4}.
\end{eqnarray*}
It follows that:
$$\vert J_{2}\vert \ll \int_{t/2}^{t} s^{-7/4} (a^{-k} + a^{-k+1/2})\cdot
s^{-3/4} ds \ll a^{-k+1/2} t^{-3/2}.$$

Now, for $J_{3}$ we proceed in a similar way as for $J_{2}$ to obtain:
$$\vert J_{3}\vert \ll \int_{t/2}^{t} a^{-k+1/2}
s^{-7/4} s^{-3/4} ds \ll a^{-k+1/2} t^{-3/2},$$
see the Appendix for all the details. From all the equations above we obtain:
\begin{align*} \vert \tr(M_{\psi_{2}} (T^{-1}e^{-t\Delta_{Z,h}}T -
e^{-t\Delta_{Z,g}})) \vert
\ll a^{-k+1/2} t^{-3/2} + a e^{-c'/t} \ll a^{-k+1/2} t^{-3/2},
\end{align*}
for $0< t <1$.
\end{proof}


\begin{theorem} Let $\nu \geq 1$. Write $z\in Z$ as $z=(y,x)$. Let $\varphi\vert_{Z}(z)$, $\Delta_{g}\varphi\vert_{Z}(z)$,
and $\vert \nabla_{g}\varphi\vert_{g} \vert_{Z}(z)$
be $O(y^{-k})$ as $y\to \infty$ with $k\geq 5\nu +8$.
In addition, if $\nu \geq 3$ we require for $2\leq \ell \leq \nu$ that
$\vert\nabla^{\ell}\varphi\vert_{g} \vert_{Z}(z) = O(y^{-k})$ with $k\geq 5(\nu -2)-1$.
Then under these conditions, there is an expansion of the relative heat trace of the form:
\begin{equation}
\tr(T^{-1}e^{-t\Delta_{h}}T - e^{-t\Delta_{g}}) = t^{-1}\sum_{i =0}^{\nu} a_{i}t^{i}
+ O(t^{\nu}), \text{\ as } t\to 0.
\label{eq:pfexpt0rhtcmg}
\end{equation} \label{theorem:asympexpanrht}
\end{theorem}

\begin{proof}
The argument of the proof started above. To complete the proof we need to put together the proofs of
Proposition \ref{prop:casmpexptraoA} and
\ref{prop:b1nauxoBaccf} in a consistent manner.
First of all, we need to make all our estimates independent of $a$. In particular, the estimate of equation
(\ref{eq:auxbisa}). This particular estimate is going to determine our result. In equation (\ref{eq:auxbisa})
the right-hand side is estimated by $t^{-1}a e^{-\frac{c_{2}}{a^4t}}$. Taking $a=t^{-1/5}$, we get $a^{4}t=t^{1/5}$. Therefore
equation (\ref{eq:auxbisa}) becomes:
\begin{multline*} \int_{1}^{t^{-1/5}}\int_{0}^{1} \vert \wt{\psi}_{2}(\wt z)
\sum_{m\neq 0} (k_{h}(\wt{z}, \wt{z}+m,
t)e^{2\hat{\varphi}(\wt{z}+m)} - k_{g}(\wt{z}, \wt{z}+m , t))\vert \
dA_{\hat{g}}(\wt z)\\
\ll e^{-{\frac{c_{2}}{2t^{1/5}}}}.
\end{multline*}
The next step is to make sure that the asymptotic expansion in equation (\ref{eq:aelhkiv}) is kept
when we pass to the limit as $t\to 0$. Before we continue with the asymptotics of $I_{1}(t)$, let us consider again the
estimate of $I_{2}(t)$ and replace $a=t^{-1/5}$ in equation (\ref{eq:bI2itakt}).
In order to have
\begin{equation} \vert \tr(M_{\psi_{2}} (T^{-1}e^{-t\Delta_{Z,h}}T -
e^{-t\Delta_{Z,g}})) \vert
\ll (t^{-1/5})^{-k+1/2} t^{-3/2}\ll t^{\nu} \label{eq:estI2tnu}
\end{equation}
with $\nu\geq 1/2$ we need that ${\frac{k}{5}} - {\frac{1}{10}}-{\frac{3}{2}} \geq \nu$. Thus, $k$ should satisfy
$k\geq 5\nu + 8$. This condition applies to the conformal factor and its derivatives up to
second order.

Now, let us go back to the asymptotics of $I_1(t)$. Let $\nu\geq 1$. Replacing $a=t^{-1/5}$ in equation
(\ref{eq:aelhkiv}), $\wt{I}_1(t)$ becomes:
\begin{eqnarray}
\wt{I}_1(t)&=& t^{-1} \int_{1}^{t^{-1/5}}\int_{0}^{1} \wt{\psi}_{2}(\wt z)\sum_{j =0}^{\nu}
t^{j} (e^{2\hat{\varphi}}a_{j}(\hat{h},\wt{z})-a_{j}(\hat{g},\wt{z}))dA_{\hat{g}}(\wt z)\label{eq:rhtupa}\\
& & + \int_{1}^{t^{-1/5}}\int_{0}^{1} \wt{\psi}_{2}(\wt z)( e^{2\hat{\varphi}}{\mathcal R}_{\nu}(\hat{h},\wt{z},t) -
{\mathcal R}_{\nu}(\hat{g},\wt{z},t))
dA_{\hat{g}}(\wt z). \notag\end{eqnarray}
The integral of the remainder terms was estimated in equation (\ref{eq:estRt}), independently of $t$ and $a$.
In what follows we set $\wt{\psi}_{2}=1$ and drop the hat in $\hat\varphi$.
To deal with the convergence of the integrals in the first term on the right-hand side in equation (\ref{eq:rhtupa})
we fix $j$ and split each integral as follows:
\begin{multline*}
\int_{1}^{t^{-1/5}}\int_{0}^{1}
(e^{2{\varphi}}a_{j}(\hat{h})-a_{j}(\hat{g}))dA_{\hat{g}}
= \int_{1}^{\infty}\int_{0}^{1} (e^{2{\varphi}}a_{j}(\hat{h})-a_{j}(\hat{g}))dA_{\hat{g}}\\
-\int_{t^{-1/5}}^{\infty}\int_{0}^{1}(e^{2{\varphi}}a_{j}(\hat{h})-a_{j}(\hat{g}))dA_{\hat{g}}
\end{multline*}
Our goal is to prove that for each $j$ the integral over $[1,\infty)\times[0,1]$ converges and that the
integral over $[t^{-\frac15},\infty)\times[0,1]$ can be suitably estimated.

First of all, note that the region of integration $[1,\infty)\times [0,1]$ has finite area respect to both metrics $\hat{g}$ and
$\hat{h}$. Since $\hat g$ is the hyperbolic metric on $\H^2$, the functions $a_{k}(\hat g,\wt z)$ are bounded, therefore integrable.
Let us describe the general picture. Our goal is to prove the following equation:
\begin{eqnarray}
\wt{I}_1(t)&=& t^{-1} \sum_{j =0}^{\nu} t^{j} \int_{1}^{\infty}\int_{0}^{1}
(e^{2{\varphi}}a_{j}(\hat{h},\wt{z})-a_{j}(\hat{g},\wt{z}))dA_{\hat{g}}(\wt z)\notag \\
& &+ \ t^{-1}\sum_{j =0}^{\nu} t^{j} \int_{t^{-1/5}}^{\infty}\int_{0}^{1}
 (e^{2{\varphi}}a_{j}(\hat{h},\wt{z})-a_{j}(\hat{g},\wt{z}))dA_{\hat{g}}(\wt z) + O(t^\nu)\notag \\
&=& \sum_{j=0}^{\nu}(t^{-1}t^{j} \wt{a}_j + O(t^{\nu})) + O(t^{\nu}) =
\sum_{j=0}^{\nu}t^{-1}t^{j} \wt{a}_j  + O(t^{\nu}),\label{eq:ilthtecs}
\end{eqnarray}
where the coefficients $\wt{a}_{j}$ are given by:
$$\wt{a}_{j} = \int_{1}^{\infty}\int_{0}^{1}
(e^{2\hat{\varphi}}a_{j}(\hat{h},\wt{z})-a_{j}(\hat{g},\wt{z})) dA_{\hat{g}}(\wt z).$$

For each $j$ with $0\leq j\leq \nu$ we find conditions on the decay of $\varphi$, on the number of derivatives that should decay,
and on the order of that decay such that the corresponding integral converges or is suitably estimated.
At the end, we impose the strongest condition on $\varphi$ and its derivatives coming from all the terms together.

At each level $j$ (the sub-index of the heat invariant) we assume that $\varphi$ and its derivatives (we will see each time how many
derivatives we need) decay as $y^{-k}$ at infinity, then we find $k$ in terms of $\nu$ and $j$.

Let us proceed with the analysis of the heat invariants. We analyze the convergence and estimation of the integrals simultaneously.

For $a_{0}$ we have:
$$\int_{1}^{\infty}\int_{0}^{1} (e^{2\varphi} -1)dA_{\hat g}= A_{\hat h}([1,\infty)\times [0,1]) - 1$$
and
$$t^{-1}\int_{t^{-1/5}}^{\infty}\int_{0}^{1} \vert e^{2\varphi} -1\vert dA_{\hat g}\ll \int_{t^{-1/5}}^{\infty} y^{-k}
\frac{dy}{y^{-2}} = t^{-1} \frac{1}{k+1} t^{\frac{k+1}{5}}$$
In order to have $t^{-1}t^{\frac{k+1}{5}}\leq t^{\nu}$ we need $\varphi$ to decay as $k\geq 5\nu +4$.

For $a_1$ the integrals are:
\begin{multline*}\int_{1}^{\infty}\int_{0}^{1} (e^{2\varphi}R_{\hat h} -R_{\hat g})dA_{\hat g}=
\int_{1}^{\infty}\int_{0}^{1} ((\Delta_{\hat g}\varphi + R_{\hat g}) -R_{\hat g})dA_{\hat g}\\
=\int_{1}^{\infty}\int_{0}^{1} \Delta_{\hat g}\varphi \ dA_{\hat g} \ll 1
\end{multline*}
and
$$\int_{t^{-1/5}}^{\infty}\int_{0}^{1} \vert e^{2\varphi}\Delta_{\hat g}\varphi\vert dA_{\hat g}\ll \int_{t^{-1/5}}^{\infty} y^{-k}
\frac{dy}{y^{-2}} = \frac{1}{k+1} t^{\frac{k+1}{5}}.$$
Here we need $\Delta_{\hat g}\varphi$ to decay as $k\geq 5\nu -1$.

The second heat invariant $a_2$ is given in \cite{OPS2} as $a_{2} = \frac{\pi}{60}\int_{M}R^{2} dA$. In our case we obtain:
\begin{multline*}\int_{1}^{\infty}\int_{0}^{1} (e^{2\varphi}R_{\hat h}^2 -R_{\hat g}^2)dA_{\hat g}=
\int_{1}^{\infty}\int_{0}^{1} e^{-2\varphi}(\Delta_{\hat g}\varphi + R_{\hat g})^2 -R_{\hat g}^2\ dA_{\hat g}\\
=\int_{1}^{\infty}\int_{0}^{1} e^{-2\varphi}(\Delta_{\hat g}\varphi)^2 + e^{-2\varphi}(\Delta_{\hat g}\varphi)R_{\hat g}
\ dA_{\hat g} \ll 1.
\end{multline*}
For integral over $[t^{-1/5}, \infty)\times [0,1]$ we have:
\begin{multline*}t\int_{t^{-1/5}}^{\infty}\int_{0}^{1} \vert e^{-2\varphi}(\Delta_{\hat g}\varphi)^2 +
e^{-2\varphi}(\Delta_{\hat g}\varphi)R_{\hat g}
\ dA_{\hat g}\vert
\ll \frac{t^{\frac{2k+1}{5}+1}}{2k+1}  + \frac{t^{\frac{k+1}{5}+1}}{k+1} \end{multline*}
The left-hand side is bounded by $t^\nu$ if $\frac{k+1}{5}+1\geq \nu$, i.e if $k\geq 5\nu - 6$. In this case $\nu \geq 2$, and we
need two derivatives. 

Now, let us go one step forward and consider the third heat invariant as it is given in \cite{Sakai}:
$$a_3=\frac{1}{4\pi}\int_{M}-9\vert \nabla R\vert^2 + 4R^3 dA$$
Before we proceed, let us perform some computations:
\begin{eqnarray*}\nabla_{\hat h} R_{\hat h} &=&
-2 e^{-2\varphi}(\Delta_{\hat g}\varphi-1)(\nabla_{\hat h} \varphi) + e^{-2\varphi}\nabla_{\hat h}(\Delta_{\hat g}\varphi)\\
\vert \nabla_{\hat h} R_{\hat h}\vert^2_{\hat h}
&=& 4 e^{-4\varphi}(\Delta_{\hat g}\varphi-1)^2 \vert \nabla_{\hat h} \varphi \vert^2_{\hat h}
 -4e^{-4\varphi}(\Delta_{\hat g}\varphi-1) \langle \nabla_{\hat h} \varphi, \nabla_{\hat h}(\Delta_{\hat g}\varphi)\rangle_{\hat h}\\
& & +\ e^{-4\varphi}\vert \nabla_{\hat h}(\Delta_{\hat g}\varphi)\vert^2_{\hat h}\\
R_{\hat h}^3 &=& e^{-6\varphi}(\Delta_{\hat g}\varphi + R_{\hat g})^3
= e^{-6\varphi}((\Delta_{\hat g}\varphi)^3 - 3(\Delta_{\hat g}\varphi)^2 + 3(\Delta_{\hat g}\varphi) - 1)
\end{eqnarray*}

Plugging the expressions above in the integrals under consideration we obtain:
\begin{multline*}\int_{1}^{\infty}\int_{0}^{1} e^{2\varphi}
(-9\vert \nabla_{\hat h} R_{\hat h}\vert^2_{\hat h} + 4R_{\hat h}^3) - (-9\vert \nabla R_{\hat g}\vert^2 + 4R_{\hat g}^3)
dA_{\hat g}\\
= 4 \int_{1}^{\infty}\int_{0}^{1} e^{-4\varphi}((\Delta_{\hat g}\varphi)^3 - 3(\Delta_{\hat g}\varphi)^2 + 3(\Delta_{\hat g}\varphi))
+ (1-e^{-4\varphi}) \ dA_{\hat g}\\
-9 \int_{1}^{\infty}\int_{0}^{1} e^{-4\varphi}\{
4 (\Delta_{\hat g}\varphi-1)^2 \vert \nabla \varphi \vert^2 - 4 (\Delta_{\hat g}\varphi-1) \langle \nabla \varphi, \nabla(\Delta_{\hat g}\varphi)\rangle + \vert \nabla(\Delta_{\hat g}\varphi)\vert^2\} \ dA_{\hat g}
\end{multline*}
Since $\vert \nabla_{\hat h} \varphi \vert^2_{\hat h}= e^{-2\varphi}\vert \nabla_{\hat g} \varphi \vert^2_{\hat g}$ and we drop the subindice when we consider the metric $\hat g$.
In the first integral of the last equality, all functions decay at infinity.
For convergence of the second integral, it is enough to require boundedness of the integrand, i.e.
$\vert \nabla (\Delta \varphi)\vert \ll 1$.

Now, we estimate the integrals on the region $[t^{-\frac15},\infty)\times [0,1]$.
As above, let us assume that $\vert \nabla^{\ell}_{\hat g}\varphi\vert = O(y^{-k})$, for $0\leq \ell\leq 3$
then $\vert \Delta_{\hat g}\varphi-1\vert\ll 1$,
$e^{-4\varphi}-1=O(y^{-k})$, and
\begin{multline*}t^{2}\int_{t^{-1/5}}^{\infty}\int_{0}^{1} \vert
e^{2\varphi} (-9\vert \nabla R_{\hat h}\vert^2 + 4R_{\hat h}^3) - (-9\vert \nabla R_{\hat g}\vert^2 + 4R_{\hat g}^3)\vert
dA_{\hat g}\\
\ll t^{2}\int_{t^{-1/5}}^{\infty}\int_{0}^{1} \left(
\vert \Delta_{\hat g}\varphi\vert^3 + \vert\Delta_{\hat g}\varphi\vert^2 + \vert\Delta_{\hat g}\varphi\vert
+ \vert1-e^{-4\varphi}\vert \right.\\ \left. +
\vert \nabla \varphi \vert^2 + \vert \nabla \varphi\vert \vert\nabla(\Delta_{\hat g}\varphi)\vert
+ \vert \nabla(\Delta_{\hat g}\varphi)\vert^2\right)
\ dA_{\hat g}\\
\ll
t^{2} \int_{t^{-1/5}}^{\infty} \left( y^{-3k} + y^{-2k} + y^{-k}\right) \ \frac{dy}{y^{2}}\\
= \frac{t^{\frac{3k+1}{5}+2}}{3k+1} + \frac{t^{\frac{2k+1}{5}+2}}{2k+1}  + \frac{t^{\frac{k+1}{5}+2}}{k+1}.
\end{multline*}
In the same way as in the previous case we need that $\frac{k+1}{5}+2\geq \nu$. This is
achieved if $k\geq 5\nu - 11$, ($\nu\geq 3$).

General formulas for the coefficients in the expansion of the heat kernel are very complicated and only known explicitly for
few of them. However, it is known that the functions $a_{k}(\hat h,\wt z)$ are polynomials of degree $2k$ in the scalar curvature
($2R_{\hat h}$) and half powers of the Laplacian. The leading coefficients of this polynomials are described in \cite{OPS2} and in a
more explicit form by Branson, Gilkey and $\O$rsted in \cite{BGO}. We refer to Lemma 1.3 and (1.4) in \cite{BGO}.
$$a_j(\Delta)= \int_{M}(j(j-1)c_j)\vert \nabla^{j-2}R\vert^2 + \text{polynomial}(R,\nabla R,\dots \nabla^{j-3}R),$$
for $j\geq 3$. These are the heat coefficients for a closed Riemann surface (in \cite{BGO} $R$ denotes the scalar curvature).
Applying this to our case, we require at least $\vert\nabla^{j-2}_{\hat h}R_{\hat h}\vert$ to be bounded for $0\leq \ell \leq j-2$. In terms of the conformal factor, this condition translates to $\vert\nabla^{\ell}\varphi \vert \ll 1$ for $2\leq \ell \leq j$.
Under these requirements, the integrals defining the coefficients $\wt{a}_j$ converge.

Now let us estimate the integral over $[t^{-1/5},\infty)\times [0,1]$, assuming that $\vert\nabla^{\ell}\varphi \vert = O(y^{-k})$ for $2\leq \ell \leq j$:
\begin{multline}\int_{t^{-1/5}}^{\infty}\int_{0}^{1} (j(j-1)c_j)(e^{2\varphi}\vert \nabla_{\hat h}^{j-2}R_{\hat h}\vert^2
-\vert \nabla_{\hat g}^{j-2}R_{\hat g}\vert^2)\\ + e^{2\varphi}\text{polynomial}(R_{\hat h},\nabla_{\hat h} R_{\hat h},\dots \nabla_{\hat h}^{j-3}R_{\hat h})\\ - \text{polynomial}(R_{\hat g},\nabla_{\hat g} R_{\hat g},\dots \nabla_{\hat g}^{j-3}R_{\hat g}) \ dA_{\hat{g}}(\wt z)\label{eq:eltj}
\end{multline}

If $j\geq 3$, $\nabla_{\hat g}^{j-2}R_{\hat g}=0$, therefore the leading term is of the form
$\vert \nabla_{\hat h}^{j-2}R_{\hat h}\vert^2 = O(y^{-2k})$.
Now, let us consider the terms involved in the polynomial. For that, we assume that the polynomial is of the form:
$$p_{2j} (x_1, \dots, x_{r})= \sum a_{i_1 \dots i_r}x_1^{i_1}\cdots x_r^{i_r},$$
then we have terms of the form:
$$e^{2\varphi}a_{i_1 \dots i_r}R_{\hat h}^{i_1}(\nabla_{\hat h}R_{\hat h})^{i_2} \cdots (\nabla_{\hat h}^{j-3} R_{\hat h})^{i_r} - a_{i_1 \dots i_r}R_{\hat g}^{i_1}(\nabla_{\hat g}R_{\hat g})^{i_2} \cdots (\nabla_{\hat g}^{j-3} R_{\hat g})^{i_r}$$

If $i_j\neq 0$ for some $j>1$, the second term vanishes. So we are left only with:
$$e^{2\varphi}a_{i_1 \dots i_r}R_{\hat h}^{i_1}(\nabla_{\hat h}R_{\hat h})^{i_2} \cdots (\nabla_{\hat h}^{j-3} R_{\hat h})^{i_r}$$
that involve at least one derivative of $R_{\hat h}$:
$\nabla^{\ell}_{\hat h}R_{\hat h}=(\nabla^{\ell}_{\hat h}(\Delta\varphi + R_{\hat g}))^{i_j}
= (\nabla^{\ell}_{\hat h}(\Delta\varphi))^{i_j} = O(y^{-k \cdot i_j})$.

If $i_j = 0$ for all $j>1$, we have terms of the form:
\begin{eqnarray*}
a_{i_1 \dots i_r}(e^{2\varphi}R_{\hat h}^{i_1}-R_{\hat g}^{i_1}) &=&
a_{i_1 \dots i_r}(e^{2(1-i_1)\varphi}(\Delta \varphi+R_{\hat g})^{i_1}-R_{\hat g}^{i_1})\\
&=& a_{i_1 \dots i_r}(e^{2(1-i_1)\varphi}\sum_{\ell=0}^{i_1} \binom{i_1}{\ell} (\Delta\varphi)^{\ell} (R_{\hat g})^{i_1-\ell}-R_{\hat g}^{i_1})   \\
\vert a_{i_1 \dots i_r}(e^{2\varphi}R_{\hat h}^{i_1}-R_{\hat g}^{i_1})\vert &\ll &
\left(\sum_{\ell=1}^{i_1} y^{-k \ell}\right) + (e^{2(1-i_1)\varphi}-1)R_{\hat g}^{i_1},
\end{eqnarray*}
and recall that $1 - e^{-2\ell \varphi}= O(y^{-k})$. Therefore:
$$t^{j-1}\int_{t^{-1/5}}^{\infty}\int_{0}^{1} (e^{2\varphi}a_{j}(\hat{h})-a_{j}(\hat{g})) dA_{\hat{g}}
\ll t^{j-1}t^{\frac{k+1}{5}},$$
the last term is bounded by $t^{\nu}$ if $k\geq 5\nu - 5(j-1) -1$. 
We have finished the proof of equation (\ref{eq:ilthtecs}).

It is interesting to see how, as we want to have more terms in the expansion, although more derivatives need to be considered,
the conditions on their decay become weaker.
However, this fact does not have any implication on our purposes of defining relative determinants.
We could try to further refine the requirements to minimize conditions on $\varphi$ but that will imply a deeper analysis of the heat
invariants that is beyond the purpose of this article.
\end{proof}

\begin{corollary}
If the conformal factor $\varphi$ and all its derivatives decay at infinity to infinite order,
then there is a complete asymptotic
expansion of the relative heat trace as $t\to 0$:
$$\tr(T^{-1}e^{-t\Delta_{h}}T - e^{-t\Delta_{g}}) = t^{-1}\sum_{j=0}^{\infty}a_j t^j.$$
\end{corollary}

\begin{corollary} Let $h=e^{2\varphi}g$ with $\varphi\vert_{Z}(z)$, $\Delta_{g}\varphi\vert_{Z}(z)$,
and $\vert \nabla_{g}\varphi\vert_{g} \vert_{Z}(z)$
be $O(y^{-k})$ as $y\to \infty$ with $k\geq 11$. Then the relative heat trace has an expansion of the form:
\begin{equation}
\tr(T^{-1}e^{-t\Delta_{h}}T - e^{-t\Delta_{g}}) = a_{0}t^{-1} + a_{1} + O(\sqrt{t})\  \text{\ as } t\to 0.
\label{eq:asymexptghtsqrt}
\end{equation}
We will see in Section \ref{section:relativedet} that this condition is sufficient to define the relative determinant.
\label{cor:aestdrd}
\end{corollary}

\begin{proof}
The condition $k\geq 11$ comes from taking $\nu=1/2$ in equation (\ref{eq:estI2tnu}).
In the part corresponding to $\wt{I}_1(t)$ we take $\nu =1$.
The heat invariants $a_0$ and $a_1$ require $\varphi$ to decay at least as
$k=9$ and $\Delta \varphi$ to decay as $k=4$. The strongest condition is then determined by $I_2$.
\end{proof}

To compute the coefficients in the expansion
(\ref{eq:pfexpt0rhtcmg}) we use that the coefficients in the local
expansion of the heat kernels are given by universal functions. Taking $\nu =2$,
we have that:
\begin{multline}
\tr(T^{-1}e^{-t\Delta_{h}}T - e^{-t\Delta_{g}})=
{\frac{t^{-1}}{4\pi}} (A_{h}-A_{g})\\ + \ t\ \frac{\pi}{60}\left(\int_{M}
R_h^2(z) dA_{h}(z) - \int_{M} R_{g}^{2}(z) dA_{g}(z)\right)
 + O(t^{2}), \text{ as } t\to 0,
\end{multline}
where the constant term vanishes due to Gauss-Bonnet's theorem.
Equation (\ref{eq:asymexptghtsqrt}) becomes:
\begin{equation}
\tr(T^{-1}e^{-t\Delta_{h}}T - e^{-t\Delta_{g}})=
{\frac{t^{-1}}{4\pi}} (A_{h}-A_{g}) + O(\sqrt{t}), \text{ as } t\to 0,
\label{eq:asett0hohahogwcc}
\end{equation}

\subsection{Asymptotics of other relative heat traces.}

Let us consider again surfaces with several cusps. Let $(M,g)$
be a swc of genus $p$ and with $m$ cusps. Assume that $M$
can be decomposed as $M=M_{0}\cup Z_{a_{1}} \cup
\cdots Z_{a_{m}}$, where $a_{i}\geq 1$ for $1\leq i \leq m$. Let $\bar{\Delta}_{a,0}$
be the direct sum $\oplus_{j=1}^m \Delta_{a_{j},0}$ of the Dirichlet
Laplacians $\Delta_{a_{j},0}$ defined in Definition \ref{def:Laplmcuspas}.

Proposition 6.4 in \cite{Mu} establishes that the operator
$e^{-t\Delta_{g}}-e^{-t\bar{\Delta}_{a,0}}$ is trace class and its
trace has the following asymptotic expansion as $t\to 0$:
\begin{multline} \begin{split}\tr(e^{-t\Delta_{g}}-e^{-t\bar{\Delta}_{a,0}}) =
{\frac{A_{g}}{4\pi}} t^{-1} + ({\frac{\gamma m}{2}}+
\sum_{j=1}^{m}\log(a_{j})){\frac{1}{\sqrt{4\pi t}}}\\ + \frac{m}{2}{\frac{\log(t)}{\sqrt{4\pi t}}} + {\frac{\chi(M)}{6}} +{\frac{m}{4}} +
O(\sqrt{t}), \label{eq:asymexptzero} \end{split}\end{multline}
where $\gamma$ is the Euler constant. A close examination of the proof of equation (\ref{eq:asymexptzero})
in \cite{Mu} shows that the term
$\sum_{j=1}^{m}{\frac{\log(a_{j})}{\sqrt{4\pi t}}}$ can be replaced
by $e^{-t/4}\sum_{j=1}^{m} {\frac{\log(a_{j})}{\sqrt{4\pi t}}}$.

In particular, we can consider the relative determinant of the pair
$(\Delta_{g},\bar{\Delta}_{1,0})$. To that purpose we consider the
trace $\tr(e^{-t\Delta_{g}}-e^{-t\bar{\Delta}_{1,0}})$, where the trace is taken
in an extended $L^{2}$ space that is given by:
\begin{multline}L^{2}(M,dA_{g})\oplus \oplus_{j=1}^{m}
L^{2}([1,a_{j}],y^{-2}dy)\\ = L^{2}(M_{0},dA_{g})\oplus
\oplus_{j=1}^{m} (L^{2}_{0}(Z_{a_{j}}) \oplus
L^{2}([1,\infty),y^{-2}dy)). \label{eq:extL2mc&dcch5}
\end{multline} Thus, using Proposition \ref{prop:tcphocahoc1} and
equations (\ref{eq:reltrexspLgL1}) and (\ref{eq:asymexptzero}) we
obtain the following asymptotic expansion as $t\to 0$:
\begin{multline}\begin{split}\tr(e^{-t\Delta_{g}}-e^{-t\bar{\Delta}_{1,0}})
={\frac{A_{g}}{4\pi}}t^{-1} + \frac{\gamma m}{2}{\frac{1}{\sqrt{4\pi t}}} +
\frac{m}{2}{\frac{\log(t)}{\sqrt{4\pi t}}} \\
+ {\frac{\chi(M)}{6}}+{\frac{m}{4}} +
O(\sqrt{t}).\label{eq:asymexptzerovar1}\end{split}
\end{multline}
Together with equation (\ref{eq:asett0hohahogwcc}) this gives:
\begin{multline}\begin{split}
\tr(T^{-1}e^{-t\Delta_{h}}T - e^{-t\bar{\Delta}_{1,0}})=
{\frac{A_{h}}{4\pi}}t^{-1} + {\frac{\gamma
m}{2}}{\frac{1}{\sqrt{4\pi t}}} + \frac{m}{2}{\frac{\log(t)}{\sqrt{4\pi t}}}\\
+ {\frac{\chi(M)}{6}} +{\frac{m}{4}} + O(\sqrt{t}),
\label{eq:asymexptzerohD0}
\end{split}
\end{multline}
where the transformation $T$ is the identity in the space
$\oplus_{j=1}^{m} L^{2}([1,a_{j}],y^{-2}dy)$.


\section{Relative determinants on surfaces with asymptotically cusp ends}
\subsection{Definition}
 \label{section:relativedet}
The relative determinant on a surface with hyperbolic cusps was already
considered by W. M\"uller in \cite{Mu3}. Therefore, we restrict our attention
to the definition and properties of the relative
determinant on asymptotically hyperbolic surfaces.
Let $(M,g)$ be a swc and let $h=e^{2\varphi}g$. In order to define the relative determinant of the pairs
$(\Delta_{h},\Delta_{g})$, and $(\Delta_{h},\Delta_{1,0})$, we need to verify that the conditions given in Section
\ref{subsection:reldet} are satisfied. Let $k\geq 1$, let us define the following set of functions:
\begin{multline*}{\mathcal F}_{k}:=\{\psi \in C^{\infty}(M) \vert \ \psi(z), \vert \nabla_{g} \psi \vert(z) \text{ and }
\Delta_{g}\psi(z)\\ \text{ are } O(i(z)^{-k}) \text{ as } y=i(z)\to
\infty\}.\end{multline*}

Sections \ref{section:2opeinthecusp} and \ref{section:asypexpant0} establish that the
first and second conditions are fulfilled provided that $\varphi \in \cF_1$ and $\varphi \in \cF_{11}$, respectively.

The third condition in Section \ref{subsection:reldet} is about the behavior of the relative heat trace
for big values $t$. The trace class property together with the fact that
$\sigma_{ac}(\Delta_{1,0})=[1/4,\infty)$ and Lemma 2.22 in
\cite{Mu3} give the existence of a constant $C_{1}>0$ such that:
\begin{equation}
\tr(T^{-1}e^{-t\Delta_{h}}T-e^{-t\Delta_{1,0}}) = 1 +
O(e^{-C_{1}t}), \quad \text{ as } t\to \infty,
\label{eq:asymextinfty}
\end{equation}
where the value $1$ on the right-hand side comes from $\dim \ker \Delta_{h} - \dim \ker \Delta_{1,0}$ and
the trace is taken in $L^{2}(M,dA_{g})$. This condition is satisfied even when $\varphi \in \cF_1$.

Let us prove that the condition $\varphi\in \cF_{11}$ suffices
to define the relative determinant of $(\Delta_{h},\Delta_{1,0})$. The
relative zeta function $\zeta(s; \Delta_{h}, \Delta_{1,0})$ converges on
$\Re(s)>1$. It follows
from the asymptotic expansions (\ref{eq:asymexptzerohD0}) and
(\ref{eq:asymextinfty}) that the function $\zeta(s; \Delta_{h},
\Delta_{1,0})$ has a meromorphic continuation to the complex plane,
that it is regular at $s=0$. This continuation is denoted again by $\zeta$.
The proof of the existence of the continuation and regularity at $s=0$ is classical
in the literature. However we include it here to remark that it is enough to have a
truncated asymptotic expansion.

For the sake of simplicity, let us take $m=1$ and let us fix the notation in equation
(\ref{eq:asymexptzerohD0}) above:
\begin{equation*} a_{0} = {\frac{A_{h}}{4\pi}} \quad
a_{10}={\frac{\gamma}{4\sqrt{\pi}}}, \quad a_{11}=
{\frac{1}{4\sqrt{\pi}}}, \quad a_{2}=
{\frac{\chi(M)}{6}}+{\frac{1}{4}}.
\end{equation*}

Now, let us write $\zeta(s; \Delta_{h}, \Delta_{1,0})$ as
$\zeta_{1}(s)+\zeta_{2}(s)$ with
\begin{align*} \zeta_{1}(s) &:= {\frac{1}{\Gamma(s)}}\int_{0}^{1} t^{s-1} (\tr
(T^{-1}e^{-t\Delta_{h}}T-e^{-t\Delta_{1,0}})-1) dt \quad \text{and}\\
\zeta_{2}(s) &:= {\frac{1}{\Gamma(s)}}\int_{1}^{\infty} t^{s-1} (\tr
(T^{-1}e^{-t\Delta_{h}}T-e^{-t\Delta_{1,0}})-1) dt.
\end{align*}

Equation (\ref{eq:asymextinfty}) implies that $\zeta_{2}(s)$ is analytic
at $s=0$. As for $\zeta_{1}(s)$ and $\Re(s)>1$, we have that:
\begin{eqnarray*}
\zeta_{1}(s)&= &{\frac{1}{\Gamma(s)}}\int_{0}^{1} t^{s-1}
(a_{0}t^{-1} + (a_{10}+ a_{11}\log t)t^{-1/2} + a_{2}-1 + \vartheta(t)) dt\\
&=& {\frac{1}{\Gamma(s)}} \left( {\frac{a_{0}}{s-1}} +
{\frac{a_{10}}{s-1/2}} -  {\frac{a_{11}}{(s-1/2)^{2}}}+
{\frac{a_{2}-1}{s}} + \vartheta_{1}(s) \right)_,
\end{eqnarray*}
where $\vartheta(t)=O(\sqrt{t}))$ and $\vartheta_{1}(s)$ is a
function that is analytic at $s=0$.

Therefore, we can define the (regularized) relative determinant of
$(\Delta_{h},{\Delta}_{1,0})$ as in Section \ref{subsection:reldet}:
\begin{equation*}
\det(\Delta_{h}, {\Delta}_{1,0}) =
\exp\left(-\frac{d}{ds}\zeta(s;\Delta_{h}, {\Delta}_{1,0})
\Big|_{s=0}\right).
\end{equation*}
Note that we only need to require that the function $\varphi$ and its derivatives up to order two,
have a decay of order $11$ at infinity. The definition of $\det(\Delta_{h},\Delta_{g})$ is done in the same way.

\subsection{Polyakov's formula for the relative determinant.
Extremals} \label{section:variational}

In \cite{OPS1} the authors proved that on compact surfaces, with and
without boundary and under suitable restrictions, the regularized
determinant of the Laplace operator has an extremum.
In this section we
discuss the gene\-ra\-li\-zation of the extremal property of
determinants given by OPS to certain cases of surfaces with
asymptotically cusp ends.
The main tool to study extremal properties of determinants is
Polyakov's formula that relates the determinant of a given
metric to the determinant of a conformal perturbation of it.
The formula obtained here for relative determinants is the same as the
one for regularized determinants on compact surfaces given in
\cite{OPS1}. The proofs of the variational formula and of Polyakov's
formula follow the main lines of the corresponding proofs in
\cite{OPS1} but we focus in the technical details that allow us to
perform each step in the main proof.

\subsubsection{Polyakov's formula}
In this section we first consider $\varphi, \psi \in {\mathcal F}_{k}$ with $k\geq 11$ and $u \in \R$, let us define the family of metrics:
$$h_{u} := e^{2(\varphi + u\psi)} g = e^{2 u \psi} h.$$
The corresponding Laplace operators and area elements are given by the equations:
$$\Delta_{u} := \Delta_{h_{u}}= e^{-2u\psi}\Delta_{h}, \quad dA_{u}:=dA_{h_{u}}= e^{2u\psi}dA_{h}.$$
Let us consider the family of unitary maps given by:
$$T_{u}:L^{2}(M,dA_{u})\to L^{2}(M,dA_{h}), f\mapsto f e^{u\psi},$$
and the following functional:
\begin{multline*} F: {\mathcal F}_{k} \to {\mathbb C}, \psi \mapsto F_{s}(\varphi + u\psi) :=
\zeta(s; \Delta_{u}, \Delta_{1,0}),\\ \zeta(s; \Delta_{u}, \Delta_{1,0}) = {\frac{1}{\Gamma(s)}}
\int_{0}^{\infty} t^{s-1} (\tr (T_{u}e^{-t\Delta_{u}}T_{u}^{-1}- T
e^{-t\Delta_{1,0}}T^{-1})-1) dt,\end{multline*}
where the trace is taken in $L^{2}(M,dA_{h})$. The variation of $\zeta$ at $\varphi$ in
the direction of $\psi$ is defined as:
$${\frac{\delta\zeta}{\delta \psi}}(s;\Delta_{h},\Delta_{1,0}) :=
\left. {\frac{\partial}{\partial u}} F_{s}(\varphi + u\psi)
\right\vert_{u=0}.$$

In order to proceed with the computation of the derivative in the
equation above, we need the following lemma:

\begin{lemma} $$\left. {\frac{d}{d u}} \tr(T_{u}e^{-t\Delta_{u}}T_{u}^{-1} - Te^{-t\Delta_{1,0}}T^{-1})\right\vert_{u=0} = -t \tr
(\dot{\Delta}_{h} e^{-t\Delta_{h}}),$$ where $\dot{\Delta}_{h}
\equiv \left.{\frac{\partial}{\partial u}}\
\Delta_{u}\right\vert_{u=0}= -2\psi\Delta_{h}$.
\end{lemma}

\begin{proof}
Let $H_{u}=T_{u}\Delta_{u}T_{u}^{-1}$. Then $H_{u}$ is a family of
self-adjoint operators acting on $L^{2}(M,dA_{h})$. Note that
$e^{-tH_{u}} = T_{u}e^{-t\Delta_{u}}T_{u}^{-1}$. It is also clear
that:
\begin{equation*}
{\frac{d}{d u}} \tr(T_{u}e^{-t\Delta_{u}}T_{u}^{-1} -
Te^{-t\Delta_{1,0}}T^{-1}) = \tr \left( {\frac{d}{d u}} e^{-t
H_{u}}\right). \end{equation*}
Let $u_{1}, u_{2}>0$, with $u_{1}> u_{2}$. Let us apply Duhamel's
principle in terms of the operators:
$$e^{-t H_{u_{1}}} - e^{-t H_{u_2}} =  \int_{0}^{t} - e^{-s H_{u_{1}}} H_{u_{1}} e^{-(t-s)
H_{u_{2}}} + e^{-s H_{u_{1}}} H_{u_{2}} e^{-(t-s) H_{u_{2}}} \ ds.$$
Dividing by $u_{1}-u_{2}$ the previous equation and letting $u_{2}\to u_{1}$, we obtain:
$$\left.{\frac{d}{d u}}
\ e^{-t H_{u}}\right\vert_{u=u_{1}} = - \int_{0}^{t} e^{-s
H_{u_{1}}} \left({\left.{\frac{d}{d u}} H_{u}
\right\vert_{u=u_{1}}}\right) e^{-(t-s) H_{u_{1}}}\ ds.$$
Therefore we get:
\begin{equation}
{\frac{d}{d u}} \tr(T_{u}e^{-t\Delta_{u}}T_{u}^{-1} -
Te^{-t\Delta_{1,0}}T^{-1}) = -t \tr \left(\dot{H}_{u} e^{-tH_{u}}\right).
\label{eq:difuDpaux1}
\end{equation}
Let us compute the derivative $\dot{H}_{u}$:
\begin{align*}
{\frac{d}{d u}} H_{u} = \psi T_{u}\Delta_{u} T_{u}^{-1} + T_{u}
\left({\frac{d}{d u}} \Delta_{u}\right) T_{u}^{-1} - T_{u}\Delta_{u}
\psi T_{u}^{-1}.
\end{align*}
Thus we get
$$ \tr \left(\dot{H}_{u} e^{-tH_{u}}\right)
= \tr\left(\psi \Delta_{u} e^{-t\Delta_{u}} \right) + \tr\left(
\dot{\Delta}_{u} e^{-t \Delta_{u}}\right) - \tr \left(\Delta_{u}\psi
e^{-t\Delta_{u}}\right).$$

From the rate of decay assumed for $\psi$ and $\Delta_{g}\psi$ we
have that the operators $\psi e^{-t\Delta_{u}}$ and $\Delta_{u}\psi
e^{-t\Delta_{u}}$ are trace class. Using in addition that $e^{-t\Delta_{u}} \Delta_{u}$
is bounded for all $t>0$ we obtain:
\begin{align*}
\tr \left(\Delta_{u}\psi e^{-t\Delta_{u}}\right) = \tr
\left(e^{-{\frac{t}{2}}\Delta_{u}} \Delta_{u}\psi
e^{-{\frac{t}{2}}\Delta_{u}}\right) = \tr \left(\psi
e^{-t\Delta_{u}} \Delta_{u}\right) = \tr \left(\psi \Delta_{u}
e^{-t\Delta_{u}} \right).
\end{align*}
In this way we get:
$$\tr \left(\dot{H}_{u} e^{-tH_{u}}\right) = \tr\left(
\dot{\Delta}_{u} e^{-t \Delta_{u}}\right) = -2\tr\left( \psi
\Delta_{u} e^{-t \Delta_{u}}\right).$$ Taking $u=0$ in the previous
equation together with equation (\ref{eq:difuDpaux1}) implies the
statement of the lemma.
\end{proof}

We are ready to compute the variation of the relative zeta function:
\begin{multline*}
{\frac{\delta\zeta}{\delta\psi}}(s; \Delta_{h}, \Delta_{1,0})\\ =
{\frac{1}{\Gamma(s)}}\int_{0}^{\infty} t^{s-1} \left.{\frac{d}{d u}}
(\tr (T_{u} e^{-t
\Delta_{u}} T_{u}^{-1} - T e^{-t\Delta_{1,0}}T^{-1})-1)\right\vert_{u=0} dt\\
= {\frac{-1}{\Gamma(s)}} \int_{0}^{\infty} t^{s}
\tr((-2\psi{\Delta}_{h}e^{-t\Delta_{h}}) dt = {\frac{-2}{\Gamma(s)}}
\int_{0}^{\infty} t^{s} {\frac {\partial} {\partial t}} \tr(\psi
e^{-t \Delta_{h}}) dt,
\end{multline*}

Since
$${\frac {\partial} {\partial t}} \psi e^{-t \Delta_{h}} = {\frac
{\partial} {\partial t}} \ \psi (e^{-t \Delta_{h}} -
P_{\ker(\Delta_{h})}),$$ we have that
\begin{equation}{\frac{\delta\zeta}{\delta\psi}}(s; \Delta_{h}, \Delta_{1,0}) = {\frac{-2}{\Gamma(s)}}
\int_{0}^{\infty} t^{s} {\frac {\partial} {\partial t}} \tr(\psi
(e^{-t \Delta_{h}} - P_{\ker(\Delta_{h})})) dt.\label{eq:vzfgms1}
\end{equation}

In the classical proof of the variational formula of the spectral zeta function, the next step is to do integration by parts in equation
(\ref{eq:vzfgms1}). Before we do that, we have to verify the good decay of
$\tr(\psi (e^{-t \Delta_{h}} - P_{\ker(\Delta_{h})}))$ for big and small values of $t$.
In addition, we need to make sure that we can obtain an expansion of the trace, for small
values of $t$, whose remainder term can be integrated. We accomplish that in the following two lemmas:

\begin{lemma} There exists a constant $c>0$ such that:
$$\tr(\psi (e^{-t \Delta_{h}}  - P_{\ker(\Delta_{h})})) = O(e^{-ct}), \text{ as } t\to \infty.$$
\label{lemma:daitmrht} \end{lemma}
\begin{proof} Let $t>1$ and let us write:
\begin{align*}
\psi (e^{-t \Delta_{h}} - P_{\ker(\Delta_{h})}) = \psi
e^{-{\frac{1}{2}} \Delta_{h}} (e^{-(t-{\frac{1}{2}}) \Delta_{h}} -
P_{\ker(\Delta_{h})}),
\end{align*}
where we used that $e^{-{\frac{1}{2}}
\Delta_{h}}P_{\ker(\Delta_{h})} = P_{\ker(\Delta_{h})}$. By
Corollary \ref{lemma:tcppsihoh} we have that $\psi e^{-{\frac{1}{2}}
\Delta_{h}}$ is trace class. On the other hand, for $f\in L^{2}(M,dA_{h})$ the spectral
theorem implies that:
$$e^{-t \Delta_{h}}f - P_{\ker(\Delta_{h})}f = e^{-t(\Delta_{h}-P_{\ker(\Delta_{h})})}f.$$
Note that $\sigma_{\text{ess}}(\Delta_{h})= [1/4,\infty)$ implies
that $0$ is an isolated eigenvalue of $\Delta_{h}$ and
$\sigma(\Delta_{h}-P_{\ker(\Delta_{h})})\subseteq [c_{1},\infty)$
for some $c_{1}\in (0,1/4]$. Thus $$\Vert e^{-t
(\Delta_{h}-P_{\ker(\Delta_{h})})}\Vert_{L^{2}(M,h)}\leq
e^{-c_{1}t}$$ for any $t>0$. If $t>1$, $t-{\frac{1}{2}}>0$; therefore
the trace satisfies the desired estimate:
\begin{multline*}
\vert \tr(\psi (e^{-t \Delta_{h}} - P_{\ker(\Delta_{h})}))\vert
\leq
\Vert \psi e^{-{\frac{1}{2}} \Delta_{h}} (e^{-(t-{\frac{1}{2}}) \Delta_{h}} - P_{\ker(\Delta_{h})})\Vert_{1}\\
\leq \Vert \psi e^{-{\frac{1}{2}} \Delta_{h}} \Vert_{1} \Vert
e^{-(t-{\frac{1}{2}})(\Delta_{h}-P_{\ker(\Delta_{h})})}\Vert_{L^{2}(M,h)}
\ll e^{-c_{1}t}.
\end{multline*}
This proves Lemma \ref{lemma:daitmrht}.\end{proof}

\begin{lemma} For $0<t\leq 1$ the trace of the operator
$\psi(e^{-t\Delta_{h}}-P_{\ker(\Delta_{h})})$ has the following
expansion:
\begin{equation*}
\tr(\psi(e^{-t\Delta_{h}}-P_{\ker(\Delta_{h})})) = \int_{M} \psi(z)
\left({\frac{1}{4\pi t}} + {\frac{R_{h}(z)}{12\pi}} -
{\frac{1}{A_{h}}}\right) \ dA_{h} + O(t)
\end{equation*}
as $t\to 0$. \label{lemma:etphohmptgt0}
\end{lemma}

\begin{proof}
In order to prove Lemma \ref{lemma:etphohmptgt0} we use a method
similar to the one used in Section \ref{subsection:expanstrace1} to
prove the existence of the expansion of the relative heat trace
$\tr(e^{-t\Delta_{h}}-e^{-t\Delta_{g}})$ for small $t$. We start by
considering the parametrix kernel $Q_{h}(z,z',t)$ defined by
equation (\ref{eq:parametrixhchas}):
$$Q_{h}(z,w,t) = \varphi_{1}(z)K_{W,h}(z,w,t)\psi_{1}(w) + \varphi_{2}(z)K_{1,h}(z,w,t)\psi_{2}(w),$$
where the functions $\varphi_{i}$ and $\psi_{i}$, $i=1,2$, are
defined in Section \ref{subsection:expanstrace1}.
From Lemma \ref{lemma:paramhkcm3}, we can restrict our attention to $\int_{M} \psi(z)
(Q_{h}(z,z,t) -{\frac{1}{A_{h}}}) dA_{h}(z)$ and split the integral
as the sum of the following two terms:
\begin{align*}
L_{1}(t) &= \int_{M_{2}} \psi(z) \psi_{1}(z) (K_{W,h}(z,z,t)-{\frac{1}{A_{h}}}) dA_{h}(z)\\
L_{2}(t) &= \int_{Z_{\frac{5}{4}}} \psi(z) \psi_{2}(z)
(K_{1,h}(z,z,t)-{\frac{1}{A_{h}}})dA_{h}(z).
\end{align*}
Using the asymptotic expansion of the kernel
$K_{W,h}(z,z,t)$ we obtain:
\begin{multline}
L_{1}(t) = \int_{M_{2}} \psi(z) \psi_{1}(z) \left({\frac{1}{4\pi t}} +
{\frac{R_{h}(z)}{12\pi}} - {\frac{1}{A_{h}}} + \cR_{1}(z,t)\right)
dA_{h}(z). \label{eq:tpshkWmpkM2}\end{multline}

For $L_{2}(t)$, we use the same construction and notation as in the proof of
Proposition \ref{prop:casmpexptraoA}. Now, let $a>5/4$ and let us
split the integral $L_{2}(t)$ as the sum $L_{2} = \wt{J}_{1}(t) +
\wt{J}_{2}(t) + \wt{J}_{3}(t)$, where the $\wt{J}_{i}$, $i=1,2,3$, are given by:
\begin{align*}
\wt{J}_{1}(t) &= \int_{{\frac{5}{4}}}^{\infty}\int_{0}^{1}
\hat{\psi}(\wt z)
\hat{\psi}_{2}(\wt z)(k_{h}(\wt z,\wt z,t)-{\frac{1}{A_{h}}})dA_{\hat{h}}(\wt z),\\
\wt{J}_{2}(t) &= \int_{{\frac{5}{4}}}^{a}\int_{0}^{1} \hat{\psi}(\wt
z) \hat{\psi}_{2}(\wt z)
\sum_{m\neq 0} k_{h}(\wt{z},\wt{z}+m,t) dA_{\hat{h}}(\wt z),\\
\wt{J}_{3}(t) &= \int_{a}^{\infty}\int_{0}^{1} \hat{\psi}(\wt z)
\hat{\psi}_{2}(\wt z) \sum_{m\neq 0} k_{h}(\wt{z},\wt{z}+m,t)
dA_{\hat{h}}(\wt z).
\end{align*}
For $\wt{J}_{1}$ we use the local asymptotic expansion of the heat kernel
$k_{h}(\wt z,\wt z,t)$, whose remainder term is uniformly bounded, see \cite{ChGT}:
\begin{equation}
\wt{J}_{1}(t) = \int_{{\frac{5}{4}}}^{\infty}\int_{0}^{1}
\hat{\psi}(\wt z) \hat{\psi}_{2}(\wt z) \left({\frac{1}{4\pi t}} +
{\frac{R_{\hat{h}}(\wt z)}{12\pi}} - {\frac{1}{A_{h}}} + \cR_{1,1}(\wt
z,t)\right) dA_{\hat{h}}(\wt z) \label{eq:tpshklecmpkJ1}
\end{equation}
For $\wt{J}_{2}(t)$, in the same way as in the proof of
Proposition \ref{prop:casmpexptraoA}, we can estimate the series
as in equation (\ref{eq:bshkmnzwra1}).
Then we estimate the integral in the same way as in equations (\ref{eq:bbaauxintfs3}) and (\ref{eq:auxbisa}):
\begin{multline}
\wt{J}_{2}(t) \ll \int_{{\frac{5}{4}}}^{a} y^{-11} e^{ -{\frac
{c_{2}}{a^{4} t}}} \sum_{m\neq 0}e^{-{\frac{c_{1} \log(1 +
{\frac{m^{2}}{2a^{2}}})^{2}}{2t}}} {\frac{dy}{y^{2}}}\\ \ll
e^{-{\frac {c_{2}}{a^{4} t}}} \int_{{\frac{5}{4}}}^{a} y^{-11}
\int_{1}^{\infty} e^{-{\frac{c_{1} \log(1 +
{\frac{u^{2}}{2a^{2}}})^{2}}{2t}}} du \ {\frac{dy}{y^{2}}} \ll a
e^{-{\frac {c_{2}}{a^{4} t}}}. \label{eq:L2J2}
\end{multline}
The integral $\wt{J}_{3}$ can be bounded as:
\begin{multline}\wt{J}_{3}(t) \leq \int_{Z_{a}} \psi(z) \psi_{2}(z) K_{1,h}(z,z,t) dA_{h}(z) \\
\ll t^{-1} \int_{a}^{\infty} y^{-12} dy \ll t^{-1} a^{-11}.
\label{eq:L2J3}\end{multline}

Taking $a=t^{-1/5}$ in the same way as we did in the proof of Theorem \ref{theorem:asympexpanrht} and putting equations
(\ref{eq:tpshkWmpkM2}) (\ref{eq:tpshklecmpkJ1}) (\ref{eq:L2J2}) and
(\ref{eq:L2J3}) together we obtain:
\begin{multline*}
\tr(\psi(e^{-t\Delta_{h}}-P_{\ker(\Delta_{h})}))\\
= \int_{M_{2}} \psi(z) \psi_{1}(z) \left({\frac{1}{4\pi t}} +
{\frac{R_{h}(z)}{12\pi}} - {\frac{1}{A_{h}}} + \cR_{1}(z,t)\right)
dA_{h}(z)\\ + \int_{{\frac{5}{4}}}^{\infty}\int_{0}^{1}
\hat{\psi}(\wt z) \hat{\psi}_{2}(\wt z) \left({\frac{1}{4\pi t}} +
{\frac{R_{\hat{h}}(\wt z)}{12\pi}} - {\frac{1}{A_{h}}} + \cR_{1,1}(\wt
z,t)\right) dA_{\hat{h}}(\wt z) +O(t),
\end{multline*}
where $O(t)$ is clearly independent of $z$. Now we know that
$\vert \cR_{1}(z,t)\vert \ll t$ and $\vert \cR_{1,1}(\wt z,t) \vert \ll
t$ uniformly in $z$. Therefore we can make the following estimate:
$$\int_{M_{2}} \psi(z) \psi_{1}(z) \cR_{1}(z,t) dA_{h}(z) + \int_{{\frac{5}{4}}}^{\infty}\int_{0}^{1}
\hat{\psi}(\wt z) \hat{\psi}_{2}(\wt z) \cR_{1,1}(\wt z,t)
dA_{\hat{h}}(\wt z) \ll t.$$
This finishes the proof of Lemma \ref{lemma:etphohmptgt0}.
\end{proof}

The rest of the proof now follows the same lines as in \cite{OPS1}. Let us mention the main steps of it.
Going back to the variation of the relative zeta function, we may now
apply integration by parts in equation (\ref{eq:vzfgms1}) to obtain
for $\Re(s)>0$:

$${\frac{\delta\zeta}{\delta\psi}}(s; \Delta_{h}, \Delta_{1,0}) = {\frac{2s}{\Gamma(s)}}
\int_{0}^{\infty} t^{s-1} \tr(\psi (e^{-t \Delta_{h}} -
P_{\ker(\Delta_{h})})) dt.$$

We split this integral as:
\begin{multline} {\frac{\delta\zeta}{\delta\psi}}(s; \Delta_{h},
\Delta_{1,0}) = {\frac{2s}{\Gamma(s)}}\left( \int_{0}^{1}
 t^{s-1} \tr(\psi (e^{-t \Delta_{h}} -
P_{\ker(\Delta_{h})})) dt\right. \\ \left.
 +\int_{1}^{\infty}  t^{s-1} \tr(\psi (e^{-t \Delta_{h}} -
P_{\ker(\Delta_{h})})) dt \right).\label{eq:auxsepzetaf}
\end{multline}

From Lemma \ref{lemma:daitmrht}, the integral in second term on the right-hand side of equation
(\ref{eq:auxsepzetaf}) is an entire function of $s$. Since $\Gamma(s)^{-1}\sim s$,
it follows that:
$$\frac{d}{ds} \ {\frac{2s}{\Gamma(s)}}\int_{1}^{\infty} t^{s-1} \tr(\psi (e^{-t \Delta_{h}} -
P_{\ker(\Delta_{h})})) dt \  \Big|_{s=0} = 0$$

Using Lemma \ref{lemma:etphohmptgt0}, the first term on the right-hand side of (\ref{eq:auxsepzetaf}),
becomes:
\begin{multline*}
{\frac{2s}{\Gamma(s)}} \int_{0}^{1}
 t^{s-1} \tr(\psi (e^{-t \Delta_{h}} -
P_{\ker(\Delta_{h})})) dt\\
= {\frac{2s}{\Gamma(s)}} \left\{{\frac{1}{s}} \int_{M} \psi(z)
({\frac{R_{h}(z)}{12\pi}} - {\frac{1}{A_{h}}})dA_{h} + \mbox{
analytic in $s$ near $0$} \right\}.
\end{multline*}
The next step is to take the derivative with respect to $s$ at $s=0$. Using ${\frac{1}{\Gamma(s)}} = s + O(s^{2})$, we have:
\begin{multline*}
{\frac{d}{ds}} {\frac{2s}{\Gamma(s)}} \int_{0}^{1}
 t^{s-1} \tr(\psi (e^{-t \Delta_{h}} -
P_{\ker(\Delta_{h})})) dt \Big|_{s=0} \\
= \int_{M} 2\psi(z) \left({\frac{R_{h}(z)}{12 \pi}} -
{\frac{1}{A_{h}}}\right)dA_{h}.
\end{multline*}
Thus,
\begin{multline}{\frac{\delta}{\delta\psi}} \log \det(\Delta_{h},\Delta_{1,0}) =
-{\frac{\delta}{\delta\psi}} \frac{d}{ds}\zeta(s;\Delta_{h},
{\Delta}_{0})\big|_{s=0}\\
= {-\frac{1}{6\pi}} \int_{M} \psi (\Delta_{g}\varphi + R_{g})\
dA_{g} + {\frac{\delta}{\delta\psi}} \log A_{h}.\label{eq:vlrd}
\end{multline}
Finally, it is very easy to show that any $\psi$ in the domain of $F$ satisfies:
\begin{eqnarray*}{\frac{1}{2}}\left.{\frac{\partial}{\partial u}} \int_{M} \vert \nabla_{g}(\varphi + u \psi) \vert^{2} \
dA_{g} \right\vert _{u=0}
&=&
\langle \psi, \Delta_{g}\varphi\rangle,\\
\left.{\frac{\partial}{\partial u}} \int_{M}  R_{g}\ (\varphi+u\psi)
\ dA_{g}\right\vert _{u=0} &=& \int_{M}  R_{g}\ \psi \ dA_{g},
\end{eqnarray*}
Integrating (\ref{eq:vlrd}) we obtain:
\begin{equation*} \log \det(\Delta_{h},\Delta_{1,0})
= {-\frac{1}{12\pi}} \int_{M} \vert \nabla_{g}\varphi\vert^{2} \
dA_{g} - {\frac{1}{6\pi}} \int_{M}  R_{g}\ \varphi \ dA_{g} +  \log
A_{h} + C.
\end{equation*}
Notice that if $\varphi = 0$, $\Delta_{h} = \Delta_{g}$. Therefore
the last equation implies $C = \log \det(\Delta_{g},\Delta_{1,0})$. In this way, we have proved
Polyakov's formula:
\begin{theorem} Let $(M,g)$ be a surface with cusps and let
$h=e^{2\varphi}g$ be a conformal transformation of $g$ with $\varphi
\in {\mathcal F}_{11}$. For the corresponding relative determinants
we have the following formula:
\begin{multline} \log \det(\Delta_{h},\Delta_{1,0})
= {-\frac{1}{12\pi}} \int_{M} \vert \nabla_{g}\varphi\vert^{2} \
dA_{g} - {\frac{1}{6\pi}} \int_{M}  R_{g}\ \varphi \ dA_{g}\\ +  \log
A_{h} + \log \det(\Delta_{g},\Delta_{1,0}).
\end{multline} \label{theorem:Polyakovf}
\end{theorem}

\subsubsection{Extremal properties of the relative determinant}

Given Pol\-ya\-kov's formula for the relative determinant,
the study of the extremal pro\-per\-ties of it is exactly the same as in OPS \cite{OPS1}
for the case when $\chi(M)<0$. We assume now that $\chi(M)<0$. Let us recall the analysis in
\cite{OPS1} as we adapt it to our case.
On ${\mathcal F}_{11}$ consider the following functional:
\begin{equation}
\Phi(\varphi) = {\frac{1}{2}} \int_{M} \vert
\nabla_{g}\varphi\vert^{2} \ dA_{g} +
 \int_{M}  R_{g}\ \varphi \  dA_{g} - \pi \chi(M)  \log \left(\int_{M} e^{2\varphi}
dA_{g}\right).
\end{equation}
It is straightforward that  $\Phi$ is translation invariant and that minimizing $\Phi$ is the same
as maximizing $\log \det(\Delta_{h},\Delta_{1,0})$ for metrics of constant area.
Since we are considering $\chi(M)< 0$, we have that $\Phi$ is convex.
In the same way as in \cite{OPS1}, we have that
$$\Phi(\varphi) = -6\pi \log \det(\Delta_{h},\Delta_{1,0}) + \pi(6-\chi(M))\log(A_{h}).$$

Let us drop the constraint $A_{h}=1$. Then, if $\varphi$ is a minimizer of
$\Phi$ the equation ${\frac{\delta \Phi}{\delta
\psi}} (\varphi) = 0$ holds for all $\psi \in {\mathcal F}_{11}$.
This implies that:
$$R_{h} =
e^{-2\varphi} (\Delta_{g}\varphi + R_{g}) = {\frac{2\pi
\chi(M)}{\int_{M} e^{2\varphi}
 dA_{g}}},$$
i.e. $R_{h}$ should be constant. If $A_{h}=2\pi(2p+m-2)$, it follows that $R_{h}
= -1$, where $p$ is the genus of $M$ and $m$ is the number of
cusps.

On the other hand if $R_{h} = \text{constant}$ we have that:
\begin{eqnarray*}
{\frac{\delta \Phi}{\delta \psi}}(\varphi) &=& \int_{M} e^{2\varphi}
\psi R_{h} dA_{g} - {\frac{\pi \chi(M)}{A_{h}}}
\int_{M} 2 \psi e^{2\varphi} dA_{g}\\
&=& \int_{M} {\frac{e^{2\varphi} \psi}{A_{h}}} (R_{h}A_{h} - 2\pi
\chi(M))  dA_{g} = 0,\end{eqnarray*} because of Gauss-Bonnet theorem.
Thus, the critical points of $\Phi$ are the metrics of constant curvature.
The convexity of $\Phi$ assures that the critical points are minima.

Our problem is to find a maximizer of the relative determinant among metrics inside the
following conformal class:
$${\conf}_{1,11}(g)=\{h\vert h=e^{2\psi}g, \text{ with } \psi \in {\mathcal F}_{11} \text{ and }
A_{h}=2\pi(2p+m-2)\}.$$

If the initial metric $g$ on $M$ is a metric of
negative constant curvature $g=\tau$ with $R_{\tau}=-1$, and we take the conformal
class ${\conf}_{1,11}(\tau)$, $\tau$ itself is the maximizer of the relative determinant
and $\tau\in{\conf}_{1,11}(\tau)$. The maximizer trivially exists inside the
conformal class. However, if the starting metric $g$ on
$M$ is a metric that is hyperbolic only in the cusps, the differential
equation for the curvature on the cusps is:
$$-e^{2\varphi} = \Delta_{g}\varphi -1.$$
This implies that in the cusps the function $\varphi$ should decay at infinity as
$y^{-1}$. In this case the function $\varphi$ is outside the
conformal class under consideration. Therefore in order to have a
maximizer of the relative determinant inside the conformal class we
need to be able to define the relative determinant for Laplacians whose
metrics have conformal factors $e^{2\varphi}$ with $\varphi$ having
a decay as $y^{-1}$ at infinity.

As it was mentioned in the introduction, in \cite{AAR} P. Albin, F. Rochon and the author consider renormalized determinants on
Laplace operator on more general surfaces that also include swac. In that case the authors use Vai\-llant's results
in \cite{Vaillant} to have an asymptotic
expansion of the renormalized trace. The conditions on the conformal factor imposed by Vaillant are di\-ffe\-rent to ours, but
conformal factors that decay as $y^{-1}$ at infinity are included. Then Ricci flow is used to prove existence of the maximizer.

We could use the fact that if an operator is trace class, its trace coincides with its renormalized trace. Thus we
could use Vaillant's result to
define our relative determinant in terms of the renormalized determinant of the Laplacian and of the one of our model operator.
However, in the Ricci flow proof in \cite{AAR} two different rescalings take place. When we consider relative determinants,
re-scaling implies to modify the model operator as well. This is an interesting open problem.

\appendix

\section{}
\label{appendix:hkeacm}
In this appendix we give the proof of Lemma \ref{lemma:hkeacm}. We prove the estimate of $K_{1,h}$.
The estimate of $K_h$ then follows by a standard gluing parametrix construction. We use the notation introduced in Section
\ref{subsection:expanstrace1} and Proposition \ref{prop:casmpexptraoA}. Let us recall equation (\ref{eq:hkhitklh}):
$$K_{1,h}(z,w,t)=\sum_{m\in \Z}
k_{h}(\wt z, \wt w+m,t),$$
where $\pi(\wt z)= z$, $\pi(\wt w)= w$, and $\wt{z}=(x_1,y_1)$ and $\wt{w}=(x_2,y_2)$ can be chosen so that $0\leq x_i \leq 1$.

We know that $d_{h}(z,w) = \inf_{m\in \Z} d_{\hat{h}}(\wt{z},\wt{w}+m)\leq d_{\hat{h}}(\wt{z},\wt{w}+m)$ for all $m\in \Z$. Then using the estimate in equation (\ref{eq:ehklhH2}) with constant $c_1>0$ corresponding to the metric $h$, we obtain:
\begin{multline*}
K_{1,h}(z,w,t) \ll  t^{-1} \sum_{m\in \Z} \exp{\left(
-{\frac{c_{1}d_{\hat h}^{2}(\wt z, \wt w +m)}{t}}\right)}\\ \leq t^{-1}
e^{-{\frac{c_{1}d_{h}^{2}(z,w)}{2t}}}
\sum_{m\in \Z} e^{-{\frac{c_{1}d_{\hat h}^{2}(\wt z, \wt w +m)}{2t}}}\\ \leq
t^{-1} e^{-{\frac{c_{1}d_{h}^{2}(z,w)}{2t}}} \left( e^{-{\frac{c_{2}d_{\hat g}^{2}(\wt z, \wt w)}{2t}}} +
\sum_{m\neq 0} e^{-{\frac{c_{2}d_{\hat g}^{2}(\wt z, \wt w +m)}{2t}}} \right)
\end{multline*}
Now we use the formula for the hyperbolic distance to estimate it; for $m\neq 0$ we have:
\begin{eqnarray*}
d_{\hat g}((x_1, y_1), (x_2 + m, y_2)) = \cosh^{-1} \left(1 + {\frac{(x_1 - x_2 - m)^{2}+ (y_1-y_2)^2}{2 y_1 y_2}}\right)\\
\geq  \log \left(1 + {\frac{(x_1 - x_2 - m)^{2}}{2y_1 y_2}}\right) \geq \log \left(1 + {\frac{(|m|-1)^2}{2y_1 y_2}}\right)
\end{eqnarray*}
since $-1 \leq x_1-x_2 \leq 1$ and 
$(|m|-1)^2 \leq (x_1-x_2 - m)^2 \leq (|m|+1)^2$, if $|m|\neq 0$. We proceed now to estimate the series in the same way
as in (\ref{eq:bbaauxintfs3}), but we do not need to restrict the values of $y_1$ and $y_2$ to $[1,a]$ any more.
We keep the value $y_1 y_2$ in the estimates instead of using the bounded $a^2$
\begin{multline*}
\sum_{|m|\geq 2} e^{-{\frac{c_{2}d_{\hat g}^{2}(\wt z, \wt w + m)}{t}}}
\leq \sum_{|m|\geq 1} e^{-{\frac{c_{2} \log(1 + {\frac{m^{2}}{2y_1 y_2}})^{2}}{2t}}}\\
\ll \int_{1}^{\infty} e^{-{\frac{c_{2} \log(1 + {\frac{u^{2}}{2y_1 y_2}})^{2}}{2t}}} du \\
\ll y_1^{1/2} y_2^{1/2} (1+\sqrt{t}e^{ct}) \leq C(\tau)y_1^{1/2} y_2^{1/2}
\end{multline*} for some constant $C(\tau)$ that depends on $\tau$, $0<t\leq \tau$. Putting all the terms together we obtain:
\begin{multline*}
K_{1,h}(z,w,t) \ll
t^{-1} e^{-{\frac{c_{1}d_{h}^{2}(z,w)}{2t}}} \left( 2 + e^{-{\frac{c_{2}d_{\hat g}^{2}(\wt z, \wt w)}{2t}}} +
\sum_{m\neq 0} e^{-{\frac{c_{1} \log(1 + {\frac{m^{2}}{2y_1 y_2}})^{2}}{2t}}} \right)\\
\ll t^{-1} y_1^{1/2} y_2^{1/2} e^{-{\frac{c_{1}d_{h}^{2}(z,w)}{2t}}}.
\end{multline*}

For the derivatives of the heat kernel we apply the results by S. Y. Cheng, P. Li and S. T. Yau in \cite{ChLY},
Theorems $6$ and $7$, to $(\H,\hat{h})$ that has bounded geometry. The fist two derivatives of the heat kernel
$K_{1,h}$ can be estimated in the same way as we did for the heat kernel. As the authors point out in \cite{ChLY},
the constant in each estimate will depend on the curvature of $M$ and its covariant derivatives.

\section{}

\subsection{Observation}
In the proof of Theorem \ref{theorem:trchocdacm}, we repeatedly make
use of the following elementary facts:
\begin{enumerate}
\item For any $a>0$, and $b,n ,m \in \R$, we have that:
$$\int_{n}^{m}e^{-ax^{2}-bx}dx = {\frac{e^{b^{2}/4a}}{\sqrt{a}}}
\int_{\sqrt{a}(n+{\frac{b}{2a}})}^{\sqrt{a}(m+{\frac{b}{2a}})}
e^{-v^{2}}dv \leq {\frac{\sqrt{\pi} e^{b^{2}/4a}}{\sqrt{a}}}.$$
\item For any $c>0$, $0<t\leq T$, $k,\ell \geq 0$ with $k+\ell>2$ we
have:
\begin{equation}
\int_{1}^{\infty} \int_{1}^{\infty} y^{-k} y'^{-\ell}
e^{-\frac{c}{t}\log(y/y')^{2}} dy dy'\leq \sqrt{t}e^{(1-k)^{2}t/c}.
\label{eq:integral1}
\end{equation}
\item
Let $\varphi \in C^{\infty}(M)$, $\psi=e^{-2\varphi}-1$ and
$\wt{\psi}=e^{2\varphi}-1$. If $\varphi\vert_{Z}(y,x)$,
$\Delta_{g}\varphi\vert_{Z}(y,x)$ and $\vert
\nabla_{g}\varphi\vert_{g}\vert_{Z} (y,x)$ are $O(y^{-k})$ as $y\to
\infty$, then so are $\psi\vert_{Z}(y,x)$,
$\Delta_{g}\psi\vert_{Z}(y,x)$, $\vert
\nabla_{g}\psi\vert_{g}\vert_{Z} (y,x)$ and the analogues functions
corresponding to $\wt{\psi}$.
\item For $a,b,c>0$, the function $f(t) = t^{-a} e^{-c t^{-b}}$
is bounded on $(0,\infty)$ and $\lim_{t\to 0} f(t)=0$.
\end{enumerate}
\label{ss:convint}

\subsection{Proof of the bounds of the integrals $J_1$, $J_2$ and $J_3$ in
Proposition \ref{prop:b1nauxoBaccf}}

Let us start with $J_1$ that is given by equation (\ref{eq:J1comp}):
\begin{multline*}
J_{1} = \int_{0}^{t}\int_{Z_{a}}\int_{[1,\frac{4a}{5}]\times S^{1}}
\psi_{2}(z) (K_{1,h}(z,z',s) +
p_{h,D}(z,z',s))e^{2\varphi(z')}\\
\psi(z') \Delta_{Z,g} (K_{1,g}(z',z,t-s) + p_{1,D}(z',z,t-s)) \
dA_{g}(z')\ dA_{g}(z)\ ds.\end{multline*}

Note that on this region $a\leq y<\infty$ and $1\leq y' \leq
{\frac{4a}{5}}$,
$\log(y/y')$ is bounded away from zero. Using the estimates of the heat kernels and their
derivatives we obtain:
\begin{multline*}
\vert J_{1}\vert \ll \int_{0}^{t}\int_{a}^{\infty}
\int_{1}^{\frac{4a}{5}} s^{-1} (t-s)^{-2} y
(e^{-{\frac{c\log(y/y')^{2}}{s}}}+ e^{-{\frac{c\log(y)^{2}}{s}}}
e^{-{\frac{c\log(y')^{2}}{s}}})\\  y'^{-k+1}
(e^{-{\frac{c\log(y/y')^{2}}{t-s}}} +
e^{-{\frac{c\log(y)^{2}}{t-s}}} e^{-{\frac{c\log(y')^{2}}{t-s}}})
{\frac{dy'}{y'^{2}}}
{\frac{dy}{y^{2}}} ds\\
\ll a t^{-2} \int_{0}^{t/2}\int_{a}^{\infty} s^{-1} y^{-1}
(e^{-{\frac{c\log(5y/4a)^{2}}{s}}}+ e^{-{\frac{c\log(y)^{2}}{s}}})
dy ds\\
 + a t^{-1}  \int_{t/2}^{t} \int_{a}^{\infty}
(t-s)^{-2} y^{-1} (e^{-{\frac{c\log(5y/4a)^{2}}{t-s}}} +
e^{-{\frac{c\log(y)^{2}}{t-s}}}) dy ds. \end{multline*}

Since $y\geq a > {\frac{5}{4}}$ we have an estimate in $s$:
$$e^{-{\frac{c\log(5y/4a)^{2}}{s}}} + e^{-{\frac{c\log(y)^{2}}{s}}}
\leq e^{-{\frac{c\log(5/4)^{2}}{2s}}}
(e^{-{\frac{c\log(5y/4a)^{2}}{2s}}} +
e^{-{\frac{c\log(y)^{2}}{2s}}})$$ and
$\int_{a}^{\infty}y^{-1}e^{-{\frac{c\log(5y/4a)^{2}}{2s}}}dy =
\int_{\frac{5}{4}}^{\infty}v^{-1}e^{-{\frac{c\log(v)^{2}}{2s}}}dv\ll
\sqrt{s}$. We get a similar estimate for $t-s$, and together these
give:
\begin{align*}
\vert J_{1}\vert &\ll a t^{-2} \int_{0}^{t/2} s^{-1}
e^{-{\frac{c\log(5/4)^{2}}{2s}}} \int_{\frac{5}{4}}^{\infty}  y^{-1}
e^{-{\frac{c\log(y)^{2}}{2 s}}} dy ds \\  & \quad \quad +  a t^{-1}
\int_{t/2}^{t} (t-s)^{-2} e^{-{\frac{c\log(5/4)^{2}}{2(t-s)}}}
\int_{\frac{5}{4}}^{\infty}  y^{-1}
e^{-{\frac{c\log(y)^{2}}{2(t-s)}}} dy ds\\ &\ll a t^{-2}
\int_{0}^{t/2} s^{-1/2} e^{-{\frac{c\log(5/4)^{2}}{2s}}} ds +  a
t^{-1} \int_{t/2}^{t} (t-s)^{-3/2}
e^{-{\frac{c\log(5/4)^{2}}{2(t-s)}}} ds\\ &\ll a t^{-2}
e^{-{\frac{c\log(5/4)^{2}}{4t}}} \int_{0}^{t/2} ds + a t^{-1}
e^{-{\frac{c\log(5/4)^{2}}{2t}}} \int_{t/2}^{t} ds \ll a
(t^{-1}+1)e^{c_{1}/t} \ll a e^{-{\frac{c'}{t}}},
\end{align*}
for some constants $c_{1}, c' >0$, where we also used part $(4)$ of Observation \ref{ss:convint}.

For $J_{2}$, we had reduced the problem to the following estimate:
\begin{align*}
\vert J_{2}\vert  \ll \int_{t/2}^{t} \Vert
M_{\chi_{Z_{\frac{4a}{5}}}} M_{\psi}
\Delta_{Z,g}e^{-s/2\Delta_{Z,g}}M_{\phi}^{-1}\Vert_{2} \Vert
M_{\phi}e^{-s/2\Delta_{Z,g}}\Vert_{2} ds.
\end{align*}
Now we proceed to estimate each of the HS norms appearing as integrand on the right-hand side as follows:
\begin{align*}
&\Vert M_{\chi_{Z_{\frac{4a}{5}}}} M_{\psi}
\Delta_{Z,g}e^{-s/2\Delta_{Z,g}}M_{\phi}^{-1}\Vert_{2}^{2}\\
& \quad \quad
= \int_{Z_{\frac{4a}{5}}}\int_{Z} \vert \psi(z) \Delta_{Z,g}
K_{Z,g}(z,z',s/2)\phi(z')^{-1}\vert^{2} dA_{g}(z')dA_{g}(z)\\ & \quad \quad \ll
\int_{\frac{4a}{5}}^{\infty} \int_{1}^{\infty} y^{-2k} y y' s^{-4}
(e^{-\frac{4c}{s} (\log(y/y'))^{2}} + e^{-\frac{4c}{s}
(\log(yy'))^{2}} )y' {\frac{dy'}{y'^{2}}} {\frac{dy}{y^{2}}}\\ & \quad \quad=
s^{-4} \int_{\frac{4a}{5}}^{\infty} \int_{1}^{\infty} y^{-2k-1}
e^{-\frac{4c}{s} (\log(y'/y))^{2}}\ dy' dy\\ & \quad \quad \quad \quad \quad \quad
+ s^{-4}
\int_{\frac{4a}{5}}^{\infty} \int_{1}^{\infty} y^{-2k-1}
e^{-\frac{4c}{s} (\log(y'))^{2}}\ dy' dy.
\end{align*}

The first integral in the last line above can be estimated by fixing $y$ and
making the change of variables $v = \log(y'/y)$, $y' = ye^{v}$,
$dy' = ye^{v}dv$:
\begin{multline*}  s^{-4}\int_{\frac{4a}{5}}^{\infty} \int_{-\log(y)}^{\infty}
y^{-2k} e^{v}e^{{\frac{-4c}{s}}v^{2}}\ dv \ dy\\ \ll s^{-4}
e^{{\frac{s}{4c}}} \sqrt{s} \int_{\frac{4a}{5}}^{\infty} y^{-2k}
\int_{-\infty}^{\infty} e^{-v^{2}}\ dv \ dy \ll s^{-7/2} a^{-2k+1}
e^{{\frac{s}{4c}}}. \end{multline*} As for the second integral, we
obtain in a similar way: $$s^{-4} \int_{\frac{4a}{5}}^{\infty}
\int_{1}^{\infty} y^{-2k-1} e^{-\frac{4c}{s} (\log(y'))^{2}}\ dy' \
dy \ll s^{-7/2} e^{{\frac{s}{4c}}} a^{-2k}.$$ Thus,
$$\Vert
M_{\chi_{Z_{\frac{4a}{5}}}} M_{\psi}
\Delta_{Z,g}e^{-s/2\Delta_{Z,g}}M_{\phi}^{-1}\Vert_{2} \ll
s^{-7/4}(a^{-k} + a^{-k+1/2}).$$

For the operator $M_{\phi}e^{-s/2\Delta_{Z,g}}$, using equation
(\ref{eq:mphkgcaux}) we have:
\begin{align*}
&\Vert M_{\phi}e^{-s/2\Delta_{Z,g}} \Vert_{2}^{2} \\
& \ll
\int_{1}^{\infty} \int_{1}^{\infty} s^{-2} y^{-1} y y'
(e^{-{\frac{2c}{s}}(\log(y/y'))^{2}} +
e^{-{\frac{2c}{s}}(\log(yy'))^{2}})^{2} {\frac{dy'}{y'^{2}}}
{\frac{dy}{y^{2}}}\\
& \ll \int_{1}^{\infty} \int_{1}^{\infty} s^{-2} y'^{-1} y^{-2}
(e^{-{\frac{4c}{s}}(\log(y/y'))^{2}} +
e^{-{\frac{4c}{s}}(\log(yy'))^{2}}) dy' dy\\ & \ll s^{-2}\sqrt{s}
e^{s/4c} + s^{-2} \int_{1}^{\infty} y'^{-1}
e^{-{\frac{4c}{s}}(\log(y'))^{2}} dy' \ll s^{-3/2} (1+e^{s/4c}).
\end{align*}
Since $s\leq t\leq 1$ we have that $\Vert M_{\phi}
e^{-s/2\Delta_{Z,g}} \Vert_{2} \ll s^{-3/4}$. It follows that:
$$\vert J_{2}\vert \ll \int_{t/2}^{t} s^{-7/4} (a^{-k} + a^{-k+1/2})\cdot
s^{-3/4} ds \ll a^{-k+1/2} t^{-3/2}.$$

Now, for $J_{3}$ we have:
\begin{multline*} J_{3}= \int_{t/2}^{t}\int_{Z_{a}}
\int_{Z_{\frac{4a}{5}}} \psi_{2}(z) K_{Z,h}(z,z',s)e^{2\varphi(z')}
\chi_{Z_{\frac{4a}{5}}}(z')\\
(\Delta_{Z,g}-\Delta_{Z,h})_{z'} K_{Z,g}(z',z,t-s)
dA_{g}(z')dA_{g}(z)ds.\end{multline*} Remember that
$\Delta_{Z,g}-\Delta_{Z,h} = (e^{2\varphi(z')}-1) \Delta_{Z,h}=
\wt{\psi}(z')\Delta_{Z,h}$, so the previous equation becomes:
\begin{multline*} J_{3}= \int_{t/2}^{t}\int_{Z_{a}}
\int_{Z_{\frac{4a}{5}}} \{\psi_{2}(z) K_{Z,h}(z,z',s)
\chi_{Z_{\frac{4a}{5}}}(z') \wt{\psi}(z') \\ (\Delta_{Z,h}
K_{Z,g}(z',z,t-s)) e^{-2\varphi(z)} \}
\ dA_{h}(z')\ dA_{h}(z)\ ds\\
= \int_{t/2}^{t}\int_{Z_{a}} \int_{Z_{\frac{4a}{5}}} \{\psi_{2}(z)
(\Delta_{Z,h} K_{Z,h}(z,z',s)\wt{\psi}(z'))
\chi_{Z_{\frac{4a}{5}}}(z')\\
 K_{Z,g}(z',z,t-s) e^{-2\varphi(z)}\}
\ dA_{h}(z')\ dA_{h}(z)\ ds\\
= \int_{t/2}^{t}\int_{Z_{a}} \int_{Z_{\frac{4a}{5}}} \{\psi_{2}(z)
e^{-2\varphi(z)}K_{Z,g}(z,z',t-s) \chi_{Z_{\frac{4a}{5}}}(z')\\
(\Delta_{Z,h} \wt{\psi}(z') K_{Z,h}(z',z,s))\}\ dA_{h}(z')\ dA_{h}(z)\ ds.
\end{multline*}
Writing this in terms of the corresponding operators we obtain:
\begin{align*}
J_{3}&=\int_{t/2}^{t}\tr(M_{\psi_{2}} M_{e^{-2\varphi}}
e^{-(t-s)\Delta_{Z,g}} M_{\chi_{Z_{\frac{4a}{5}}}} \Delta_{Z,h}
M_{\wt{\psi}} e^{-s\Delta_{Z,h}})
ds,\\
\vert J_{3}\vert & \leq \int_{t/2}^{t} \Vert
M_{\chi_{Z_{\frac{4a}{5}}}} \Delta_{Z,h} M_{\wt{\psi}}
e^{-s\Delta_{Z,h}} \Vert_{1}\ ds.
\end{align*}
We are now working in $L^{2}(M,dA_{h})$ therefore to simplify
notation we do not write the subindex $h$ in the trace and the HS
norms. In the same way as above we do:
\begin{multline*}
\Vert M_{\chi_{Z_{\frac{4a}{5}}}} \Delta_{Z,h} M_{\wt{\psi}}
e^{-s\Delta_{Z,h}} \Vert_{1}\\ \leq \Vert M_{\chi_{Z_{\frac{4a}{5}}}}
\Delta_{Z,h} M_{\wt{\psi}}
e^{-s\Delta_{Z,h}/2}M_{\phi^{-1}}\Vert_{2} \Vert
 M_{\phi} e^{-s\Delta_{Z,h}/2}\Vert_{2}
\end{multline*}
The kernel of the operator $M_{\chi_{Z_{\frac{4a}{5}}}} \Delta_{Z,h}
M_{\wt{\psi}} e^{-s\Delta_{Z,h}/2}M_{\phi^{-1}}$ is
\begin{align*}
\chi_{Z_{\frac{4a}{5}}}(z')(\Delta_{Z,h}(\wt{\psi}(z')
K_{Z,h}(z',z,s))\phi(z)^{-1}.\end{align*} Using the decay assumptions on
$\varphi$ and its derivatives, we have that:
\begin{multline*}
\vert \Delta_{Z,h}(\wt{\psi} K_{Z,h})\vert^{2} \ll \vert
\wt{\psi}\Delta_{Z,h}K_{Z,h}\vert^{2} + \vert K_{Z,h} \Delta_{Z,h}
\wt{\psi}\vert^{2} + 2\vert \langle \nabla \wt{\psi}, \nabla
K_{Z,h}\rangle \vert^{2}\\
\ll y'^{-2k+1}y (s^{-4} + s^{-2} +
s^{-3})(e^{-{\frac{c}{s}}(\log(y/y'))^{2}} +
e^{-{\frac{c}{s}}(\log(yy'))^{2}})^{2}.
\end{multline*}
Since for $0<s<1$ we have that $s^{-4} + s^{-2} + s^{-3}\ll s^{-4}$,
we can estimate the HS norm by:
\begin{align*}
&\Vert M_{\chi_{Z_{\frac{4a}{5}}}} \Delta_{Z,h} M_{\wt{\psi}}
e^{-s\Delta_{Z,h}/2}M_{\phi^{-1}} \Vert_{2}^{2}\\ & \quad \quad= \int_{Z} \int_{Z}
\vert \chi_{Z_{\frac{4a}{5}}}(z') \wt{\psi}(z') \Delta_{h,z'}
K_{h}(z',z,s/2) \phi(z)^{-1} \vert^{2}  dA_{h}(z') dA_{h}(z)\\ &
\quad \quad\ll s^{-4} \int_{1}^{\infty} \int_{\frac{4a}{5}}^{\infty}
\ y^{2}\ y'^{-2k+1} (e^{-{\frac{2c}{s}}(\log(y/y'))^{2}} +
e^{-{\frac{2c}{s}}(\log(yy'))^{2}})^{2}{\frac{dy'}{y'^{2}}}
{\frac{dy}{y^{2}}}\\ & \quad \quad \ll  s^{-4}
\int_{\frac{4a}{5}}^{\infty} \int_{1}^{\infty} (y'^{-2k-1}
e^{-{\frac{4c}{s}}(\log(y/y'))^{2}} + y'^{-2k-1}
e^{-{\frac{4c}{s}}(\log(y))^{2}}) \ dy \ dy'\\ & \quad \quad \ll
(a^{-2k+1}+a^{-2k}) s^{-7/2} e^{s/4c} \ll a^{-2k+1} s^{-7/2}.
\end{align*}
We finally obtain:
$$\Vert M_{\phi}^{-1} e^{-s/2 \Delta_{Z,h}}
 \wt{\psi}
\Delta_{h} \Vert_{2} \leq a^{-k+1/2} s^{-7/4}.$$

For the operator $e^{-s/2 \Delta_{Z,h}}M_{\phi}$, the proof goes in
the same way as for the operator $M_{\phi}e^{-s/2 \Delta_{Z,g}}$. At
the end we obtain:

\begin{align*}
\Vert e^{-s \Delta_{Z,h}}M_{\phi}\Vert_{2} &= \left( \int_{Z}
\int_{Z} \vert K_{Z,h}(z,z',s/2) \phi(z')\vert^{2} dA_{h}(z')
dA_{h}(z)\right)^{1/2} \ll s^{-3/4}.
\end{align*}

In this way: $$\vert J_{3}\vert \ll \int_{t/2}^{t} a^{-k+1/2}
s^{-7/4} s^{-3/4} ds \ll a^{-k+1/2} t^{-3/2}.$$

\begin{acknowledgements}
This paper expands part of my doctoral thesis. I thank my supervisor Werner M\"uller
for his guidance during and after the project. I am also grateful to Rafe Mazzeo,
Eugenie Hunsicker, and Sylvie Paycha for helpful discussions and their interest in this work.
I thank an anonymous referee for the suggestions and comments. Finally, I thank the Ma\-the\-ma\-tical Institute at the University of Bonn for hosting me during my graduate studies.
\end{acknowledgements}

\end{document}